\documentclass[11pt]{amsart}

\usepackage[T1]{fontenc}
\usepackage[utf8]{inputenc}
\usepackage{lmodern}
\usepackage[a4paper,margin=1in]{geometry}
\usepackage{amsmath,amssymb,amsthm,mathtools,mathrsfs}
\usepackage{array,booktabs,longtable}
\usepackage{enumitem}
\usepackage{microtype}
\usepackage{xcolor}
\usepackage[colorlinks=true,linkcolor=blue!45!black,citecolor=blue!45!black,urlcolor=blue!45!black]{hyperref}
\usepackage[nameinlink,noabbrev]{cleveref}

\newtheorem{theorem}{Theorem}[section]
\newtheorem{proposition}[theorem]{Proposition}
\newtheorem{lemma}[theorem]{Lemma}
\newtheorem{corollary}[theorem]{Corollary}
\theoremstyle{definition}
\newtheorem{definition}[theorem]{Definition}
\theoremstyle{remark}
\newtheorem{remark}[theorem]{Remark}

\newcommand{\C}{\mathbf C}

\newcommand{\Q}{\mathbf Q}
\newcommand{\F}{\mathbf F}
\newcommand{\cO}{\mathcal O}
\newcommand{\p}{\mathfrak p}
\newcommand{\Tr}{\operatorname{Tr}}
\newcommand{\Nm}{\operatorname{N}}
\newcommand{\Gal}{\operatorname{Gal}}
\newcommand{\ph}{\operatorname{ph}}

\newcommand{\Err}{\mathcal E}
\newcommand{\HerbPhi}{\varphi^{\mathrm H}}
\newcommand{\HerbPsi}{\psi^{\mathrm H}}

\newcommand{\ceil}[1]{\left\lceil #1\right\rceil}
\newcommand{\floor}[1]{\left\lfloor #1\right\rfloor}

\title[The First Main Lemma for Local Epsilon Factors]
{A Local Proof of Langlands's First Main Lemma\\
for Local Epsilon Factors over Nonarchimedean Local Fields}
\author{Fukuhiro Ueda}
\address{Research Institute for Mathematical Sciences, Kyoto University, Kyoto 606--8502, Japan}
\email{fueda@kurims.kyoto-u.ac.jp}
\date{August 2026}
\subjclass[2020]{11S37, 11S40, 11R42}
\keywords{local epsilon factors, Artin $L$-functions, local constants,
Gauss sums, ramification, Lamprecht formula, cyclic extensions,
equal characteristic}
\hypersetup{
  pdftitle={A Local Proof of Langlands's First Main Lemma for Local Epsilon Factors over Nonarchimedean Local Fields},
  pdfauthor={Fukuhiro Ueda},
  pdfsubject={A local proof of Langlands's First Main Lemma for local epsilon factors in mixed and equal characteristic},
  pdfkeywords={local epsilon factors, Artin L-functions, Gauss sums, ramification, First Main Lemma}
}

\begin{document}
\begin{abstract}
We give a complete local proof of Langlands's First Main Lemma for
local epsilon factors over nonarchimedean local fields.
The proof builds on Dwork's $p$-adic calculations and Langlands's
formulation of the local method. The present paper reformulates these
calculations in a common normalization and extends the argument to
equal characteristic. A more systematic treatment of stationary
parameters and conductor cases is intended to make the calculations
easier to follow. 

The author has also completed a local proof of Langlands's Second Main
Lemma~\cite{UedaSML}. Together with the present result and Langlands's
remaining reductions, this gives a complete proof of the local
construction of local epsilon factors envisioned by Langlands.

\end{abstract}
\maketitle

\clearpage
\tableofcontents
\clearpage

\section{Introduction}\label{sec:introduction}

In \cite{Langlands}, Langlands developed a local construction of epsilon factors for representations of the Weil group of a  nonarchimedean local field.  Brauer induction reduces their
construction to local constants of quasi-characters, but it also creates
an independence problem: distinct Brauer decompositions must give the
same result.    Langlands reduced this problem to four
local identities.  The present paper proves the identity he called the First Main Lemma.

For a nonarchimedean local field $E$, a (continuous) quasi-character
$\chi_E:E^\times\to\C^\times$, and a nontrivial (continuous) additive character
$\psi_E:E\to\C^\times$, Langlands associates a nonzero complex number
$\Delta_E(\chi_E,\psi_E)$, called the local constant; it is essentially a normalized Gauss sum.
Its definition is recalled in Section~\ref{sec:local-notation}.
He also defines the epsilon factor \cite[pp.~5--6]{Langlands}
\[
  \varepsilon(s,\chi_E,\psi_E)
  :=
  \Delta_E\!\left(
    |\cdot|_E^{\,s-\frac12}\chi_E,\psi_E
  \right),
\]
where $s\in\C$ and $|x|_E=q_E^{-v_E(x)}$, with $v_E$ the normalized
valuation and $q_E$ the cardinality of the residue field of $E$.
Writing $m=m_E(\chi_E)$ and $n=n_E(\psi_E)$ for the conductors, we have
\[
  \varepsilon(s,\chi_E,\psi_E)
  =
  q_E^{-(m+n)(s-\frac12)}
  \Delta_E(\chi_E,\psi_E).
\]
The product of these local epsilon factors over all places is the epsilon factor in Tate's functional
equation for Hecke $L$-functions \cite[p.~271]{Langlands}.

Let $K/F$ be a cyclic extension of  nonarchimedean local fields of prime degree $\ell$.  For a 
quasi-character $\chi_F:F^\times\to\C^\times$ and a nontrivial
 additive character $\psi_F:F\to\C^\times$, put
\[
 \chi_K=\chi_F\circ \Nm_{K/F},
 \qquad
 \psi_K=\psi_F\circ\Tr_{K/F}.
\]
Let $S(K/F)$ be the group of continuous characters of $F^\times$
that are trivial on $\Nm_{K/F}(K^\times)$. 

The following identity is quoted in \cite[p.~144]{Langlands} as Lemma~10.1.

\begin{theorem}[First Main Lemma]\label{thm:first-main}
For every cyclic extension $K/F$ of prime degree, every 
quasi-character $\chi_F$ of $F^\times$, and every nontrivial
additive character $\psi_F$ of $F$,
\begin{equation}\label{eq:first-main}
 \Delta_K(\chi_K,\psi_K)
 \prod_{\mu\in S(K/F)}\Delta_F(\mu,\psi_F)
 =
 \prod_{\mu\in S(K/F)}\Delta_F(\mu\chi_F,\psi_F).
\end{equation}
\end{theorem}

\subsection{Historical background}
Langlands attributes the First Main Lemma to Dwork
\cite[p.~6]{Langlands}.  Dwork's 1954 thesis \cite{Dwork} contains the
$p$-adic calculations for the cyclic-prime relation, including the
endpoint and non-stable wild cases
\cite[Chapter~III, pp.~72--95]{Dwork}.  These calculations are carried
out for finite extensions of $\Q_p$ and use a normalization different
from Langlands's.

The surviving chapters of Langlands's manuscript already prove the
unramified and tame cases and the stable odd-prime wild case of the
First Main Lemma.  Chapters~8 and~9 also develop the stationary
parameters and finite-field identities used in the remaining wild
cases. Chapter~10, entitled ``The First Main
Lemma'', is missing, and hence the manuscript does not contain the
final assembly of these cases into a proof of the First Main Lemma.
Langlands's editorial comments explain that Deligne's global argument
provided a different proof of the existence of the local epsilon factors, after
which Langlands did not complete the local manuscript \cite{Langlands}.


\subsection{Structure of the proof}\label{sec:first-main-completion}
We compare the two sides of \eqref{eq:first-main} using norm and
trace identities for $K/F$. Neither local class field theory nor a global functional equation is used.

First, Proposition~\ref{prop:finite-local-realization} converts
Langlands's Haar-integral definition into a Gauss sum over a finite
quotient of the unit group and proves this sum nonzero by Fourier
duality, independently of Lamprecht's formula.  This nonvanishing is
used in the inverse-pairing identity of Lemma~\ref{lem:inverse-pair}.

Second, Lemma~\ref{lem:stationary-class-existence} constructs the
stationary parameter attached to a character as a quotient class.
For conductor greater than one, Lamprecht's formula
(Theorem~\ref{thm:Lamprecht}) gives the factorization
\eqref{eq:Delta-factorization-new} into three factors: the
multiplicative-character value at an admissible element, the stationary
factor attached to a chosen representative, and a residue-field factor.
The last factor is $1$ for even conductor and the phase of the sum of
the Hasse function over the residue field for odd conductor.
Lemma~\ref{lem:complete-stationary-invariant} describes how changing the
representative changes the stationary factor and the Hasse function;
the product in the local constant is unchanged.

Third, Corollary~\ref{cor:stationary-class-under-norm} proves the
transport of stationary classes under norm.  Propositions~\ref{prop:high-parameter-table},
\ref{prop:wild-quadratic-high-parameters}, and~\ref{prop:low-parameter-table}
then give stationary-class representatives over $F$ and $K$, in the relevant conductor ranges.

Finally, Lemma~\ref{lem:twist-invariance} permits $\chi_F$ to be
chosen with minimal conductor in its $S(K/F)$-orbit.  Write $m$ for
its conductor.  For ramified $K/F$, let $t$ be the lower ramification
break and put $T=t+1$, the conductor of every nontrivial character in
$S(K/F)$.  The proof is divided into the following cases.
\begin{center}
\begin{tabular}{@{}>{\raggedright\arraybackslash}p{0.30\textwidth}>{\raggedright\arraybackslash}p{0.64\textwidth}@{}}
\toprule
Extension and conductor & Argument\\
\midrule
Unramified
& For $m=0$, use the unramified-character formula.  For $m=1$, use the
Hasse--Davenport lifting identity. For $m>1$, use Lamprecht's formula
and compare sums over the residue fields
(Proposition~\ref{prop:unramified-new}).\\
\addlinespace
Tamely ramified, $\ell$ odd
& For $m=0,1$, use the unramified-character and Hasse--Davenport product
formulas, respectively.  For $m>1$, multiplication by a character in
$S(K/F)$ preserves the stationary class of $\chi_F$; compare the
remaining Gauss sums (Proposition~\ref{prop:tame-odd-new}).\\
\addlinespace
Tamely ramified, $\ell=2$
& The cases $m=0,1$ use the same formulas as for tame odd degree.
For $m>1$, evaluate quadratic Gauss sums
(Proposition~\ref{prop:tame-two-new}).\\
\addlinespace
Wildly ramified, $m\le1$\newline
\emph{(Endpoint cases)}
& For $m=0$, use the unramified-character formula.  For $m=1$, Frobenius
identifies the Gauss sums for $\chi_F$ and $\chi_K$; norm identities
compare the remaining character values
(Section~\ref{sec:wild-endpoints}).\\
\addlinespace
Wildly ramified, $\ell$ odd,\newline
$m\ge2T$ \emph{(stable range)}
& Multiplication by a character in $S(K/F)$ preserves the stationary
class of $\chi_F$.  Lamprecht's formula reduces the comparison to
sums over the residue field (Proposition~\ref{prop:wild-stable-new}).\\
\addlinespace
Wildly ramified, $\ell$ odd,\newline
$1<m<2T$\newline
\emph{(non-stable range)}
& Multiplication by a character in $S(K/F)$ can change the stationary
class of $\chi_F$.  Compute the classes for $\chi_K$ and the twists of $\chi_F$, then cancel the
resulting character values and Gauss-sum factors
(Corollary~\ref{cor:wild-odd}).\\
\addlinespace
Wildly ramified, $\ell=2$,\newline
$m>1$
& Compare the cases $m<T$, $m=T$, and $m>T$.  At the odd boundary,
cancel the residual quotient with the critical-function values
(Proposition~\ref{prop:wild-quadratic-cancellation}).\\
\bottomrule
\end{tabular}
\end{center}

\begin{proof}[Proof of Theorem~\ref{thm:first-main}]
Let $F$ be a nonarchimedean local field and let $K/F$ be cyclic of
prime degree $\ell$.  If $K/F$ is unramified, apply
Proposition~\ref{prop:unramified-new}.  If it is ramified and tame, the odd-degree and quadratic cases are
Propositions~\ref{prop:tame-odd-new} and
\ref{prop:tame-two-new}.

Suppose that $K/F$ is wild.  Then $\ell=p$, the residue
characteristic.  By \eqref{eq:error-twist-invariance}, replace $\chi_F$ by a
member of minimal conductor in its $S(K/F)$-orbit.  Conductors $0$ and
$1$ are treated by Lemma~\ref{lem:ramified-conductor-zero} and
Proposition~\ref{prop:wild-conductor-one}, respectively.  For odd $p$,
the stable range is Proposition~\ref{prop:wild-stable-new}, and the
remaining non-stable range is Corollary~\ref{cor:wild-odd}.  For
$p=2$, we have Proposition \ref{prop:wild-quadratic-cancellation}.  These cases
exhaust all cyclic extensions of prime degree. 
\end{proof}

\subsection{Equal and mixed characteristic}
The unramified, tame,  endpoint conductors, and the stable
wild odd-degree cases are proved by the same arguments in mixed and
equal characteristic.  Separate characteristic-dependent
terms occur only in the remaining non-stable wild calculations.

In the odd-prime non-stable case, Proposition~\ref{prop:exact-odd-assembly}
reduces the remaining additive factor to a single norm--trace correction
$X$.  Proposition~\ref{prop:X-closed} gives a  formula for $X$
whose numerator contains the residue characteristic $p$, hence $X=0$ in characteristic $p$.
Proposition~\ref{prop:X-depth} shows that   its contribution to the additive character is trivial in mixed
characteristic, by bounding its valuation.

For a wild quadratic extension, Lemma~\ref{lem:quad-X} gives the
analogous norm--trace correction; its numerator contains the factor
$2$, so the correction vanishes in characteristic two.  In mixed
characteristic, Proposition~\ref{prop:quadratic-norm-trace-phase} gives
$X\in\p_F^T$ except at $m=T$ with $T$ odd, where $v_F(X)=T-1$.
Its additive phase is $1$ in all cases; at the odd boundary this follows
from an Artin--Schreier identity rather than the valuation bound.
Proposition~\ref{prop:wild-quadratic-cancellation} separately cancels the
residual quotient with the remaining critical-function values.

\subsection{Organization of the paper}
Section~\ref{sec:local-notation} establishes the two realizations of the
local constant, proves nonvanishing, introduces the error term, and
reduces the proof to characters of minimal conductor in their
$S(K/F)$-orbits.  The next four sections
develop the cyclic-prime ramification, stationary-phase, finite-field,
and stationary-parameter tools. Section~7 treats the unramified and
tame cases, and Section~8 treats the stable wild odd-prime range.
Section~\ref{sec:wild-endpoints} treats wild conductors zero and one;
Sections~\ref{sec:nonstable-odd} and~\ref{sec:wild-quadratic} treat the
non-stable wild odd-prime range and wild quadratic extensions,
respectively. These cases are combined in the proof of
Theorem~\ref{thm:first-main} in Subsection~\ref{sec:first-main-completion}.
 
\section{Local constants and the elementary reduction}
\label{sec:local-notation}

\subsection{Local fields, characters, and exact conductors}

Let $E$ be a nonarchimedean local field.  Write
\[
 \cO_E=\{x\in E:v_E(x)\ge0\},
 \qquad
 \p_E=\{x\in E:v_E(x)\ge1\},
\]
and let $k_E=\cO_E/\p_E$ have cardinality $q_E$.  We normalize
$v_E(E^\times)=\mathbf Z$.  Put
\[
 U_E^0=\cO_E^\times,
 \qquad
 U_E^r=1+\p_E^r\quad(r\ge1).
\]
All quasi-characters are continuous homomorphisms
$E^\times\to\C^\times$, and all additive characters are continuous
homomorphisms from the additive group of $E$ to $\C^\times$.

\begin{remark}[Characteristic convention]
Write $p=\operatorname{char}k_E$.  Thus either
$\operatorname{char}E=0$ with residue characteristic $p$, or
$\operatorname{char}E=p$.  An integer occurring in an identity inside
$E$ denotes its image under $\mathbf Z\to E$; in particular $p=0$ in
equal characteristic $p$.
\end{remark}

\begin{definition}[Multiplicative conductor]
The conductor exponent $m_E(\theta)$ is the least $m\ge0$ for which
$\theta$ is trivial on $U_E^m$.  Thus $m_E(\theta)=0$ precisely when
$\theta$ is unramified.
\end{definition}

\begin{definition}[Additive conductor]
For a nontrivial additive character $\psi$, the integer $n_E(\psi)$
is characterized by the condition that $\p_E^{-n_E(\psi)}$ is the
largest fractional ideal on which $\psi$ is trivial.
\end{definition}


\subsection{The Haar-integral and finite realizations of the local constant}

Define the total phase
\[
 \ph(z)=
 \begin{cases}
 z/|z|,&z\ne0,\\
 1,&z=0.
 \end{cases}
\]
Let $m=m_E(\theta)$ and $n=n_E(\psi)$, and choose
$\gamma\in E^\times$ such that
\begin{equation}\label{eq:admissible-gamma}
 \gamma\cO_E=\p_E^{m+n}.
\end{equation}
Such an element $\gamma$ is called \emph{admissible} for $(\theta,\psi)$.
For a positive Haar measure $du$ on $U_E^0$, Langlands defines
\begin{equation}\label{eq:def-delta}
 \Delta_E(\theta,\psi)
 =\theta(\gamma)\,
  \ph\left(\int_{U_E^0}
      \psi(u/\gamma)\theta(u)^{-1}\,du\right).
\end{equation}

Put $Q_m=U_E^0/U_E^m$.  For $\bar u\in Q_m$, represented by
$u\in U_E^0$, define
\begin{equation*}
 g_{\theta,\psi,\gamma}(\bar u)
 =\psi(u/\gamma)\theta(u)^{-1}
\end{equation*}
and
\begin{equation}\label{eq:finite-local-gauss-sum}
 G_E(\theta,\psi;\gamma)
 =\sum_{\bar u\in Q_m}
   g_{\theta,\psi,\gamma}(\bar u).
\end{equation}

\begin{proposition}[Finite realization and nonvanishing]
\label{prop:finite-local-realization}
The summand just defined is well defined.
There is a positive real number $c_m$, depending only on the chosen Haar
measure and $m$, such that
\begin{equation*}
 \int_{U_E^0}\psi(u/\gamma)\theta(u)^{-1}\,du
 =c_m G_E(\theta,\psi;\gamma).
\end{equation*}
Moreover
\[
 G_E(\theta,\psi;\gamma)\ne0.
\]
Consequently
\begin{equation}\label{eq:delta-finite-realization}
 \Delta_E(\theta,\psi)
 =\theta(\gamma)\ph\bigl(G_E(\theta,\psi;\gamma)\bigr).
\end{equation}
\end{proposition}

\begin{proof}
If $u/v\in U_E^m$, then $\theta(u)=\theta(v)$ and
$u-v\in\p_E^m$.  By \eqref{eq:admissible-gamma},
$(u-v)/\gamma\in\p_E^{-n}$, so
$\psi(u/\gamma)=\psi(v/\gamma)$.  Thus the summand descends to
$Q_m$.  The integrand is constant on every coset of $U_E^m$ in
$U_E^0$, and all such cosets have the same positive measure $c_m$;
this proves the stated integral formula.

It remains to prove nonvanishing.  If $m=0$, the quotient $Q_0$ is a
singleton and $G_E(\theta,\psi;\gamma)=1$.
Assume $m>0$ and put $A=\cO_E/\p_E^m$.  The pairing
\[
 A\times A\longrightarrow\C^\times,
 \qquad
 (\bar x,\bar y)\longmapsto\psi(xy/\gamma)
\]
is perfect.  Indeed, its annihilator in either variable is exactly
$\p_E^m$: this follows from \eqref{eq:admissible-gamma} and the
maximality of $\p_E^{-n}$ in the definition of $n_E(\psi)$.

Extend the function $u\mapsto\theta(u)^{-1}$ on $A^\times$ by zero
to a function $f:A\to\C$, and write
\[
 \widehat f(x)=
 \sum_{u\in A^\times}\theta(u)^{-1}\psi(xu/\gamma).
\]
Then $\widehat f(1)=G_E(\theta,\psi;\gamma)$.  If $x\in A^\times$,
the substitution $u\mapsto x^{-1}u$ gives
$\widehat f(x)=\theta(x)\widehat f(1)$.
If $x$ is not a unit, exactness of the conductor provides
$a\in U_E^{m-1}$ with $\theta(a)\ne1$.  Since
$x(a-1)u/\gamma\in\p_E^{-n}$ for every $u\in U_E^0$, multiplication
by $a$ permutes $A^\times$ and gives
$\widehat f(x)=\theta(a)^{-1}\widehat f(x)$,
whence $\widehat f(x)=0$.  If $\widehat f(1)$ were zero,
the formula for the Fourier coefficients at units and the preceding vanishing
would make every Fourier coefficient of $f$ zero.  Perfect finite
Fourier duality would then give $f=0$, contradicting $f(1)=1$.
Thus $G_E(\theta,\psi;\gamma)\ne0$.  Finally the positive
factor $c_m$ does not change phase, proving
\eqref{eq:delta-finite-realization}.
\end{proof}

\begin{corollary}[Nonvanishing of the local Gauss integral]
\label{cor:local-gauss-integral-nonzero}
For every quasi-character $\theta$ and every nontrivial additive
character $\psi$, the integral in \eqref{eq:def-delta} is nonzero.
\end{corollary}

\begin{lemma}[Unramified twists]\label{lem:unramified-twists}
Let $\chi$ be an unramified quasi-character of $E^\times$, let
$\theta$ be any quasi-character, let $\psi$ be a nontrivial additive
character of $E$, and let $\varpi_E$ be a uniformizer.
Then
\begin{equation}\label{eq:unramified-twist}
 \Delta_E(\theta\chi,\psi)
 =\chi(\varpi_E)^{m_E(\theta)+n_E(\psi)}\Delta_E(\theta,\psi).
\end{equation}
\end{lemma}
\begin{proof}
The quasi-character $\chi$ is trivial on units, so it changes neither
the conductor nor the finite Gauss sum. Its value at the admissible
element is the displayed multiplier, including when $m_E(\theta)=0$.
\end{proof}

\begin{lemma}[Independence of the denominator and Haar normalization]
\label{lem:gamma-independent}
The right side of \eqref{eq:def-delta} is independent of the choice of
$\gamma$ satisfying \eqref{eq:admissible-gamma} and of the positive
normalization of Haar measure.
\end{lemma}

\begin{proof}
Let $\gamma'=a\gamma$ with $a\in U_E^0$.  Multiplication by $a$
permutes $Q_m$, and \eqref{eq:finite-local-gauss-sum} gives
\[
 G_E(\theta,\psi;\gamma')
 =\theta(a)^{-1}G_E(\theta,\psi;\gamma).
\]
The restriction of $\theta$ to $U_E^0$ factors through the finite group
$Q_m$, so $|\theta(a)|=1$.  By nonvanishing,
\[
 \ph\bigl(G_E(\theta,\psi;\gamma')\bigr)
 =\theta(a)^{-1}
  \ph\bigl(G_E(\theta,\psi;\gamma)\bigr).
\]
Since $\theta(\gamma')=\theta(a)\theta(\gamma)$, the two factors
cancel.  The finite formula \eqref{eq:delta-finite-realization} also proves
independence of the positive normalization of Haar measure.
\end{proof}

\begin{lemma}[Normalization of the residual additive character]
\label{lem:residual-normalization}
Let $\psi_F$ have conductor exponent $n$, and fix a uniformizer
$\varpi_F$.  The formula
\begin{equation}\label{eq:residual-additive-character}
 \bar\psi(\bar x)
 =\psi_F(x\varpi_F^{-n-1}),
 \qquad x\in\cO_F\text{ a lift of }\bar x\in k_F,
\end{equation}
defines a nontrivial additive character of $k_F$, independently of the
choice of the lift $x$.  Define the canonical additive character
$\psi_{\rm can}:k_F\to\C^\times$ by
\begin{equation}\label{eq:finite-additive-parametrization}
 x\longmapsto
 \exp\left(\frac{2\pi i}{p}
   \Tr_{k_F/\F_p}(x)\right).
\end{equation}
Then there is a unit $a\in\cO_F^\times$ such that the residual
character attached to $\psi_F^a(x)=\psi_F(ax)$ is $\psi_{\rm can}$.
\end{lemma}

\begin{proof}
Well-definedness follows because two lifts differ by an element of
$\p_F$, and therefore their arguments differ by an element of
$\p_F^{-n}$, on which $\psi_F$ is trivial.  Nontriviality follows
from the maximality of $\p_F^{-n}$.  Since the trace pairing
$k_F\times k_F\to\F_p$ is nondegenerate, there is a unique
$c\in k_F^\times$ such that
$\bar\psi(x)=\psi_{\rm can}(cx)$.  Choose a unit $a\in\cO_F^\times$
with residue $c^{-1}$; then the residual character attached to
$\psi_F^a$ is $\psi_{\rm can}$.
\end{proof}

Let $k$ be a finite field, let $\varphi:k\to\C^\times$ be a
nontrivial additive character, and let
$\eta:k^\times\to\C^\times$ be a multiplicative character.  Define
\begin{equation}\label{eq:tau-definition}
 \tau_k(\eta,\varphi)
 =-\sum_{x\in k^\times}\eta(x)^{-1}\varphi(x).
\end{equation}
Thus $\tau_k(1,\varphi)=1$.

\begin{lemma}[Local constants of conductor zero and one]
\label{lem:low-conductor-local-constants}
Let $\theta$ be a quasi-character of $E^\times$, let $\psi$ be a
nontrivial additive character, and choose a uniformizer $\varpi_E$.
If $m_E(\theta)=0$, then
\begin{equation}\label{eq:unramified-character-local-constant}
 \Delta_E(\theta,\psi)=\theta(\varpi_E)^{n_E(\psi)}.
\end{equation}
If $m_E(\theta)=1$, define the residue characters by
\[
 \bar\theta(\bar u)=\theta(u),\qquad
 \bar\psi(\bar x)=\psi(x/\varpi_E^{n_E(\psi)+1}),
\]
where $u\in\cO_E^\times$ and $x\in\cO_E$ lift
$\bar u\in k_E^\times$ and $\bar x\in k_E$, respectively. Then
\begin{equation}\label{eq:conductor-one-local-gauss}
 \Delta_E(\theta,\psi)
 =-\theta(\varpi_E)^{n_E(\psi)+1}
   \ph\bigl(\tau_{k_E}(\bar\theta,\bar\psi)\bigr).
\end{equation}
\end{lemma}
\begin{proof}
For conductor zero, take the admissible element
$\gamma=\varpi_E^{n_E(\psi)}$ in the finite formula
\eqref{eq:delta-finite-realization}; the quotient $Q_0$ has one element
and its summand is $1$. For conductor one, take
$\gamma=\varpi_E^{n_E(\psi)+1}$ and identify $Q_1$ with
$k_E^\times$. The sum is $-\tau_{k_E}(\bar\theta,\bar\psi)$ by
\eqref{eq:tau-definition}, which gives the second formula.
\end{proof}

\subsection{Norm and trace, elementary identities, and the error term}

For a finite extension $K/F$ of nonarchimedean local fields, write
\[
 \Nm_{K/F}:K^\times\to F^\times,
 \qquad
 \Tr_{K/F}:K\to F.
\]
For a nontrivial additive character $\psi_F$ and a quasi-character
$\chi_F$, put
\[
 \psi_K=\psi_F\circ\Tr_{K/F},
 \qquad
 \chi_K=\chi_F\circ \Nm_{K/F}.
\]
If $K/F$ is abelian, set
\begin{equation*}
 S(K/F)=\{\mu\in\operatorname{Hom}_{\mathrm{cont}}(F^\times,\C^\times):
   \mu\circ \Nm_{K/F}=1\}.
\end{equation*}
By the universal property of the quotient, this is the group of
continuous characters of the topological quotient
$F^\times/\Nm_{K/F}(K^\times)$.  For cyclic $K/F$ of prime degree, the
quotient and $S(K/F)$ are proved directly to have order $\ell$ in
Corollary~\ref{cor:prime-norm-characters} below.  That calculation uses
the norm filtration and not local class field theory.

\begin{lemma}[Change of additive character]\label{lem:additive-scale}
For $a\in E^\times$, put $\psi^a(x)=\psi(ax)$.  Then
\begin{equation*}
 \Delta_E(\theta,\psi^a)
 =\theta(a)\Delta_E(\theta,\psi).
\end{equation*}
\end{lemma}

\begin{proof}
If $v_E(a)=r$, then $n_E(\psi^a)=n_E(\psi)+r$.  Hence $a\gamma$ is
admissible for $(\theta,\psi^a)$ whenever $\gamma$ is admissible for
$(\theta,\psi)$, and
$\psi^a(u/(a\gamma))=\psi(u/\gamma)$.
\end{proof}

\begin{corollary}[Scaling invariance of the First Main identity]
\label{cor:additive-independence}
For every $a\in F^\times$, the First Main Lemma for $\psi_F$ is
equivalent to the First Main Lemma for $\psi_F^a$.
\end{corollary}

\begin{proof}
By Corollary~\ref{cor:prime-norm-characters}, the group $S(K/F)$ has
$\ell$ elements.  By Lemma~\ref{lem:additive-scale}, the two sides of \eqref{eq:first-main} acquire
the respective factors
\[
 \chi_K(a)\prod_{\mu\in S(K/F)}\mu(a)
 \quad\text{and}\quad
 \chi_F(a)^\ell\prod_{\mu\in S(K/F)}\mu(a).
\]
Since $\Nm_{K/F}(a)=a^\ell$ for $a\in F^\times$, these factors agree.
\end{proof}

\begin{lemma}[The trivial character]\label{lem:delta-one}
For every nontrivial additive character $\psi$,
\begin{equation}\label{eq:trivial-local-constant}
 \Delta_E(1,\psi)=1.
\end{equation}
\end{lemma}

\begin{proof}
Here $m_E(1)=0$.  Formula
\eqref{eq:unramified-character-local-constant} gives the result.
\end{proof}

\begin{lemma}[Inverse pairing]\label{lem:inverse-pair}
Let $\mu$ be a finite-order character of $E^\times$.  Then
\begin{equation*}
 \Delta_E(\mu,\psi)\Delta_E(\mu^{-1},\psi)=\mu(-1).
\end{equation*}
\end{lemma}

\begin{proof}
Choose a common admissible $\gamma$ and write
$I_\mu=\int_{U_E^0}\psi(u/\gamma)\mu(u)^{-1}\,du$.
Both $I_\mu$ and $I_{\mu^{-1}}$ are nonzero by
Corollary~\ref{cor:local-gauss-integral-nonzero}.  Since $\mu$ and
$\psi$ are unitary,
\[
 \overline{I_\mu}
 =\int_{U_E^0}\psi(-u/\gamma)\mu(u)\,du
 =\mu(-1)I_{\mu^{-1}}.
\]
The product of the two phases is therefore $\mu(-1)^{-1}=\mu(-1)$,
and the factors $\mu(\gamma)$ and $\mu^{-1}(\gamma)$ cancel.
\end{proof}

\begin{corollary}[Product of odd-degree norm-character local constants]\label{cor:S-product-one}
If $K/F$ is cyclic of odd prime degree, then
\begin{equation}\label{eq:S-product-one}
 \prod_{\mu\in S(K/F)}\Delta_F(\mu,\psi_F)=1.
\end{equation}
\end{corollary}

\begin{proof}
The finiteness of $S(K/F)$ follows from
Corollary~\ref{cor:prime-norm-characters}.  For $x\in F^\times$,
viewed as an element of $K^\times$, one has
$\Nm_{K/F}(x)=x^\ell$.  Hence, for every $\mu\in S(K/F)$,
$\mu(x)^\ell=\mu(\Nm_{K/F}(x))=1$.
Thus every nontrivial $\mu$ has order $\ell$, which is odd, and in
particular $\mu(-1)=1$.  Apply Lemma~\ref{lem:inverse-pair} to each pair
$\mu,\mu^{-1}$ and \eqref{eq:trivial-local-constant} to the trivial character.
\end{proof}

\begin{definition}[Error term]
Define
\begin{equation}\label{eq:error-term}
 \Err_{K/F}(\chi_F)=
 \frac{\Delta_K(\chi_K,\psi_K)
       \prod_{\mu\in S(K/F)}\Delta_F(\mu,\psi_F)}
      {\prod_{\mu\in S(K/F)}\Delta_F(\mu\chi_F,\psi_F)}.
\end{equation}
The denominator is nonzero by
Corollary~\ref{cor:local-gauss-integral-nonzero}, and the First Main
Lemma is the assertion $\Err_{K/F}(\chi_F)=1$.
\end{definition}

\begin{lemma}[Twist invariance]\label{lem:twist-invariance}
For every $\nu\in S(K/F)$,
\begin{equation}\label{eq:error-twist-invariance}
 \Err_{K/F}(\nu\chi_F)=\Err_{K/F}(\chi_F).
\end{equation}
It therefore suffices to prove the identity  \eqref{eq:first-main}  for $\chi_F$ of minimal
conductor in its $S(K/F)$-orbit.
\end{lemma}

\begin{proof}
Since $\nu\circ \Nm_{K/F}=1$, the character on $K^\times$ is unchanged.
Multiplication by $\nu$ permutes the group $S(K/F)$, including its
trivial character, so the denominator is only reindexed.
\end{proof}

\section{Ramification and norm filtration for cyclic prime extensions}

Let $K/F$ be a finite Galois extension of nonarchimedean local fields,
with Galois group $G$.  For an integer $i\geq -1$, the lower
ramification group is
\[
 G_i=\{\sigma\in G:\sigma(x)\equiv x\pmod{\p_K^{i+1}}
                      \text{ for every }x\in\cO_K\}.
\]
For a real $u\geq-1$, put $G_u=G_{\lceil u\rceil}$.  Define the
Herbrand function and its inverse by
\[
 \HerbPhi_{K/F}(u)=\int_0^u\frac{ds}{[G_0:G_s]},
 \qquad
 \HerbPsi_{K/F}=(\HerbPhi_{K/F})^{-1}.
\]
The superscript ${\rm H}$ is used throughout to prevent confusion with
additive characters.  

When $K/F$ is cyclic of prime degree $\ell$ and ramified, it is totally
ramified and there is a unique integer $t\geq0$ such that
\[
 G=G_0=\cdots=G_t,\qquad G_{t+1}=1.
\]
Consequently
\begin{equation}\label{eq:Herbrand-prime-new}
 \HerbPsi_{K/F}(u)=
 \begin{cases}
 u,&-1\leq u\leq t,\\
 t+\ell(u-t),&u\geq t.
 \end{cases}
\end{equation}

For a finite separable extension $E/F$, let
$\mathfrak D_{E/F}\subset\cO_E$ denote the different ideal, and write
\[
 D_{E/F}=v_E(\mathfrak D_{E/F})
\]
for its exponent.

\begin{theorem}[Ramification, trace, and norm filtration for a cyclic prime extension]
\label{thm:graded-norm}\label{thm:norm-filtration}
Let $K/F$ be ramified cyclic of prime degree $\ell$, with lower break
$t$.  Then:
\begin{enumerate}[label=(\alph*)]
\item The different exponent is
\[
 D_{K/F}=(\ell-1)(t+1),
\]
and for every integer $a$,
\begin{equation}\label{eq:trace-ideal}
 \Tr_{K/F}(\p_K^a)=
 \p_F^{\left\lfloor(a+D_{K/F})/\ell\right\rfloor}.
\end{equation}

\item For $0\le i\le t$, let
\[
 \overline{\Nm}_i:\operatorname{gr}^iU_K\longrightarrow
                   \operatorname{gr}^iU_F,
 \qquad
 \operatorname{gr}^iU_E=U_E^i/U_E^{i+1},
\]
be the map induced by the norm.  For $0\le i<t$ it is an isomorphism.
At $i=t$, its kernel and cokernel both have order $\ell$, and its
kernel is the image of $G_t/G_{t+1}$ under the ramification map
\[
 \sigma\longmapsto
 \begin{cases}
  \overline{\sigma(\varpi_K)/\varpi_K},&t=0,\\[1mm]
  \overline{(\sigma(\varpi_K)-\varpi_K)/\varpi_K^{t+1}},&t>0.
 \end{cases}
\]

\item Splicing the graded maps below the break gives
\begin{equation}\label{eq:norm-below-break}
 K^\times/U_K^t\xrightarrow{\sim}F^\times/U_F^t,
 \qquad
 U_K^r/U_K^t\xrightarrow{\sim}U_F^r/U_F^t
 \quad(0\le r\le t).
\end{equation}
\item For every integer $r>t$, put
$a=\HerbPsi_{K/F}(r)$.  The norm induces an isomorphism
\[
 U_K^a/U_K^{a+1}\longrightarrow U_F^r/U_F^{r+1}.
\]
Moreover, for every $a\ge0$, the subgroup $\Nm_{K/F}(U_K^a)$ is closed in
$U_F^0$, and consequently
\begin{equation}\label{eq:norm-above-break}
 \Nm_{K/F}(U_K^{\HerbPsi_{K/F}(r)})=U_F^r,
 \qquad
 \Nm_{K/F}(U_K^{\HerbPsi_{K/F}(r)+1})=U_F^{r+1}.
\end{equation}
\end{enumerate}
\end{theorem}

\begin{proof}
Choose a uniformizer $\varpi_K$.  For every $\sigma\ne1$ one has
$v_K(\sigma(\varpi_K)-\varpi_K)=t+1$.  The different formula
\[
 v_K(\mathfrak D_{K/F})
 =\sum_{\sigma\ne1}v_K(\sigma(\varpi_K)-\varpi_K),
\]
which is \cite[Chapter~IV, \S1, Proposition~4]{Serre}, gives the first
assertion in (a).

For the trace formula, the trace-dual of $\p_K^a$ under
$(x,y)\mapsto\Tr_{K/F}(xy)$ is
$\p_K^{-a-D_{K/F}}$ by
\cite[Chapter~III, \S3, Proposition~7]{Serre}.  Write
$\Tr(\p_K^a)=\p_F^c$.  An element $y\in F$ belongs to the dual of
this trace ideal if and only if
\[
 y\p_K^a\subseteq\p_K^{-D_{K/F}},
 \qquad\text{equivalently}\qquad
 \ell v_F(y)+a\ge-D_{K/F}.
\]
Thus its dual is
$\p_F^{\lceil(-a-D_{K/F})/\ell\rceil}$, whereas the dual of
$\p_F^c$ is $\p_F^{-c}$.  Hence
$c=\left\lfloor\frac{a+D_{K/F}}\ell\right\rfloor$,
which proves \eqref{eq:trace-ideal}.

The graded maps in (b), including the critical kernel, are the
cyclic-prime specialization of
\cite[Chapter~V, \S3, Propositions~4--5]{Serre}.  For (c), splice the
graded isomorphisms for $0\le i<t$.  The valuation quotient is preserved
because the residue degree is one:
$v_F(\Nm_{K/F}x)=v_K(x)$.  If $t=0$, then
$K^\times/U_K^0\simeq\mathbf Z$ and norm is the identity on the
valuation quotient.  If $t>0$, the residue norm
$U_K^0/U_K^1\to U_F^0/U_F^1$ is $x\mapsto x^\ell=x^p$, hence is
Frobenius, and the remaining graded pieces below $t$ are isomorphisms.
This proves \eqref{eq:norm-below-break}.

For (d), the above-break graded isomorphism is
\cite[Chapter~V, \S3, Corollaries~1--3]{Serre}.  Each $U_K^a$ is
compact, so $\Nm_{K/F}(U_K^a)$ is compact and hence closed in $U_F^0$.
Successive lifting through the graded isomorphisms therefore converges
to a norm and gives \eqref{eq:norm-above-break}.
\end{proof}

\begin{lemma}[Frobenius differences and trace-zero hyperplanes]
\label{lem:frobenius-difference-trace}
Let $k/\F_p$ be a finite field extension and let
$\lambda\in k^\times$.  Then the $\F_p$-linear maps
\[
 \wp(z)=z^p-z,
 \qquad
 \wp_\lambda(z)=z^p-\lambda^{p-1}z
\]
satisfy
\[
 \ker(\wp)=\F_p,
 \qquad
 \operatorname{im}(\wp)=\ker\!\left(\Tr_{k/\F_p}\right),
\]
and
\begin{equation}\label{eq:scaled-frobenius-trace-image}
 \ker(\wp_\lambda)=\lambda\F_p,
 \qquad
 \operatorname{im}(\wp_\lambda)
 =\lambda^p\ker\!\left(\Tr_{k/\F_p}\right).
\end{equation}
Under the perfect trace pairing $(a,x)\mapsto\Tr_{k/\F_p}(ax)$, the
annihilator of the last image is $\lambda^{-p}\F_p$.
\end{lemma}

\begin{proof}
The kernel of $\wp$ is the prime field.  Hence its image has
$\F_p$-dimension $[k:\F_p]-1$.  Frobenius invariance of the trace gives
$\Tr(z^p-z)=0$, so the image is contained in the trace-zero hyperplane;
the two spaces have the same dimension and therefore coincide.  Since
$\wp_\lambda(\lambda y)=\lambda^p(y^p-y)$,
the scaled kernel and image follow.  The orthogonal complement of
$\ker(\Tr_{k/\F_p})$ is $\F_p$: it contains $\F_p$, and both  have
dimension one.  Scaling the second variable by $\lambda^p$ gives the
annihilator $\lambda^{-p}\F_p$.
\end{proof}

\begin{lemma}[Strong wild symmetric-coefficient estimate]
\label{lem:wild-symmetric-bound}
Assume that $K/F$ is wildly ramified cyclic of prime degree
$\ell=p=\operatorname{char}k_F$, with lower break $t$, and put
$T=t+1$.
For $x\in K$ let $s_j(x)$ be the $j$th elementary symmetric function
of its $p$ conjugates.  If $v_K(x)\ge a$, then, for
$1\le j\le p-1$,
\begin{equation}\label{eq:symmetric-bound}
 v_F(s_j(x))\ge
 \left\lfloor\frac{ja+(p-1)T}{p}\right\rfloor.
\end{equation}
For the last coefficient,
\begin{equation}\label{eq:symmetric-bound-last}
 v_F(s_p(x))=v_F(\Nm_{K/F}x)=v_K(x)\ge a.
\end{equation}
\end{lemma}

\begin{proof}
Put $D=D_{K/F}=v_K(\mathfrak D_{K/F})$.  By
Theorem~\ref{thm:norm-filtration}(a),
$D=(p-1)T$.  Put $P_i(x)=\Tr_{K/F}(x^i)$.  Since
$x^i\in\p_K^{ia}$, equation
\eqref{eq:trace-ideal} gives
\begin{equation*}
 v_F(P_i(x))\ge
 \left\lfloor\frac{ia+D}{p}\right\rfloor
 \qquad(i\ge1).
\end{equation*}
We now use Newton's identities
\begin{equation*}
 j s_j(x)=\sum_{i=1}^j(-1)^{i-1}s_{j-i}(x)P_i(x),
 \qquad s_0(x)=1.
\end{equation*}
Because $j<p$, the integer $j$ is a unit of $\cO_F$.  We prove
\eqref{eq:symmetric-bound} by induction on $j$.  For $j=1$ it is
equation~\eqref{eq:trace-ideal}, applied to $x$.  Suppose it is known below $j$.
The term $i=j$ in Newton's identity has the required
valuation by \eqref{eq:trace-ideal}, applied to $x^j$.
For $i<j$, the induction hypothesis and \eqref{eq:trace-ideal},
applied to $x^i$, give
\[
 \left\lfloor\frac{(j-i)a+D}{p}\right\rfloor
 +\left\lfloor\frac{ia+D}{p}\right\rfloor.
\]
For real numbers $A,B$ one has
$\lfloor A\rfloor+\lfloor B\rfloor\ge\lfloor A+B\rfloor-1$.
Here the sum before taking floors is
$\frac{ja+D}{p}+\frac Dp$.
The extension is wild, so $T\ge2$ and
$D/p=(p-1)T/p\ge1$.  Hence the last displayed integral lower bound is
at least $\lfloor(ja+D)/p\rfloor$.  Every term in
Newton's identity therefore has the required valuation.
Finally $s_p(x)=\Nm_{K/F}(x)$, and total ramification gives
$v_F(\Nm_{K/F}x)=v_K(x)$.

Alternatively, for $1\le j<p$, the cyclic group $G=\Gal(K/F)$ acts
freely on its $j$-element subsets. Choosing one subset $S$ from each
orbit gives
\[
 s_j(x)=\sum_{[S]}\Tr_{K/F}\!\left(\prod_{\sigma\in S}\sigma(x)\right).
\]
Every product lies in $\p_K^{ja}$, so \eqref{eq:trace-ideal} gives
\eqref{eq:symmetric-bound} directly.
\end{proof}

\begin{proposition}[Critical norm polynomial at a wild break]
\label{prop:critical-norm-polynomial}
Let $K/F$ be totally ramified cyclic of prime degree
$p=\operatorname{char}k_F$, with positive lower break $t$, and identify
the residue fields with $k$.
Choose a uniformizer $\varpi_K$, put
$\varpi_F=\Nm_{K/F}(\varpi_K)$, choose a generator $\sigma$ of
$\Gal(K/F)$, and set
\begin{equation}\label{eq:critical-ramification-coordinate}
 \lambda=
 \overline{\frac{\sigma(\varpi_K)-\varpi_K}
                   {\varpi_K^{t+1}}}\in k^\times.
\end{equation}
For $z,w\in k$, let $\widetilde z,\widetilde w\in\cO_F\subset\cO_K$
be any lifts.  We use the standard identifications
\[
 \operatorname{gr}^tU_K=U_K^t/U_K^{t+1}\simeq(k,+),
 \qquad
 (1+\varpi_K^t\widetilde z)U_K^{t+1}\longleftrightarrow z,
\]
and
\[
 \operatorname{gr}^tU_F=U_F^t/U_F^{t+1}\simeq(k,+),
 \qquad
 (1+\varpi_F^t\widetilde w)U_F^{t+1}\longleftrightarrow w.
\]
In these source and target coordinates, the critical graded norm is the
additive polynomial
\begin{equation}\label{eq:critical-norm-polynomial-extracted}
 P_\lambda(Z)=Z^p-\lambda^{p-1}Z.
\end{equation}
Consequently
\begin{equation}\label{eq:critical-norm-image}
 \ker P_\lambda=\lambda\F_p,
 \qquad
 \operatorname{im}P_\lambda
 =\lambda^p\ker\!\left(\Tr_{k/\F_p}\right),
\end{equation}
and its annihilator under the trace pairing is
$\lambda^{-p}\F_p$.  The linear coefficient is the normalized trace:
\begin{equation}\label{eq:critical-norm-trace-coefficient}
 \overline{\frac{\Tr_{K/F}(\varpi_K^t)}{\varpi_F^t}}
 =-\lambda^{p-1}.
\end{equation}
\end{proposition}

\begin{proof}
For $2\le j\le p-1$, let $s_j$ be the $j$th elementary symmetric
function of the conjugates of $\varpi_K^t$.
Lemma~\ref{lem:wild-symmetric-bound} gives
\[
 v_F(s_j)\ge
 \left\lfloor\frac{jt+(p-1)(t+1)}p\right\rfloor
 \ge t+1,
\]
since $(j-1)t\ge1$.  There are no such terms when $p=2$.
The trace-ideal formula \eqref{eq:trace-ideal} gives
$\Tr_{K/F}(\varpi_K^t)\in\p_F^t$, while
\[
 \Nm_{K/F}(\varpi_K^t\widetilde z)=\varpi_F^t\widetilde z^{\,p}.
\]
Thus the norm expansion gives
\[
 \Nm_{K/F}(1+\varpi_K^t\widetilde z)
 \equiv1+\Tr_{K/F}(\varpi_K^t)\widetilde z
          +\varpi_F^t\widetilde z^{\,p}
 \pmod{\p_F^{t+1}}.
\]
The critical norm map is therefore represented by $Z^p+aZ$, where
\[
 a=\overline{\frac{\Tr_{K/F}(\varpi_K^t)}{\varpi_F^t}}.
\]
By Theorem~\ref{thm:graded-norm}(b), its kernel is
$\lambda\F_p$.  In particular $\lambda^p+a\lambda=0$; since
$\lambda\ne0$, this gives $a=-\lambda^{p-1}$.
This proves \eqref{eq:critical-norm-polynomial-extracted} and
\eqref{eq:critical-norm-trace-coefficient}.
The kernel, image, and annihilator assertions follow from
Lemma~\ref{lem:frobenius-difference-trace}.
\end{proof}

\begin{corollary}[Norm truncation]
\label{cor:wild-norm-truncation}
Under the hypotheses of Lemma~\ref{lem:wild-symmetric-bound}, let
$q,r\ge0$ and assume
\begin{equation*}
 2q+(p-1)T\ge pr.
\end{equation*}
Then, for every $y\in\p_K^q$,
\begin{equation}\label{eq:wild-norm-truncation}
 \Nm_{K/F}(1+y)
 \equiv1+\Tr_{K/F}(y)+\Nm_{K/F}(y)
 \pmod{\p_F^r}.
\end{equation}
\end{corollary}

\begin{proof}
For $2\le j\le p-1$, Equation~\eqref{eq:symmetric-bound} gives
\[
 v_F(s_j(y))\ge
 \left\lfloor\frac{jq+(p-1)T}{p}\right\rfloor
 \ge
 \left\lfloor\frac{2q+(p-1)T}{p}\right\rfloor\ge r.
\]
The norm expansion then gives
\eqref{eq:wild-norm-truncation}.
\end{proof}

\begin{proposition}[Stable-to-critical normalized trace]
\label{prop:stable-critical-trace}
Let $K/F$ be totally ramified cyclic of odd prime degree $p$, with
positive lower break $t$.  Choose a uniformizer $\varpi_K$ and a generator $\sigma$ of
$\Gal(K/F)$, put $\varpi_F=\Nm_{K/F}(\varpi_K)$, and define $\lambda$ by
\eqref{eq:critical-ramification-coordinate}.  Suppose $d\ge t+1$
and put
\[
 d'=pd-\frac{p-1}{2}t.
\]
Then
\begin{equation*}
 \frac{\varpi_F}{\varpi_K^p}
 \equiv1\pmod{\p_K^{t+1}},
\end{equation*}
\begin{equation*}
 \varpi_K^{2d'}
 \equiv \varpi_F^{2d-t}\varpi_K^t
 \pmod{\p_K^{2d'+1}},
\end{equation*}
and
the normalized traces are units with common residue
\begin{equation}\label{eq:stable-critical-trace-residue}
 \overline{\frac{\Tr_{K/F}(\varpi_K^{2d'})}{\varpi_F^{2d}}}
 =\overline{\frac{\Tr_{K/F}(\varpi_K^t)}{\varpi_F^t}}
 =-\lambda^{p-1}.
\end{equation}
\end{proposition}

\begin{proof}
For $j\in\F_p$, put
$a_j=\frac{\sigma^j(\varpi_K)}{\varpi_K}-1$.
The ramification homomorphism gives $a_j\in\p_K^t$ and
$\overline{a_j/\varpi_K^t}=j\lambda$.  Products of at least two of the
$a_j$ lie in $\p_K^{t+1}$, and $\sum_{j\in\F_p}j=0$.  Hence
\[
 \frac{\varpi_F}{\varpi_K^p}
 =\prod_{j\in\F_p}(1+a_j)
 \equiv1\pmod{\p_K^{t+1}},
\]
which proves the first congruence.  Since
$2d'=p(2d-t)+t$, raising the inverse ratio to the power $2d-t$ and
multiplying by $\varpi_K^t$ gives
the second congruence.

The trace-ideal formula \eqref{eq:trace-ideal} sends the difference in
the second congruence into $\p_F^{2d+1}$.
After division by $\varpi_F^{2d}$, the normalized traces have equal
residues.  In the critical source
coordinate the linear coefficient of the graded norm is
$\overline{\Tr(\varpi_K^t)/\varpi_F^t}$.  Proposition
\ref{prop:critical-norm-polynomial} identifies it with
$-\lambda^{p-1}$, which is nonzero.  This proves
\eqref{eq:stable-critical-trace-residue} and the unit assertions.
\end{proof}

\begin{corollary}[Different and additive conductors]\label{cor:additive-conductor}
Let $\psi_K=\psi_F\circ\Tr_{K/F}$.  If $K/F$ is totally ramified of
degree $\ell$, then
\begin{equation}\label{eq:additive-conductor-trace}
 n_K(\psi_K)=\ell n_F(\psi_F)+D_{K/F}.
\end{equation}
\end{corollary}

\begin{proof}
The largest fractional ideal $I$ on which $\psi_F\circ\Tr$ is trivial
is characterized by $\Tr(I)\subseteq\p_F^{-n_F(\psi_F)}$.  By the trace
duality used in the proof of \eqref{eq:trace-ideal}, this ideal is
$\p_K^{-\ell n_F(\psi_F)-D_{K/F}}$.
\end{proof}

\begin{corollary}[Ramified norm quotient and norm characters]
\label{cor:norm-characters}
Let $K/F$ be ramified cyclic of prime degree $\ell$.  Then
$F^\times/\Nm_{K/F}(K^\times)$ is cyclic of order $\ell$, the group
$S(K/F)$ has $\ell$ elements, and every nontrivial
$\mu\in S(K/F)$ has conductor exponent $t+1$.
\end{corollary}

\begin{proof}
Write $\Nm=\Nm_{K/F}$.
The first isomorphism in \eqref{eq:norm-below-break} shows that
$F^\times=\Nm(K^\times)U_F^t$.
Its injectivity also shows that
$\Nm(K^\times)\cap U_F^t=\Nm(U_K^t)$.
Consequently
$F^\times/\Nm(K^\times)\simeq U_F^t/\Nm(U_K^t)$.
Taking $r=t+1$ in \eqref{eq:norm-above-break} gives
\[
 U_F^{t+1}=\Nm(U_K^{\HerbPsi_{K/F}(t+1)})\subseteq \Nm(U_K^t),
\]
because $\HerbPsi_{K/F}(t+1)>t$.  Hence
\[
 U_F^t/\Nm(U_K^t)
 \simeq
 \operatorname{coker}\left(
 U_K^t/U_K^{t+1}\longrightarrow U_F^t/U_F^{t+1}
 \right).
\]
By Theorem~\ref{thm:norm-filtration}(b), this cokernel has order
$\ell$.  Thus $F^\times/\Nm(K^\times)$ has order $\ell$, and is
cyclic because $\ell$ is prime.

By \eqref{eq:norm-above-break}, $U_F^{t+1}\subseteq \Nm(K^\times)$,
so $\Nm(K^\times)$ is open and the quotient is discrete.  The elements
of $S(K/F)$ are therefore exactly the complex characters of this cyclic
quotient.  A
character is determined by the image of a generator, which may be any
$\ell$th root of unity.  Hence $S(K/F)$ has $\ell$ elements.

The identification $F^\times/\Nm(K^\times)\simeq U_F^t/\Nm(U_K^t)$ shows that
$U_F^t$ maps onto the norm quotient, whereas $U_F^{t+1}$ maps to the
identity.  A nontrivial norm character is therefore nontrivial on
$U_F^t$ and trivial on $U_F^{t+1}$.  Its conductor exponent is exactly
$t+1$.
\end{proof}

\begin{proposition}[Conductor of norm pullback]\label{prop:pullback-conductor}
Let $K/F$ be ramified cyclic of prime degree $\ell$, with break $t$.
Let $m=m_F(\chi_F)$ and $m'=m_K(\chi_F\circ \Nm_{K/F})$.
\begin{enumerate}[label=(\roman*)]
\item If $m>t+1$, then
\begin{equation}\label{eq:pullback-high}
 m'-1=\HerbPsi_{K/F}(m-1).
\end{equation}
\item If $m\leq t$, then $m'=m$.
\item If $m=t+1$, then $m'\leq m$; moreover $m'<m$ if and only if
there is a $\mu\in S(K/F)$ for which
$m_F(\mu\chi_F)<t+1$.
\end{enumerate}
\end{proposition}

\begin{proof}
For (i), put $a=\HerbPsi_{K/F}(m-1)$.  By
\eqref{eq:norm-above-break}, $\Nm_{K/F}(U_K^{a+1})=U_F^m$ and
$\Nm_{K/F}(U_K^a)=U_F^{m-1}$.  Since $m$ is the exact conductor of $\chi_F$,
the pullback is trivial on $U_K^{a+1}$ and nontrivial on $U_K^a$.
Thus $m'=a+1$, which is \eqref{eq:pullback-high}.

For (ii), first separate $m=0$: an unramified character is trivial on
$U_F^0$, hence its norm pullback is unramified and $m'=0$.  Now assume
$1\le m\le t$.  The character $\chi_F$ is nontrivial on
$U_F^{m-1}/U_F^m$.  The graded norm map is an isomorphism at every
level $<t$, so its pullback is nontrivial on
$U_K^{m-1}/U_K^m$ and trivial on $U_K^m$.  Hence $m'=m$.

For (iii), the pullback is always trivial on $U_K^{t+1}$, so
$m'\le t+1=m$.  It has smaller conductor precisely when the restriction
of $\chi_F$ to the image
\[
 I=\operatorname{im}\bigl(U_K^t/U_K^{t+1}\longrightarrow
                          U_F^t/U_F^{t+1}\bigr)
\]
is trivial.  By Theorem~\ref{thm:norm-filtration}(b), the quotient
$(U_F^t/U_F^{t+1})/I$ has order $\ell$.  The proof of
Corollary~\ref{cor:norm-characters} identifies this quotient with
$F^\times/\Nm_{K/F}(K^\times)$, so its character group is exactly the
restriction of $S(K/F)$ to the last unit layer.  Therefore
$\chi_F|_I=1$ if and only if there is a unique restriction
$\mu|_{U_F^t/U_F^{t+1}}$ with
$(\mu\chi_F)|_{U_F^t/U_F^{t+1}}=1$, equivalently
$m_F(\mu\chi_F)<t+1$.
\end{proof}

This is \cite[Lemmas 8.8 and 8.12, pp.~73 and 78--79]{Langlands}.

\begin{lemma}[Unramified norm, trace, and conductor compatibility]
\label{lem:unramified-compatibility}
Let $K/F$ be an unramified extension of degree $f$, and let $\varpi$ be a
uniformizer of $F$.  Then $\varpi$ is also a uniformizer of $K$,
$[k_K:k_F]=f$, and the following statements hold.
\begin{enumerate}[label=(\alph*)]
\item Reduction commutes with trace and norm:
\begin{equation}\label{eq:unramified-residue-trace-norm}
 \overline{\Tr_{K/F}(x)}=\Tr_{k_K/k_F}(\bar x)
 \quad(x\in\cO_K),
 \qquad
 \overline{\Nm_{K/F}(u)}=\Nm_{k_K/k_F}(\bar u)
 \quad(u\in\cO_K^\times).
\end{equation}
\item For every integer $a$,
$\Tr_{K/F}(\p_K^a)=\p_F^a$.
\item For every $r\geq0$,
\begin{equation}\label{eq:unramified-norm-units}
 \Nm_{K/F}(U_K^r)=U_F^r,
 \qquad \Nm_{K/F}(\varpi)=\varpi^f.
\end{equation}
\item If $\psi_K=\psi_F\circ\Tr_{K/F}$ and
$\chi_K=\chi_F\circ \Nm_{K/F}$, then
\[
 n_K(\psi_K)=n_F(\psi_F),
 \qquad m_K(\chi_K)=m_F(\chi_F).
\]
\item The quotient $F^\times/\Nm_{K/F}(K^\times)$ is cyclic of order $f$,
generated by the class of $\varpi$.
\end{enumerate}
\end{lemma}

\begin{proof}
Write $\Nm=\Nm_{K/F}$.
Since $K/F$ is unramified, its ramification index is $1$.  Thus
$\p_K=\varpi\cO_K$, and the residue extension has degree $f$.
The conjugates of an integral element reduce to the conjugates of its
residue class.  Taking the first and last coefficients of the
characteristic polynomial therefore gives
\eqref{eq:unramified-residue-trace-norm}.

We use two elementary finite-field facts.  The map
$\Tr_{k_K/k_F}$ is onto: if $q=|k_F|$, its polynomial
$X+X^q+\cdots+X^{q^{f-1}}$
is nonzero and has degree $q^{f-1}<q^f=|k_K|$, so it cannot vanish on
all of $k_K$; a nonzero $k_F$-linear map to the one-dimensional space
$k_F$ is surjective.  The map $\Nm_{k_K/k_F}:k_K^\times\to k_F^\times$
is onto because $k_K^\times$ is cyclic and
$\Nm_{k_K/k_F}(x)=x^{(q^f-1)/(q-1)}$
has image of order $q-1$.

Choose $x\in\cO_K$ whose residue has trace $1$.  Then
$\Tr_{K/F}(x)$ is a unit of $\cO_F$; after multiplying $x$ by its
inverse, we obtain an integral element of trace $1$.  Hence
$\Tr_{K/F}(\cO_K)=\cO_F$.  Since
$\p_K^a=\varpi^a\cO_K$ for every $a\in\mathbf Z$, this proves
part~(b).

For $r\geq1$ and $x\in\cO_K$, multiplication of the conjugates gives
\begin{equation*}
 \Nm(1+\varpi^r x)
 \equiv 1+\varpi^r\Tr_{K/F}(x)\pmod{\p_F^{r+1}}.
\end{equation*}
Thus the map induced by norm on
$U_K^r/U_K^{r+1}\to U_F^r/U_F^{r+1}$ is the finite-field trace and is
surjective.  Starting with $v\in U_F^r$, choose a lift of its class,
divide $v$ by the norm of that lift, and repeat at levels
$r+1,r+2,\ldots$.  Completeness makes the product of the successive
lifts converge to an element of $U_K^r$ with norm $v$.  This proves
$\Nm(U_K^r)=U_F^r$ for $r\geq1$.  Surjectivity of the finite-field norm,
followed by the case $r=1$, gives $\Nm(U_K^0)=U_F^0$.  Finally, because
$\varpi\in F$, its norm is $\varpi^f$.  This proves
\eqref{eq:unramified-norm-units}.

Put $n=n_F(\psi_F)$.  By part~(b),
$\psi_F\circ\Tr$ is trivial on $\p_K^{-n}$ and is not trivial on
$\p_K^{-n-1}$, so $n_K(\psi_K)=n$.  By
\eqref{eq:unramified-norm-units}, for every $r\geq0$ the character
$\chi_F\circ \Nm$ is trivial on $U_K^r$ if and only if $\chi_F$ is
trivial on $U_F^r$.  Hence the multiplicative conductors are equal.
This proves the conductor equalities in part~(d).

Every $y\in K^\times$ has the form $y=\varpi^a u$ with
$u\in U_K^0$, so
$\Nm(K^\times)=\varpi^{f\mathbf Z}U_F^0$
by \eqref{eq:unramified-norm-units}.  The last assertion follows.
\end{proof}

\begin{corollary}[Norm characters in cyclic prime degree]
\label{cor:prime-norm-characters}
Let $K/F$ be cyclic of prime degree $\ell$.  Then
$F^\times/\Nm_{K/F}(K^\times)$ is cyclic of order $\ell$, and
$S(K/F)$ has $\ell$ elements.
\end{corollary}

\begin{proof}
If $K/F$ is ramified, this is
Corollary~\ref{cor:norm-characters}.  Suppose that it is unramified.
Lemma~\ref{lem:unramified-compatibility}(e) gives a cyclic norm
quotient of order $\ell$, and
\eqref{eq:unramified-norm-units} shows that the norm group contains the
open subgroup $U_F^0$.  Thus the quotient is discrete, and the remainder
of the proof is the same as in Corollary~\ref{cor:norm-characters}.
\end{proof}

By \eqref{eq:error-twist-invariance}, we henceforth retain this
minimal-conductor choice of $\chi_F$ in its $S(K/F)$-orbit.  If $K/F$ is
ramified, Proposition~\ref{prop:pullback-conductor}(ii)--(iii) then gives
\begin{equation}\label{eq:minimal-critical}
 m\leq t+1\quad\Longrightarrow\quad m'=m.
\end{equation}

\section{Lamprecht's formula}

Let $E$ be a nonarchimedean local field, let $\psi:E\to\C^\times$ be a
nontrivial additive character, and let $\theta$ be a quasi-character of
conductor $m=2d+\varepsilon>1$, with $\varepsilon\in\{0,1\}$.  Choose
$\gamma\in E^\times$ with $\gamma\cO_E=\p_E^{m+n_E(\psi)}$.

\begin{lemma}[Finite stationary-phase duality]\label{lem:finite-duality}
Let $n=n_E(\psi)$, let $M,q,r\in\mathbf Z$ with $r\le q$, and let
$\Gamma\in E^\times$ satisfy
\[
 \Gamma\cO_E=\p_E^{M+n}.
\]
Then
\begin{equation*}
 \p_E^{M-q}/\p_E^{M-r}\ \times\
 \p_E^r/\p_E^q
 \longrightarrow\C^\times,
 \qquad(\bar c,\bar x)\longmapsto\psi(cx/\Gamma),
\end{equation*}
is a perfect pairing.
\end{lemma}

\begin{proof}
The pairing is well defined: changing $c$ by an element of
$\p_E^{M-r}$ or $x$ by an element of $\p_E^q$ changes $cx/\Gamma$ by
an element of $\p_E^{-n}$, on which $\psi$ is trivial.  The annihilator
of $\p_E^r/\p_E^q$ in the first factor is zero, because
\[
 c\p_E^r/\Gamma\subseteq\p_E^{-n}
 \quad\Longleftrightarrow\quad
 c\in\p_E^{M-r}.
\]
Similarly, the annihilator of the first factor in the second is zero,
because
\[
 \p_E^{M-q}x/\Gamma\subseteq\p_E^{-n}
 \quad\Longleftrightarrow\quad
 x\in\p_E^q.
\]
Thus the pairing is nondegenerate in both variables.  Both finite
additive groups have cardinality $q_E^{q-r}$, hence the induced maps to
the respective character groups are isomorphisms.
\end{proof}

\begin{lemma}[Stationary numerator class at a chosen depth]
\label{lem:stationary-class-existence}
Let $q=m_E(\theta)>1$, let $n=n_E(\psi)$, choose an integer $M$ and
an element $\Gamma\in E^\times$ such that
\[
 \Gamma\cO_E=\p_E^{M+n},
\]
and let $r$ satisfy
\[
 \ceil{q/2}\le r\le q-1.
\]
There is a unique class
\begin{equation*}
 \operatorname{St}_{\Gamma,r}(\theta)
 \in \p_E^{M-q}/\p_E^{M-r}
\end{equation*}
such that every representative $c$ of this class satisfies
\begin{equation}\label{eq:general-stationary-linearization}
 \theta(1+x)=\psi(cx/\Gamma)
 \qquad(x\in\p_E^r).
\end{equation}
Every representative $c$ has exact valuation
\begin{equation}\label{eq:general-stationary-valuation}
 v_E(c)=M-q.
\end{equation}
\end{lemma}

\begin{proof}
For $x,y\in\p_E^r$, the quotient
$\frac{(1+x)(1+y)}{1+x+y} =1+\frac{xy}{1+x+y}$
belongs to $U_E^q$, because $2r\ge q$.  Hence
$x\mapsto\theta(1+x)$ is an additive character of
$\p_E^r/\p_E^q$.  Lemma~\ref{lem:finite-duality}, applied with the
present $M,q,r,\Gamma$, identifies the character group of
$\p_E^r/\p_E^q$ with
$\p_E^{M-q}/\p_E^{M-r}$.  This gives the unique class
in the statement satisfying
\eqref{eq:general-stationary-linearization}.

Every representative belongs to $\p_E^{M-q}$.  If it belonged to
$\p_E^{M-q+1}$, then for every $x\in\p_E^{q-1}$ one would have
$cx/\Gamma\in\p_E^{-n}$.  Since $r\le q-1$, such $x$ lies in the
linearization domain, and
\eqref{eq:general-stationary-linearization} would make $\theta$ trivial
on $U_E^{q-1}$, contrary to the exact conductor $q$.  Thus
$v_E(c)=M-q$.  
\end{proof}

\begin{lemma}[Stationary-class calculus]
\label{lem:stationary-class-calculus}
Use the notation of Lemma~\ref{lem:stationary-class-existence}.
\begin{enumerate}[label=(\alph*)]
\item If $r\le s\le q-1$, restriction from $\p_E^r$ to $\p_E^s$
      sends $\operatorname{St}_{\Gamma,r}(\theta)$ to the class of the
      same representative modulo $\p_E^{M-s}$.
\item If $\Gamma'=a\Gamma$, then $ac$ represents the stationary class
      relative to $\Gamma'$ whenever $c$ represents the class relative
      to $\Gamma$.
\item Let $\theta_1,\theta_2$ be quasi-characters and suppose that, on a
      common depth $r$, coefficients $c_1,c_2$ satisfy
      \[
       \theta_i(1+x)=\psi(c_ix/\Gamma)\qquad(x\in\p_E^r).
      \]
      Then $c_1+c_2$ represents the restriction of
      $\theta_1\theta_2$ on that depth, and $jc_1$ represents the
      restriction of the integer power $\theta_1^j$.  These assertions
      are equalities of classes modulo $\p_E^{M-r}$.
\item Suppose $m_E(\theta_1)=m_E(\theta_2)=q$, take
      $\ceil{q/2}\le r\le q-1$, and choose representatives
      $c_i$ of $\operatorname{St}_{\Gamma,r}(\theta_i)$ for $i=1,2$.
      Then
      \begin{equation}\label{eq:stationary-leading-cancellation}
       m_E(\theta_1\theta_2)<q
       \quad\Longleftrightarrow\quad
       c_1+c_2\in\p_E^{M-q+1}.
      \end{equation}
      If $1<q'=m_E(\theta_1\theta_2)<q$, put
      $r'=\ceil{q'/2}$ and choose a representative
      $c'$ of $\operatorname{St}_{\Gamma,r'}(\theta_1\theta_2)$.  Then
      \[
       c'\equiv c_1+c_2\pmod{\p_E^{M-r}}.
      \]
\end{enumerate}
\end{lemma}

\begin{proof}
Parts (a) and (b) follow immediately from
\eqref{eq:general-stationary-linearization} and uniqueness.  Part (c)
follows by multiplying the two character identities; uniqueness at the
common depth identifies the resulting class modulo $\p_E^{M-r}$.

For (d), the product has conductor less than $q$ precisely when it is
trivial on $U_E^{q-1}$.  For $x\in\p_E^{q-1}$, part (c) gives
$(\theta_1\theta_2)(1+x)=\psi((c_1+c_2)x/\Gamma)$.
The annihilator of $\p_E^{q-1}$ is $\p_E^{M-q+1}$, which proves
\eqref{eq:stationary-leading-cancellation}.

Now suppose $1<q'=m_E(\theta_1\theta_2)<q$.  Since $r'\le r$, the
linearization defining $c'$ is valid for every $x\in\p_E^r$, and part
(c) gives the same character there using $c_1+c_2$.  Thus
$\psi((c'-c_1-c_2)x/\Gamma)=1 \qquad(x\in\p_E^r)$.
The annihilator of $\p_E^r$ is $\p_E^{M-r}$, proving the asserted
congruence.
\end{proof}


\begin{lemma}[Magnitude of a nondegenerate finite quadratic sum]
\label{lem:finite-quadratic-magnitude}
Let $V$ be a finite abelian group and let $q:V\to\C^\times$ satisfy
$|q(x)|=1$, $q(0)=1$, and
\[
 B(x,y)=\frac{q(x+y)}{q(x)q(y)}.
\]
Assume that $B$ is a nondegenerate bicharacter.  Then
\begin{equation}\label{eq:finite-quadratic-magnitude}
 \left|\sum_{x\in V}q(x)\right|^2=|V|.
\end{equation}
In particular, the sum is nonzero.
\end{lemma}

\begin{proof}
Put $S=\sum_xq(x)$.  Since $\overline{q(y)}=q(y)^{-1}$, the change of
variables $x=y+h$ gives
\[
 \begin{aligned}
 |S|^2
 &=\sum_{x,y\in V}q(x)\overline{q(y)}
   =\sum_{h,y\in V}q(y+h)q(y)^{-1}\\
 &=\sum_{h\in V}q(h)\sum_{y\in V}B(h,y).
 \end{aligned}
\]
For $h=0$ the inner sum is $|V|$.  For $h\ne0$, nondegeneracy says that
$B(h,\cdot)$ is a nontrivial character, whose sum is zero.  Hence only
$h=0$ survives and $|S|^2=|V|$.
\end{proof}

\begin{theorem}[Lamprecht]\label{thm:Lamprecht}
Put
\[
 \beta=\beta(\theta):=
 \operatorname{St}_{\gamma,d+\varepsilon}(\theta)
 \in\cO_E/\p_E^d.
\]
Equivalently, $\beta$ is the unique class satisfying
\begin{equation}\label{eq:Lamprecht-linearization}
 \theta(1+x)=\psi(\beta x/\gamma)
 \qquad(x\in\p_E^{d+\varepsilon}).
\end{equation}
Every representative of this class is a unit.  

Fix any representative,
again denoted $\beta\in\cO_E^\times$.  If $\varepsilon=0$, then
\begin{equation}\label{eq:Lamprecht-even}
 \Delta_E(\theta,\psi)=
 \theta(\gamma)\psi(\beta/\gamma)\theta(\beta)^{-1}.
\end{equation}
If $\varepsilon=1$, choose $\delta\cO_E=\p_E^d$; we call
$\delta$ a \emph{critical element}.  Via
$\bar x\mapsto\delta x\pmod{\p_E^{d+1}}$, identify $k_E$ with
$\p_E^d/\p_E^{d+1}$, the \emph{critical line}.  For
$\bar x\in k_E$, choose a lift $x\in\cO_E$ and put
\begin{equation}\label{eq:Lamprecht-Hasse-function}
 H_E(\bar x)=
 \psi(\delta\beta x/\gamma)\theta(1+\delta x)^{-1}.
\end{equation}
We call $H_E$ the \emph{critical function}.  It is well defined and
unimodular, and
\begin{equation}\label{eq:Lamprecht-odd}
 \Delta_E(\theta,\psi)=
 \theta(\gamma)\psi(\beta/\gamma)\theta(\beta)^{-1}
 \ph\!\left(\sum_{x\in k_E}H_E(x)\right).
\end{equation}
The right side of \eqref{eq:Lamprecht-odd} is independent of the
choice of $\delta$ with $\delta\cO_E=\p_E^d$.

Moreover, $H_E$ is a \emph{Hasse function} with polar character $\psi_0$,
meaning that, for $\bar x,\bar y\in k_E$,
\begin{equation}\label{eq:Hasse-functional-equation}
 H_E(\bar x+\bar y)=
 H_E(\bar x)H_E(\bar y)\psi_0(\bar x\bar y),
\end{equation}
where
\begin{equation}\label{eq:Lamprecht-residual-character}
 \psi_0(\bar z)=\psi(\beta\delta^2 z/\gamma),
 \qquad z\in\cO_E\text{ a lift of }\bar z\in k_E,
\end{equation}
is a well-defined nontrivial additive character of $k_E$.  Consequently
\[
 \left|\sum_{x\in k_E}H_E(x)\right|=q_E^{1/2}.
\]
\end{theorem}

\begin{proof}
Apply Lemma~\ref{lem:stationary-class-existence} with
\[
 q=M=m=2d+\varepsilon,
 \qquad \Gamma=\gamma,
 \qquad r=d+\varepsilon.
\]
Since $M-q=0$ and $M-r=d$, it gives the class
$\beta\in\cO_E/\p_E^d$ characterized by
\eqref{eq:Lamprecht-linearization};
\eqref{eq:general-stationary-valuation} gives $v_E(\beta)=0$ for every
representative.  This proves the first assertions of the theorem.

Put
\[
 I=\int_{U_E^0}\psi(u/\gamma)\theta(u)^{-1}\,du,
 \qquad A=\p_E^{d+\varepsilon}/\p_E^m.
\]
For each $x\in A$, substitute $u\mapsto u(1+x)$ in the integral.  Haar
measure on $U_E^0$ is translation invariant, and
\eqref{eq:Lamprecht-linearization} gives
\[
 I=\int_{U_E^0}\psi(u/\gamma)\theta(u)^{-1}
       \psi((u-\beta)x/\gamma)\,du.
\]
Average this identity over $x\in A$.  Character orthogonality and
Lemma~\ref{lem:finite-duality}, specialized to
$M=q=m$, $r=d+\varepsilon$, and $\Gamma=\gamma$, give
\begin{equation*}
 \frac1{|A|}\sum_{x\in A}\psi((u-\beta)x/\gamma)
 =\begin{cases}
 1,&u-\beta\in\p_E^d,\\
 0,&u-\beta\notin\p_E^d.
 \end{cases}
\end{equation*}
Hence only the compact open subset
$\beta(1+\p_E^d)$ survives.

Assume first that $m=2d$.  Write $u=\beta(1+z)$ with
$z\in\p_E^d$.  The linearization identity gives
$\theta(1+z)=\psi(\beta z/\gamma)$.  Therefore on the surviving set
\[
 \psi(u/\gamma)\theta(u)^{-1}
 =\psi(\beta/\gamma)\theta(\beta)^{-1}.
\]
The integral is this nonzero constant times the positive measure of the
surviving set.  Taking its phase proves \eqref{eq:Lamprecht-even}.

Now assume that $m=2d+1$.  Since $m>1$, one has $d\ge1$.  Partition
$\beta(1+\p_E^d)$ into its $q_E$ cosets of
$1+\p_E^{d+1}$.  They are represented by
$u_x=\beta(1+\delta x)$, where $x\in\cO_E$ is a lift of $\bar x\in k_E$.  If
$w\in\p_E^{d+1}$, then the quotient of the integrand at
$u_x(1+w)$ by its value at $u_x$ is
$\psi(u_xw/\gamma)\theta(1+w)^{-1}$.
Linearization gives $\theta(1+w)=\psi(\beta w/\gamma)$, while
$(u_x-\beta)w\in\p_E^{2d+1}=\p_E^m$; hence the displayed quotient is
one.  Thus the integrand is constant on each of the $q_E$ cosets.  In
particular, changing the lift $x$ of $\bar x$ does not change the value
of $H_E(\bar x)$, so $H_E$ is well defined.  It is
unimodular because $|\psi(\delta\beta x/\gamma)|=1$ and the
restriction of the quasi-character $\theta$ to the compact group
$U_E^0$ has values of absolute value one.  The value of the integrand
at $u_x$ is
$\psi(\beta/\gamma)\theta(\beta)^{-1}H_E(\bar x)$.
As in the even case, integrating a constant over each coset contributes
its value times the common positive Haar measure
$c=\operatorname{vol}(1+\p_E^{d+1})$.  Hence
\[
 I=c\,\psi(\beta/\gamma)\theta(\beta)^{-1}
      \sum_{\bar x\in k_E}H_E(\bar x).
\]
Since $I\ne0$, one has $\sum_{\bar x\in k_E}H_E(\bar x)\ne0$.  Taking phases therefore
gives \eqref{eq:Lamprecht-odd}.  If $\delta'=v\delta$ with $v\in\cO_E^\times$, and
$H_E'$ is formed using $\delta'$, then
$H_E'(\bar x)=H_E(\bar v\,\bar x)$,
because $vx$ is a lift of $\bar v\,\bar x$.  Multiplication by
$\bar v\in k_E^\times$ permutes $k_E$, so
$\sum_{\bar x}H_E'(\bar x)=
 \sum_{\bar x}H_E(\bar x)$.  This proves the asserted
independence of $\delta$.

For $\bar x,\bar y\in k_E$, choose lifts $x,y\in\cO_E$ and put
$w=\frac{\delta^2xy}{1+\delta(x+y)}$.
Then
$(1+\delta x)(1+\delta y)=(1+\delta(x+y))(1+w)$.
The additive linear factors in \eqref{eq:Lamprecht-Hasse-function}
cancel, so
\[
 \frac{H_E(\bar x+\bar y)}
      {H_E(\bar x)H_E(\bar y)}
 =\theta(1+w)=\psi(\beta w/\gamma).
\]
Since $w-\delta^2xy\in\p_E^{3d}$ and
$3d\ge2d+1=m$, the last expression equals
$\psi(\beta\delta^2xy/\gamma)$.  This proves
\eqref{eq:Hasse-functional-equation} and
\eqref{eq:Lamprecht-residual-character}.  Put
$a=\beta\delta^2/\gamma$.  Since
$v_E(a)=-n_E(\psi)-1$, changing a lift $z$ of $\bar z$ by an element
of $\p_E$ changes $az$ by an element of $\p_E^{-n_E(\psi)}$, on
which $\psi$ is trivial.  Hence $\psi_0$ is well defined.  If
$\psi_0$ were trivial, then $\psi$ would be trivial on
$a\cO_E=\p_E^{-n_E(\psi)-1}$, contradicting the definition of
$n_E(\psi)$.  Thus $\psi_0$ is nontrivial.  The bicharacter
$(x,y)\mapsto\psi_0(xy)$ is nondegenerate: for $x\ne0$, multiplication
by $x$ is a bijection of $k_E$, so $y\mapsto\psi_0(xy)$ is nontrivial.
Equation~\eqref{eq:finite-quadratic-magnitude} now gives
$\left|\sum_{x\in k_E}H_E(x)\right|=q_E^{1/2}$.
\end{proof}

This is \cite[Lemma 8.1, pp.~60--62, and Lemma 9.3, pp.~86--87]{Langlands}.
For the $p$-adic stationary evaluation in Dwork's normalization, see
\cite[Chapter~II, pp.~25--27]{Dwork}.

\begin{lemma}[Independence of the stationary representative]
\label{lem:complete-stationary-invariant}
The stationary expression in Lamprecht's formula is independent
of the representative of the class
$\beta\in\cO_E/\p_E^d$.  More precisely, if
$\beta'=\beta+h$ with $h\in\p_E^d$, then the right sides of
\eqref{eq:Lamprecht-even} and \eqref{eq:Lamprecht-odd}, respectively,
are unchanged when $\beta$ is replaced by $\beta'$ and the Hasse
function \eqref{eq:Lamprecht-Hasse-function} is formed with $\beta'$.
\end{lemma}

\begin{proof}
Suppose first that $m=2d$.  Since $\beta$ is a unit,
$z=h/\beta$ lies in $\p_E^d$, and the linearization formula gives
$\theta(1+z)=\psi(\beta z/\gamma)=\psi(h/\gamma)$.  Hence
\[
 \frac{\psi(\beta'/\gamma)\theta(\beta')^{-1}}
      {\psi(\beta/\gamma)\theta(\beta)^{-1}}
 =\psi(h/\gamma)\theta(1+h/\beta)^{-1}=1.
\]

Now let $m=2d+1$.  Choose $\delta\cO_E=\p_E^d$ and write
\[
 a=\frac{h}{\beta\delta}\in\cO_E,
 \qquad \bar a\in k_E
\]
for its residue class.  Let $H_\beta$ and $H_{\beta'}$ denote the
Hasse functions formed with $\beta$ and $\beta'$.  For
$\bar x\in k_E$, choose a lift $x\in\cO_E$.  Directly from
\eqref{eq:Lamprecht-Hasse-function},
\begin{equation*}
 H_{\beta'}(\bar x)
 =H_\beta(\bar x)\,\psi(\delta h x/\gamma).
\end{equation*}
The functional equation \eqref{eq:Hasse-functional-equation} gives
\[
 H_\beta(\bar x+\bar a)
 =H_\beta(\bar x)H_\beta(\bar a)
   \psi(\beta\delta^2ax/\gamma).
\]
Since $\beta\delta a=h$, one has
$\psi(\beta\delta^2ax/\gamma)=\psi(\delta hx/\gamma)$.  Therefore
\[
 H_{\beta'}(\bar x)
 =H_\beta(\bar a)^{-1}H_\beta(\bar x+\bar a),
 \qquad
 \sum_{\bar x\in k_E}H_{\beta'}(\bar x)
 =H_\beta(\bar a)^{-1}\sum_{\bar x\in k_E}H_\beta(\bar x).
\]
On the other hand,
\[
 \frac{\psi(\beta'/\gamma)\theta(\beta')^{-1}}
      {\psi(\beta/\gamma)\theta(\beta)^{-1}}
 =\psi(h/\gamma)\theta(1+\delta a)^{-1}
 =H_\beta(\bar a).
\]
Because the Hasse sums are nonzero, taking phases shows that these two
multipliers cancel.  Thus the product of the stationary and residue-field
factors is unchanged.
\end{proof}

For a chosen representative $\beta$ of the stationary class, set
\begin{equation}\label{eq:stationary-factor-definitions}
 \Delta_{2,E}(\theta)=
 \psi(\beta/\gamma)\theta(\beta)^{-1},
 \qquad
 \Delta_{3,E}(\theta)=
 \begin{cases}
 1,&m_E(\theta)\text{ even},\\
 \ph(\sum H_E),&m_E(\theta)\text{ odd}.
 \end{cases}
\end{equation}
Then
\begin{equation}\label{eq:Delta-factorization-new}
 \Delta_E(\theta,\psi)=\theta(\gamma)\Delta_{2,E}(\theta)
                                      \Delta_{3,E}(\theta).
\end{equation}
When the conductor is odd, $\Delta_{2,E}(\theta)$ and
$\Delta_{3,E}(\theta)$ separately depend on the chosen representative;
their product is independent by
Lemma~\ref{lem:complete-stationary-invariant}.  Every later comparison
therefore either keeps these two factors together or applies the
induced translation to the critical function.

\begin{proposition}[Stable twist]\label{lem:stable-twist-new}
Let $\nu$ be a quasi-character of $E^\times$.  If
$m_E(\theta)=2d+\varepsilon>1$ and $m_E(\nu)\leq d$, then, for any
representative $\beta$ of the stationary class of $\theta$,
\begin{equation}\label{eq:stable-twist-new}
 \Delta_E(\nu\theta,\psi)=
 \nu(\gamma/\beta)\Delta_E(\theta,\psi).
\end{equation}
\end{proposition}

\begin{proof}
Since $m_E(\nu)\le d<m$, the character $\nu$ is trivial on
$U_E^{m-1}$.  Thus $\nu\theta$ is nontrivial on
$U_E^{m-1}$ and trivial on $U_E^m$, so
$m_E(\nu\theta)=m_E(\theta)=m$.
Moreover, $\nu$ is trivial on $U_E^d$, hence on
$U_E^{d+\varepsilon}$.  Therefore $\theta$ and $\nu\theta$ have the
same restriction to the stationary layer $U_E^{d+\varepsilon}$, so
their stationary classes coincide.  Choose the same representative
$\beta$ for both.  If $m$ is even, both residual factors are $1$;
if $m$ is odd, their Hasse functions are identical.  Taking the ratio in
\eqref{eq:Delta-factorization-new} gives
$\nu(\gamma)\nu(\beta)^{-1}$, proving
\eqref{eq:stable-twist-new}.
\end{proof}

Compare \cite[Lemma~9.4, pp.~87--88]{Langlands}.

\begin{lemma}[Conductors and leading classes in a minimal norm-character orbit]
\label{lem:minimal-orbit-stationary}
Let $K/F$ be ramified cyclic of prime degree $\ell$, let $t$ be its
break, put $T=t+1$, and let $\tau$ be a nontrivial element of
$S(K/F)$.  Retain the minimal-conductor choice of $\chi_F$.
Put $m=m_F(\chi_F)$.  Then, for every nontrivial
power $\tau^j$,
\begin{equation}\label{eq:minimal-orbit-conductors}
 m_F(\tau^j\chi_F)=\max\{m,T\}.
\end{equation}
If $T>1$ and $m=T$, choose $\gamma$ with
\[
 \gamma\cO_F=\p_F^{m+n_F(\psi_F)},
\]
choose
\[
 \ceil{m/2}\le r\le m-1,
\]
and choose $a,b\in\cO_F^\times$ such that
\[
 \tau(1+x)=\psi_F(ax/\gamma),
 \qquad
 \chi_F(1+x)=\psi_F(bx/\gamma)
 \qquad(x\in\p_F^r).
\]
Then, for $j=0,\ldots,\ell-1$, viewed as integers in $F$,
\begin{equation}\label{eq:minimal-orbit-no-leading-cancellation}
 b+ja\notin\p_F.
\end{equation}
\end{lemma}

\begin{proof}
A nontrivial norm character has conductor $T$ by
Corollary~\ref{cor:norm-characters}.  If $m<T$, then $\chi_F$ is trivial
on $U_F^{T-1}$, whereas $\tau^j$ is not, so the product has conductor
$T$.  If $m>T$, then $\tau^j$ is trivial on $U_F^{m-1}$, whereas
$\chi_F$ is not, so the product has conductor $m$.  If $m=T$, every
member of the orbit has conductor at least $T$ by minimality of $\chi_F$ in its orbit, and at most $T$ because both factors are trivial on $U_F^T$.
This proves \eqref{eq:minimal-orbit-conductors}.

Assume $T>1$ and $m=T$.  For $x\in\p_F^r$, the two displayed
linearizations give
$(\tau^j\chi_F)(1+x)=\psi_F((b+ja)x/\gamma)$.
If $b+ja\in\p_F$, equation
\eqref{eq:stationary-leading-cancellation}, with $M=q=m$, would give
conductor less than $m$, contradicting
\eqref{eq:minimal-orbit-conductors}.  This proves
\eqref{eq:minimal-orbit-no-leading-cancellation}.\end{proof}

\section{Finite-field identities}

Let $k$ be a finite field and fix a nontrivial additive character
$\psi:(k,+)\to\C^\times$.  Multiplicative characters in this section
are characters of $k^\times$, unless another field is specified.
We use the normalized Gauss sum
$\tau_k(\chi,\psi)$ defined in \eqref{eq:tau-definition}; in
particular $\tau_k(1,\psi)=1$.  

\begin{lemma}[Finite-field Gauss-sum magnitude]
\label{lem:finite-gauss-magnitude}
For
\[
 G_k(\chi,\psi)=\sum_{x\in k^\times}\chi(x)\psi(x)
\]
one has $G_k(1,\psi)=-1$.  If $\chi$ is nontrivial, then
\begin{equation}\label{eq:finite-gauss-magnitude}
 |G_k(\chi,\psi)|^2=|k|.
\end{equation}
Consequently every $G_k(\chi,\psi)$ and every
$\tau_k(\chi,\psi)$ is nonzero.
\end{lemma}

\begin{proof}
The trivial-character formula is
$\sum_{x\in k^\times}\psi(x)=-1$, because the sum over all of $k$ is
zero.  Assume $\chi$ nontrivial and write $q=|k|$.  Since the values of
$\chi$ and $\psi$ have absolute value one,
\[
 \begin{aligned}
 |G_k(\chi,\psi)|^2
 &=\sum_{x,y\in k^\times}\chi(x/y)\psi(x-y)\\
 &=\sum_{t\in k^\times}\chi(t)
       \sum_{y\in k^\times}\psi((t-1)y).
 \end{aligned}
\]
The inner sum is $q-1$ when $t=1$ and is $-1$ when $t\ne1$.
Therefore
$|G_k(\chi,\psi)|^2 =(q-1)-\sum_{t\ne1}\chi(t)$.
Character orthogonality gives $\sum_{t\in k^\times}\chi(t)=0$, so
$\sum_{t\ne1}\chi(t)=-1$ and the last expression is $q$.
\end{proof}

\begin{theorem}[Hasse--Davenport lift]\label{thm:HD-lift}
For a finite extension $k'/k$ of degree $f$,
\begin{equation}\label{eq:HD-lift}
 \tau_{k'}(\chi\circ \Nm_{k'/k},\psi\circ\Tr_{k'/k})
 =\tau_k(\chi,\psi)^f.
\end{equation}
\end{theorem}

\begin{proof}
This is the classical lifting relation of Davenport and Hasse
\cite{DavenportHasse}.  We include the Euler-product proof in the
normalization used here.

For a monic polynomial
\[
 P(X)=X^r+a_1X^{r-1}+\cdots+a_r\in k[X],\qquad a_r\ne0,
\]
put
$w(P)=\chi((-1)^r a_r)\psi(-a_1)$.
If $P,Q$ are monic with nonzero constant terms, then the constant terms
multiply and the negative of the coefficient of degree
$\deg(PQ)-1$ is the sum of the roots; hence $w(PQ)=w(P)w(Q)$.  Therefore
\begin{align*}
 L(T)
 &=\prod_{\substack{P\ {\rm monic\ irreducible}\\ P\ne X}}
       (1-w(P)T^{\deg P})^{-1}\\
 &=\sum_{\substack{P\ {\rm monic}\\ P(0)\ne0}}w(P)T^{\deg P}.
\end{align*}
For degree $r\ge2$, the coefficient is zero, because $a_1$ ranges freely
through $k$ and $\sum_{a_1\in k}\psi(-a_1)=0$.  The coefficient in
degree $1$ is
$\sum_{a_1\ne0}\chi(-a_1)\psi(-a_1)=G_k(\chi,\psi)$.
Thus
\begin{equation*}
 L(T)=1+G_k(\chi,\psi)T.
\end{equation*}
On the other hand, taking logarithms of the Euler product and grouping
closed points after extension of scalars gives
\begin{equation*}
 \log L(T)=\sum_{r\ge1}\frac{T^r}{r}
 \sum_{x\in k_r^\times}
     \chi(\Nm_{k_r/k}x)\psi(\Tr_{k_r/k}x),
\end{equation*}
where $k_r/k$ is the degree-$r$ extension.  Indeed, an element of $k_r$
of degree $d\mid r$ contributes through its monic minimal polynomial
with multiplicity $d$, as in the usual Euler-product proof for
the zeta function of the affine line.  Comparing the displayed formulas for $L(T)$ and $\log L(T)$ yields
\begin{equation*}
 G_{k_r}(\chi\circ \Nm_{k_r/k},\psi\circ\Tr_{k_r/k})
 =(-1)^{r-1}G_k(\chi,\psi)^r.
\end{equation*}
Now apply this identity to $\chi^{-1}$.  Since
$\tau_k(\chi,\psi)=-G_k(\chi^{-1},\psi)$, the signs cancel:
\[
 \tau_{k_r}(\chi\circ \Nm_{k_r/k},\psi\circ\Tr_{k_r/k})
 =-(-1)^{r-1}G_k(\chi^{-1},\psi)^r
 =\tau_k(\chi,\psi)^r.
\]
Taking $r=f$ proves the theorem.
\end{proof}

\begin{theorem}[Hasse--Davenport multiplication formula]
\label{thm:HD-product}
Let $\ell\mid |k^\times|$, let
$\eta:k^\times\to\C^\times$ be a multiplicative character of order
$\ell$, and let $\chi:k^\times\to\C^\times$ be any multiplicative
character.  Then
\begin{equation}\label{eq:HD-product}
 \chi(\ell^\ell)\tau_k(\chi^\ell,\psi)
 \prod_{j=1}^{\ell-1}\tau_k(\eta^j,\psi)
 =\tau_k(\chi,\psi)
 \prod_{j=1}^{\ell-1}\tau_k(\chi\eta^j,\psi).
\end{equation}
\end{theorem}

\begin{proof}
For
\[
 G(\alpha)=\sum_{x\in k^\times}\alpha(x)\psi(x),
 \qquad G(1)=-1,
\]
the classical Davenport--Hasse multiplication formula
\cite{DavenportHasse} is
\[
 \prod_{j=0}^{\ell-1}G(\alpha\eta^j)
 =\alpha(\ell^{-\ell})G(\alpha^\ell)
   \prod_{j=1}^{\ell-1}G(\eta^j).
\]
Apply it with $\alpha=\chi^{-1}$.  Inversion permutes the nontrivial powers of
$\eta$, and $G(\alpha)=-\tau_k(\alpha^{-1},\psi)$.  Both sides contain
exactly $\ell$ ordinary Gauss sums when the factor
$G(\chi^{-\ell})$ is counted, so the signs cancel.  Finally
$\chi^{-1}(\ell^{-\ell})=\chi(\ell^\ell)$, which gives
\eqref{eq:HD-product} in Langlands's normalization.
\end{proof}

\begin{theorem}[Hasse-function lift]\label{thm:Hasse-lift}
Let $k'/k$ be finite and let a nowhere-zero function $H:k\to\C^\times$
satisfy
\[
 H(x+y)=H(x)H(y)\psi(xy).
\]
Put
\begin{equation}\label{eq:Hasse-lift-definition}
 \omega_{k'/k}(x)=\sum_{\{\sigma,\tau\}}x^\sigma x^\tau,
 \qquad
 H_{k'/k}(x)=
 H(\Tr x)\psi(-\omega_{k'/k}(x)).
\end{equation}
Here the sum defining $\omega_{k'/k}$ is over the unordered pairs of
distinct $k$-embeddings $k'\hookrightarrow\overline{k}$.  Then
\begin{equation}\label{eq:Hasse-lift-sum}
 -\sum_{x\in k'}H_{k'/k}(x)
 =\left(-\sum_{x\in k}H(x)\right)^{[k':k]}.
\end{equation}
\end{theorem}

\begin{proof}
For the sum identity, let
$P(X)=X^r+a_1X^{r-1}+a_2X^{r-2}+\cdots$
be monic and put
$w(P)=H(-a_1)\psi(-a_2)$,
where $a_2=0$ in degree $1$.  If $Q$ has first coefficients $b_1,b_2$,
then the first two coefficients of $PQ$ are
$a_1+b_1$ and $a_2+b_2+a_1b_1$.  Hence
\begin{align*}
 w(PQ)
 &=H(-a_1-b_1)\psi(-a_2-b_2-a_1b_1)\\
 &=H(-a_1)H(-b_1)\psi(a_1b_1)
   \psi(-a_2-b_2-a_1b_1)=w(P)w(Q).
\end{align*}
Thus
\[
 L(T)=\prod_{P\ {\rm monic\ irreducible}}
       (1-w(P)T^{\deg P})^{-1}
     =\sum_{P\ {\rm monic}}w(P)T^{\deg P}.
\]
For degree $r\ge2$, the coefficient is zero because $a_2$ ranges freely
and $\psi$ is nontrivial.  In degree $1$ it is $\sum_{x\in k}H(x)$.
Therefore
\begin{equation*}
 L(T)=1+\left(\sum_{x\in k}H(x)\right)T.
\end{equation*}
We next justify the closed-point expansion, including compatibility
with repeated closed points.  Let $P$ be monic irreducible of degree
$d$, choose a root $\xi\in k_d$, and put
$s=\Tr_{k_d/k}(\xi),\qquad e=\omega_{k_d/k}(\xi)$.
The first two coefficients of $P$ are $-s$ and $e$, so
$w(P)=H(s)\psi(-e)$.
If $r=hd$ and $\xi$ is viewed in $k_r$, each of its $d$ conjugates
occurs $h$ times.  Consequently
\begin{equation*}
 \Tr_{k_r/k}(\xi)=hs,
 \qquad
 \omega_{k_r/k}(\xi)=he+\binom h2s^2.
\end{equation*}
Iteration of the functional equation for $H$ gives
\begin{equation*}
 H(hs)=H(s)^h\psi\!\left(\binom h2s^2\right).
\end{equation*}
Combining these identities yields
\begin{equation*}
 H(\Tr_{k_r/k}\xi)\psi(-\omega_{k_r/k}(\xi))
 =\bigl(H(s)\psi(-e)\bigr)^h=w(P)^h.
\end{equation*}
By this identity, the $d$ roots belonging
to the closed point $P$ contribute
$d\,w(P)^h$ to the sum over $k_r$.  Expanding the logarithm of the
Euler product and regrouping by $r=h\deg P$ therefore gives
\begin{equation*}
 \log L(T)
 =\sum_{r\ge1}\frac{T^r}{r}
 \sum_{x\in k_r}H(\Tr_{k_r/k}x)
              \psi(-\omega_{k_r/k}(x)).
\end{equation*}
Comparing the logarithm of the linear formula for $L(T)$ with
this expansion gives
\[
 \sum_{x\in k_r}H_{k_r/k}(x)
 =(-1)^{r-1}\left(\sum_{x\in k}H(x)\right)^r.
\]
Multiplying by $-1$ gives \eqref{eq:Hasse-lift-sum}.
\end{proof}

This is \cite[Lemma 9.1, pp.~81--84]{Langlands}; see also the
finite-field lifting theorem in \cite[Chapter~II, pp.~31--33]{Dwork}.

\begin{lemma}[Quadratic normal form]\label{lem:quadratic-normal-form}
Suppose $k$ has odd characteristic. If $a\in k^\times$ and
$H:k\to\C^\times$ satisfies
$H(x+y)=H(x)H(y)\psi(axy)$, then there is a unique $b\in k$ such that
\[
 H(x)=\psi(ax^2/2+bx)\qquad(x\in k).
\]
\end{lemma}
\begin{proof}
The function $x\mapsto H(x)\psi(-ax^2/2)$ is an additive character.
The perfect pairing $(b,x)\mapsto\psi(bx)$ identifies it with
$x\mapsto\psi(bx)$ for a unique $b\in k$.
\end{proof}

\begin{lemma}[Quadratic phase]\label{lem:quadratic-phase-new}
Assume $\operatorname{char}k\neq2$.  Define the quadratic
character
$\nu:k^\times\to\{\pm1\}$ by
\[
 \nu(a)=
 \begin{cases}
  1,&a\in(k^\times)^2,\\
  -1,&a\notin(k^\times)^2,
 \end{cases}
\]
and extend it by $\nu(0)=0$.  Put
\[
 Q_\psi(a,b)=\ph\!\left(\sum_{x\in k}
     \psi\left(\frac a2x^2+bx\right)\right),\qquad a\neq0.
\]
Then
\begin{equation}\label{eq:quadratic-phase-new}
 Q_\psi(a,b)=\nu(a)\psi(-b^2/2a)Q_\psi(1,0),
 \qquad Q_\psi(1,0)^2=\nu(-1).
\end{equation}
The defining sum has absolute value $|k|^{1/2}$.
\end{lemma}

\begin{proof}
Let
\[
 S(a,b)=\sum_{x\in k}\psi\left(\frac a2x^2+bx\right).
\]
Completing the square gives
$S(a,b)=\psi(-b^2/2a)S(a,0)$.
For $c\ne0$, every $y\in k^\times$ has $1+\nu(y)$ square roots, while
zero has one square root.  Hence
\begin{equation}\label{eq:quadratic-gauss-rescaling}
 \sum_{x\in k}\psi(cx^2)
 =\sum_{y\in k^\times}\nu(y)\psi(cy)
 =\nu(c)G_k(\nu,\psi).
\end{equation}
The additive-character part disappeared because
$\sum_{y\in k}\psi(cy)=0$.  Taking $c=a/2$ in \eqref{eq:quadratic-gauss-rescaling} and comparing with $a=1$
proves the first phase identity.  Equation~\eqref{eq:finite-gauss-magnitude}
shows at once that $|S(a,b)|=|k|^{1/2}$.

It remains to compute the square of the phase. Put $G=G_k(\nu,\psi)$.
Conjugation and the substitution $y\mapsto-y$ give
\[
 \overline G=\sum_{y\in k^\times}\nu(y)\psi(-y)=\nu(-1)G.
\]
Lemma~\ref{lem:finite-gauss-magnitude} gives $|G|^2=|k|$, hence
$G^2=\nu(-1)|k|$. Since
$S(1,0)=\nu(1/2)G$ and $\nu(1/2)^2=1$, dividing by absolute values
gives $Q_\psi(1,0)^2=\nu(-1)$.
\end{proof}

\begin{proposition}[Ramified Hasse-phase comparison]
\label{thm:ramified-Hasse-comparison}
Let $K/F$ be totally ramified cyclic of odd prime degree $\ell$, with
single break $t$.  Let $\chi_F:F^\times\to\C^\times$ be a
quasi-character and let $\psi_F:F\to\C^\times$ be a nontrivial additive
character.  Put
\[
 \chi_K=\chi_F\circ \Nm_{K/F},
 \qquad
 \psi_K=\psi_F\circ\Tr_{K/F}.
\]
Suppose
\[
 m_F(\chi_F)=2d+1>1,
\]
and assume either
\[
 t=0\qquad\text{or}\qquad d\ge t+1.
\]
Let $\gamma\in F^\times$ be admissible for $(\chi_F,\psi_F)$, and let
$\beta\in\cO_F^\times$ represent the stationary class of $\chi_F$
relative to $\gamma$.  Then $\gamma$ is also admissible for
$(\chi_K,\psi_K)$, and $\beta$ represents its stationary class.
With this common denominator and representative, the residual factors
satisfy
\begin{equation}\label{eq:ramified-Hasse-comparison}
 \Delta_{3,K}(\chi_K)=\Delta_{3,F}(\chi_F)^\ell.
\end{equation}
\end{proposition}

\begin{proof}
Write $\Nm=\Nm_{K/F}$.
By Proposition~\ref{prop:pullback-conductor}(i),
\[
 m_K(\chi_K)=\ell(2d+1)-(\ell-1)(t+1)=2d'+1,
 \qquad
 d'=\ell d-\frac{\ell-1}{2}t.
\]
Put $n=n_F(\psi_F)$ and choose uniformizers with
$\varpi_F=\Nm(\varpi_K)$.  If $t>0$, choose a generator $\sigma$ of $\Gal(K/F)$
and define $\lambda$ by \eqref{eq:critical-ramification-coordinate}.
The different exponent is
$(\ell-1)(t+1)$, and therefore
\[
 n_K(\psi_K)=\ell n+(\ell-1)(t+1),
 \qquad
 (2d'+1)+n_K(\psi_K)=\ell(2d+1+n).
\]
Since $v_K(\gamma)=\ell(2d+1+n)$, the denominator $\gamma$ is also
admissible for $\chi_K$.
For $y\in\p_K^{d'+1}$, \eqref{eq:trace-ideal} gives
$\Tr_{K/F}(y)\in\p_F^{d+1}$, and the norm expansion gives
\[
 \Nm(1+y)\equiv1+\Tr_{K/F}(y)\pmod{\p_F^{2d+1}}.
\]
Indeed, if $t=0$, then $d'=\ell d$ and every degree-$j$ term with
$j\ge2$ has valuation at least $\lceil j(d'+1)/\ell\rceil\ge2d+1$.
If $t>0$, then $\ell=p$, and \eqref{eq:symmetric-bound} gives this
bound for $2\le j<p$, since
$2(d'+1)+(p-1)(t+1)=p(2d+1)+1$; the norm term has valuation at least
$d'+1\ge2d+1$ by \eqref{eq:symmetric-bound-last}.
The chosen representative $\beta$ therefore satisfies
\[
 \chi_K(1+y)=\chi_F(1+\Tr_{K/F}(y))
 =\psi_F(\beta\Tr_{K/F}(y)/\gamma)
 =\psi_K(\beta y/\gamma)
 \qquad(y\in\p_K^{d'+1}).
\]
Let $k$ be the common residue field and, for $x\in k$, choose a
lift $\widetilde x\in\cO_F\subset\cO_K$; use the same notation for
lifts of other elements of $k$.  Write $s_2(y)$ for the second
elementary symmetric function of the conjugates of $y$.  The critical function for $\chi_F$ is
\begin{equation}\label{eq:ramified-Hasse-F-function}
 H_F(x)=
 \psi_F\!\left(\frac{\beta\varpi_F^d\widetilde x}{\gamma}\right)
 \chi_F(1+\varpi_F^d\widetilde x)^{-1},
\end{equation}
and its polar character is
\[
 \psi_0(z)=
 \psi_F\!\left(\frac{\beta\varpi_F^{2d}\widetilde z}{\gamma}\right).
\]
Put $Y=\varpi_K^{d'}\widetilde x$.  We claim that the critical function for $\chi_K$ is
\begin{equation*}
 H_K(x)=
 \psi_F\!\left(\frac{\beta\Tr_{K/F}(Y)}{\gamma}\right)
 \chi_F\!\left(1+\Tr_{K/F}(Y)+s_2(Y)\right)^{-1}.
\end{equation*}
Indeed, it is enough to prove
\begin{equation*}
 \Nm(1+Y)
 \equiv1+\Tr_{K/F}(Y)+s_2(Y)
       \pmod{\p_F^{2d+1}}.
\end{equation*}
If $t=0$, then $d'=\ell d$.  Every term of degree at least three in
the norm has valuation at least $3d\ge2d+1$, and the norm term has
valuation at least $\ell d\ge2d+1$.  If $t>0$, then $\ell=p$ is the residue
characteristic and $d\ge t+1$.  For $3\le j\le p-1$,
Equation~\eqref{eq:symmetric-bound} gives
\[
 v_F(s_j(Y))\ge
 \left\lfloor\frac{j d'+(p-1)(t+1)}p\right\rfloor
 \ge2d+1,
\]
while
$v_F(\Nm_{K/F}(Y))\ge d'=pd-\frac{p-1}{2}t\ge2d+1$.
This proves the claimed norm congruence.
Theorem~\ref{thm:Lamprecht} gives the independence of the lift in
the displayed formula for $H_K$.  The same estimates also give
$\Tr(Y)\in\p_F^d, \qquad s_2(Y)\in\p_F^{2d}$,
and one extra power of $\p_F$ when $\widetilde x\in\p_K$.

Set
\begin{equation*}
 \epsilon=
 \frac{\Tr_{K/F}(\varpi_K^{2d'})}{\varpi_F^{2d}}\in\cO_F.
\end{equation*}
The preceding valuation estimates show that this quotient is integral.
We now show that it is a unit and determine its residue.  If $t=0$,
then
$c=\overline{\varpi_K^{d'}/\varpi_F^d}\in k^\times$
is fixed by $\Gal(K/F)$, because $d'=\ell d$.  Hence
$\bar\epsilon=\ell c^2$.
If $t>0$, equation~\eqref{eq:stable-critical-trace-residue} gives
$\bar\epsilon=-\lambda^{p-1}$. In particular, $\epsilon$ is a unit in both cases.

The functional equation \eqref{eq:Hasse-functional-equation} over $K$
gives directly
\begin{align*}
 \frac{H_K(x+y)}{H_K(x)H_K(y)}
 &=\psi_K\!\left(
   \frac{\beta\varpi_K^{2d'}\widetilde x\widetilde y}{\gamma}\right)\\
 &=\psi_0(\bar\epsilon xy),
 \qquad x,y\in k,
\end{align*}
with $\epsilon=\Tr_{K/F}(\varpi_K^{2d'})/\varpi_F^{2d}$.  Since $\bar\epsilon\ne0$,
this polar bicharacter is nondegenerate.

We compare its normalized sum with that of $H_F$.  First assume that
$k$ has odd characteristic.  The polar character of $H_F$ is
$(x,y)\mapsto\psi_0(xy)$, so for a unique $\alpha_0\in k$ one has
\[
 H_F(x)=\psi_0(x^2/2+\alpha_0x).
\]
Put
\[
 T_x=\overline{\Tr(Y)/\varpi_F^d},
 \qquad
 Q_x=\overline{\Tr(Y^2)/\varpi_F^{2d}}.
\]
Using $\Tr(Y)/\varpi_F^d$ itself as a lift of $T_x$ in
\eqref{eq:ramified-Hasse-F-function}, one has
\[
 H_F(T_x)=
 \psi_F\!\left(\frac{\beta\Tr(Y)}{\gamma}\right)
 \chi_F(1+\Tr(Y))^{-1}.
\]
Because $s_2(Y)\in\p_F^{2d}$,
\[
 \frac{1+\Tr(Y)}{1+\Tr(Y)+s_2(Y)}
 \equiv1-s_2(Y)\pmod{\p_F^{2d+1}}.
\]
Equation~\eqref{eq:Lamprecht-linearization}, applied to
$-s_2(Y)\in\p_F^{2d}\subseteq\p_F^{d+1}$, gives
\begin{equation*}
 H_K(x)=
 H_F(T_x)\,
 \psi_0\!\left(
  -\overline{\frac{s_2(Y)}{\varpi_F^{2d}}}\right).
\end{equation*}
Now
$2s_2(Y)=(\Tr Y)^2-\Tr(Y^2)$, so
\[
 -\overline{\frac{s_2(Y)}{\varpi_F^{2d}}}
 =\frac{Q_x-T_x^2}{2}.
\]
Substituting
$H_F(T_x)=\psi_0(T_x^2/2+\alpha_0T_x)$ in
$H_K(x)=H_F(T_x)\psi_0(-\overline{s_2(Y)/\varpi_F^{2d}})$ yields
\begin{equation*}
 H_K(x)=\psi_0(Q_x/2+\alpha_0T_x).
\end{equation*}
If $t=0$, substitute $T_x=\ell cx$ and $Q_x=\ell c^2x^2$ in
this expression for $H_K$ and replace $x$ by $c^{-1}x$:
\begin{equation*}
 H_K(c^{-1}x)=\psi_0(\ell x^2/2+\ell\alpha_0x).
\end{equation*}
Multiplication by $c^{-1}$ permutes $k$, so
$\sum_{x\in k}H_K(c^{-1}x)=\sum_{x\in k}H_K(x)$.
If $t>0$, the trace-ideal formula \eqref{eq:trace-ideal} gives $T_x=0$, while
$Q_x=\bar\epsilon x^2$, and hence
\begin{equation*}
 H_K(x)=\psi_0(\bar\epsilon x^2/2).
\end{equation*}
Let $\nu$ be the quadratic character of $k^\times$ and put
$Q=Q_{\psi_0}(1,0)$.  Suppose first that $t=0$.
Applying \eqref{eq:quadratic-phase-new} to the formula for
$H_K(c^{-1}x)$, the phase defining $\Delta_{3,F}(\chi_F)$ is
$Q\psi_0(-\alpha_0^2/2)$ and the phase defining
$\Delta_{3,K}(\chi_K)$ is
$\nu(\ell)Q\psi_0(-\ell\alpha_0^2/2)$.
Since $Q^2=\nu(-1)$, the required power identity is equivalent to
\begin{equation*}
 \nu(\ell)=\nu(-1)^{(\ell-1)/2}.
\end{equation*}
Let $|k|=p^f$.  If $f$ is even, both sides are $1$.  If $f$ is odd,
the order of $p$ modulo $\ell$ is odd because $\ell\mid p^f-1$;
therefore $(p/\ell)=1$.  Quadratic reciprocity now gives
the required sign identity.

Now suppose $t>0$.  Then $\ell=p$ and, by
\eqref{eq:stable-critical-trace-residue}, $\bar\epsilon=-\lambda^{p-1}$, so
$\nu(\bar\epsilon)=\nu(-1)$.  The formula for $H_K(x)$ in the case $t>0$ and $Q^2=\nu(-1)$ reduce the desired
identity to
$\nu(-1)^{(p-3)/2}=1$.
If $p\equiv1\pmod4$, then $\nu(-1)=1$; if
$p\equiv3\pmod4$, the exponent $(p-3)/2$ is even.  This proves the
comparison in odd residue characteristic.

It remains to treat residue characteristic two.  Then $t=0$, because
$\ell$ is odd.  With $c=\overline{\varpi_K^{d'}/\varpi_F^d}$,
the norm calculation gives
\begin{equation*}
 H_K(c^{-1}x)
 =H_F(\ell x)\psi_0\!\left(-\binom{\ell}{2}x^2\right)
 =H_F(x)^\ell.
\end{equation*}
The second equality follows by iterating
$H_F(x+y)=H_F(x)H_F(y)\psi_0(xy)$.  Put
$S=\sum_x H_F(x)$ and $H=\ph(S)$.  Choose $c_0\in k^\times$ so
that $\psi_0(x^2)=\psi_0(c_0x)$.  Then
\[
 \overline{H_F(x)}=H_F(x)\psi_0(c_0x),
 \qquad
 \overline S=H_F(c_0)^{-1}S,
 \qquad
 H^2=H_F(c_0).
\]
As $H_F(x)^4=1$, the identity for $H_K(c^{-1}x)$ shows that the desired power
identity is equivalent to
\[
 H_F(c_0)^{(\ell-1)/2}=1\quad(\ell\equiv1\bmod4),
 \qquad
 H_F(c_0)^{(\ell+1)/2}=1\quad(\ell\equiv3\bmod4).
\]
Write $\psi_0(x)=(-1)^{\Tr_{k/\F_2}(a x)}$.  One may take
$c_0=a^{2^{f-1}-1}$, where $f=[k:\F_2]$; hence
$H_F(c_0)^2=\psi_0(c_0^2)=(-1)^f$.
The two displayed powers are automatic for
$\ell\equiv1,7\pmod8$.  If $\ell\equiv3,5\pmod8$, then
$(2/\ell)=-1$; since $\ell\mid2^f-1$, the order of $2$ modulo
$\ell$ and therefore $f$ are even.  Thus $H_F(c_0)^2=1$, and the
remaining two residue classes also follow.

In every case the phase defining $\Delta_{3,K}(\chi_K)$ is the
$\ell$th power of the phase defining $\Delta_{3,F}(\chi_F)$, which is
\eqref{eq:ramified-Hasse-comparison}.
\end{proof}

\begin{proposition}[Homogeneous affine pencil of quadratic phases]
\label{lem:affine-pencil}\label{prop:affine-pencil-homogeneous}
Let $k$ be a finite field of odd characteristic $p$, let $\psi$ be a
nontrivial additive character of $k$, and let $c,\rho\in k^\times$
satisfy $1+j\rho\ne0$ for every $j\in\F_p$.
For $\sigma,\upsilon\in k$, one has
\[
\begin{aligned}
 &Q_\psi\!\left(c(1-\rho^{1-p}),c(\sigma-\upsilon/\rho)\right)
       \prod_{j\in\F_p^\times}Q_\psi(cj\rho,cj\upsilon)\\
 &\qquad=\prod_{j\in\F_p}
       Q_\psi\!\left(c(1+j\rho),c(\sigma+j\upsilon)\right).
\end{aligned}
\]
\end{proposition}
\begin{proof}
Apply \eqref{eq:quadratic-phase-new}. Each side contains $p$ factors
$Q_\psi(1,0)$. The quadratic-character factors agree because
\[
 (1-\rho^{1-p})\prod_{j\in\F_p^\times}j\rho
 =1-\rho^{p-1}=\prod_{j\in\F_p}(1+j\rho).
\]
The additive factors agree by
\[
 \sum_{j\in\F_p}\frac{(\sigma+j\upsilon)^2}{1+j\rho}
 =\frac{(\sigma-\upsilon/\rho)^2}{1-\rho^{1-p}},
 \qquad
 \sum_{j\in\F_p^\times}\frac{(j\upsilon)^2}{j\rho}=0.
\]
For the first identity, write
$\sigma+j\upsilon=(\upsilon/\rho)(1+j\rho)
 +(\sigma-\upsilon/\rho)$.
The polynomial terms sum to zero, and logarithmic differentiation of
$\prod_j(X+j)=X^p-X$ gives
$\sum_j(1+j\rho)^{-1}=\rho^{p-1}/(\rho^{p-1}-1)$.
The second identity follows from $\sum_{j\in\F_p^\times}j=0$.
\end{proof}

\begin{lemma}[Artin--Schreier phase]\label{lem:AS-phase-new}
Let $k$ have characteristic $2$ and let
$\psi_0(x)=(-1)^{\Tr_{k/\F_2}(x)}$.  Then
\begin{equation}\label{eq:AS-phase}
 \psi_0(c+c^2)=1\qquad(c\in k).
\end{equation}
\end{lemma}

\begin{proof}
The absolute trace is Frobenius-invariant, so
$\Tr(c+c^2)=0$.
\end{proof}

For the following result, let $k$ have characteristic two and put
$\psi_0(z)=(-1)^{\Tr_{k/\F_2}(z)}$.

\begin{proposition}[Normalized refinement]
\label{lem:normalized-refinement}
Let $h:k\to\C^\times$ be unimodular and satisfy
$h(r+s)=h(r)h(s)\psi_0(rs)$.
Multiplying $h$ by an additive character gives a unimodular function
$\mathfrak q:k\to\C^\times$ such that
\begin{equation}\label{eq:q-refinement}
 \mathfrak q(r+s)=\mathfrak q(r)\mathfrak q(s)\psi_0(rs),
 \qquad
 \mathfrak q(r^2)=\mathfrak q(r).
\end{equation}
Every other function with the same polar character is uniquely
\begin{equation}\label{eq:q-linear-family}
 h_\lambda(z)=\mathfrak q(z)\psi_0(\lambda z)
 \qquad(\lambda\in k).
\end{equation}
Write
\begin{equation}\label{eq:quadratic-normalized-sum}
 H(\lambda)=\ph\!\left(\sum_{z\in k}h_\lambda(z)\right),
 \qquad H_0=H(0).
\end{equation}
We have
\begin{equation}\label{eq:q-phase-translation}
\begin{gathered}
 H(\lambda)=H_0\overline{\mathfrak q(\lambda)},\qquad
 H_0^2=\mathfrak q(1),\qquad \mathfrak q(r)^2=\psi_0(r),\\
 H(\lambda^2)=H(\lambda),\qquad
 H(\lambda)^2=h_\lambda(1),\qquad
 H(1+\lambda)H(\lambda)=1.
\end{gathered}
\end{equation}
\end{proposition}

\begin{proof}
The quotient
$\ell(z)=\frac{h(z^2)}{h(z)}$
is additive.  Indeed, using the defining functional equation for $h$ and Frobenius invariance
of $\psi_0$,
\[
 \ell(r+s)
 =\frac{h((r+s)^2)}{h(r+s)}
 =\frac{h(r^2)}{h(r)}\frac{h(s^2)}{h(s)}
 =\ell(r)\ell(s).
\]
The trace pairing on $k$ is perfect, so there is a unique $c\in k$
such that $\ell(z)=\psi_0(cz)$.  Let $f=[k:\F_2]$.  For every
$z\in k$, telescoping around the Frobenius orbit gives
\begin{align*}
 1
 &=\frac{h(z^{2^f})}{h(z)}
   =\prod_{i=0}^{f-1}\frac{h(z^{2^{i+1}})}{h(z^{2^i})}\\
 &=\psi_0\!\left(c\sum_{i=0}^{f-1}z^{2^i}\right)
   =\psi_0\!\left(c\Tr_{k/\F_2}(z)\right).
\end{align*}
Choose $z$ of absolute trace $1$.  Then $\psi_0(c)=1$, or equivalently
$\Tr_{k/\F_2}(c)=0$.

Consider the $\F_2$-linear map
$L:k\longrightarrow k, \qquad L(d)=d+d^{2^{f-1}}$.
Writing $\operatorname{Fr}(z)=z^2$, one has
$L=(\operatorname{Fr}-1)\circ\operatorname{Fr}^{-1}$.  Lemma~\ref{lem:frobenius-difference-trace}
therefore gives $\ker L=\F_2$ and
$\operatorname{im}L=\ker\Tr_{k/\F_2}$.  Since $\Tr_{k/\F_2}(c)=0$,
there is $d\in k$ such that
$c=d+d^{2^{f-1}}$.
Replace $h(z)$ by
$\mathfrak q(z)=h(z)\psi_0(dz)$.
The polar character is unchanged.  Moreover
\begin{align*}
 \frac{\mathfrak q(z^2)}{\mathfrak q(z)}
 &=\psi_0(cz)\psi_0(dz^2+dz)\\
 &=\psi_0\!\left((c+d^{2^{f-1}}+d)z\right)=1,
\end{align*}
where
$\psi_0(dz^2)=\psi_0(d^{2^{f-1}}z)$ and
the displayed choice of $d$ were used.  This proves
$\mathfrak q(z^2)=\mathfrak q(z)$.  If $h'$ has the same polar
character as $\mathfrak q$, then $h'/\mathfrak q$ is an additive
character of $k$, hence is uniquely $z\mapsto\psi_0(\lambda z)$.
This proves \eqref{eq:q-linear-family}.

From \eqref{eq:q-refinement},
\[
 \mathfrak q(z)\psi_0(\lambda z)
 =\mathfrak q(z+\lambda)\mathfrak q(\lambda)^{-1},
\]
so substituting $z\mapsto z+\lambda$ in $\sum_{z\in k}h_\lambda(z)$
gives the first formula in
\eqref{eq:q-phase-translation}.  Setting $r=s$ in
\eqref{eq:q-refinement} gives
$\mathfrak q(r)^2=\psi_0(r^2)=\psi_0(r)$,
and therefore
$\overline{\mathfrak q(r)}=\mathfrak q(r)\psi_0(r)$.

Put $S_0=\sum_{r\in k}\mathfrak q(r)$.  Its polar bicharacter is
$(r,s)\mapsto\psi_0(rs)$, which is nondegenerate by the perfect trace
pairing.  Equation~\eqref{eq:finite-quadratic-magnitude} therefore gives
$|S_0|^2=|k|$,
so $S_0\ne0$.  Finally, translation by $1$ gives
\[
 \overline{S_0}
 =\sum_r\mathfrak q(r)\psi_0(r)
 =\frac1{\mathfrak q(1)}\sum_r\mathfrak q(r+1)
 =\frac{S_0}{\mathfrak q(1)}.
\]
Taking phases yields
$H_0^{-1}=H_0/\mathfrak q(1)$, and hence
$H_0^2=\mathfrak q(1)$.
Since $\mathfrak q(\lambda^2)=\mathfrak q(\lambda)$, the first formula in
\eqref{eq:q-phase-translation} gives $H(\lambda^2)=H(\lambda)$.
Squaring it gives
$H(\lambda)^2=\mathfrak q(1)\psi_0(\lambda)=h_\lambda(1)$.
Finally, $\mathfrak q(1+\lambda)\mathfrak q(\lambda)=\mathfrak q(1)$
by \eqref{eq:q-refinement}, so
\[
 H(1+\lambda)H(\lambda)
 =H_0^2\overline{\mathfrak q(1+\lambda)\mathfrak q(\lambda)}=1.
\]
\end{proof}

\section{Stationary parameters under norm}\label{sec:stationary-parameters}

Let $K/F$ be a cyclic extension of nonarchimedean local fields of prime
degree $\ell$.  Let
$\chi_F:F^\times\to\C^\times$ be a quasi-character and let
$\psi_F:F\to\C^\times$ be a nontrivial additive character.  Put
\[
 \chi_K=\chi_F\circ \Nm_{K/F},
 \qquad
 \psi_K=\psi_F\circ\Tr_{K/F}.
\]
If $K/F$ is ramified, let $t$ be its unique
lower ramification break.  Write
\[
 m=m_F(\chi_F)=2d+\varepsilon>1,
 \qquad
 m'=m_K(\chi_K)=2d'+\varepsilon',
 \qquad \varepsilon,\varepsilon'\in\{0,1\}.
\]
The preceding conductor results give the following three cases for $m'$.
\begin{enumerate}[label=(\roman*)]
\item If $K/F$ is unramified, then $m'=m$ by
Lemma~\ref{lem:unramified-compatibility}(d).
\item If $K/F$ is ramified and $m>t+1$, then
\[
 m'=\HerbPsi_{K/F}(m-1)+1
\]
by Proposition~\ref{prop:pullback-conductor}(i).
\item If $K/F$ is ramified and $m\le t+1$, retain the minimal-conductor choice of $\chi_F$.  Then $m'=m$
by \eqref{eq:minimal-critical}.
\end{enumerate}
\begin{lemma}[Norm filtration at stationary depth]
\label{lem:stationary-depth-norm}
In each of the three cases above,
\begin{align}
 \Nm_{K/F}(U_K^{d'+\varepsilon'})&\subseteq U_F^{d+\varepsilon},
                                                        \label{eq:stationary-depth-first-inclusion}\\
 \Nm_{K/F}(U_K^{m'})&\subseteq U_F^m,             \label{eq:stationary-depth-second-inclusion}\\
 \Nm_{K/F}(U_K^{d'})&\subseteq U_F^d.             \label{eq:stationary-parameter-ambiguity}
\end{align}
The third
inclusion implies that if $b,b'\in\cO_K^\times$ and
$b'\equiv b\pmod{\p_K^{d'}}$, then
\begin{equation}\label{eq:norm-of-stationary-class-well-defined}
 \Nm_{K/F}(b')\equiv \Nm_{K/F}(b)\pmod{\p_F^d}.
\end{equation}
\end{lemma}

\begin{proof}
Write $\Nm=\Nm_{K/F}$.
Put
\[
 q=\left\lceil\frac m2\right\rceil=d+\varepsilon,
 \qquad
 q'=\left\lceil\frac {m'}2\right\rceil=d'+\varepsilon'.
\]
In case (i), Lemma~\ref{lem:unramified-compatibility}(c) gives all
three inclusions.

Now consider case (ii), so $K/F$ is ramified and $m>t+1$.  We first
prove \eqref{eq:stationary-depth-first-inclusion}.  Set
$r=m-1>t$ and $s=\lfloor r/2\rfloor=q-1$.  Then
\[
 m'=\HerbPsi_{K/F}(r)+1,
 \qquad
 q'-1=\left\lfloor\frac{\HerbPsi_{K/F}(r)}2\right\rfloor.
\]
We claim that
\begin{equation*}
 \left\lfloor\frac{\HerbPsi_{K/F}(r)}2\right\rfloor
 \ge \HerbPsi_{K/F}(s).
\end{equation*}
If $s\le t$, then the right side is $s$, whereas
$\HerbPsi_{K/F}(r)\ge r\ge2s$.  If $s>t$, write
$r=2s+e$ with $e\in\{0,1\}$.  Equation~\eqref{eq:Herbrand-prime-new} gives
\[
 \HerbPsi_{K/F}(r)-2\HerbPsi_{K/F}(s)
   =(\ell-1)t+\ell e\ge0,
\]
which again proves the claimed inequality.  Hence
$q'\ge\HerbPsi_{K/F}(q-1)+1$.

If $q-1<t$, norm preserves the filtration below the break, so
$\Nm(U_K^q)\subseteq U_F^q$.  If $q-1=t$, the critical norm map gives
$\Nm(U_K^{t+1})\subseteq U_F^{t+1}$.  If $q-1>t$,
\eqref{eq:norm-above-break}, with $r=q-1$, gives
$\Nm(U_K^{\HerbPsi_{K/F}(q-1)+1})=U_F^q$.  Together with
the bound for $q'$ above and the decreasing unit
filtration, these three cases prove
\eqref{eq:stationary-depth-first-inclusion}.

For \eqref{eq:stationary-depth-second-inclusion},
\eqref{eq:norm-above-break}, applied with $r=m-1$, gives
$\Nm(U_K^{m'})=U_F^m$.

We next prove \eqref{eq:stationary-parameter-ambiguity}.  Since
$\HerbPsi_{K/F}(u)\ge u$, one has $m'\ge m$ and hence $d'\ge d$.  If
$d\le t$, norm preserves the filtration at depth $d$, so
$\Nm(U_K^{d'})\subseteq \Nm(U_K^d)\subseteq U_F^d$.
If $d=t+1$, the same conclusion follows from the critical inclusion
$\Nm(U_K^{t+1})\subseteq U_F^{t+1}$.  Finally suppose $d\ge t+2$ and put
$R=\HerbPsi_{K/F}(d-1)+1$.
Using \eqref{eq:Herbrand-prime-new} and $m=2d+\varepsilon$, one obtains
\begin{align*}
 m'-2R
 &=\HerbPsi_{K/F}(2d+\varepsilon-1)+1
     -2\bigl(\HerbPsi_{K/F}(d-1)+1\bigr)\\
 &=(\ell-1)(t+1)+\ell\varepsilon\ge0.
\end{align*}
Thus $d'=\lfloor m'/2\rfloor\ge R$.  Equation
\eqref{eq:norm-above-break}, applied with $r=d-1$, now gives
$\Nm(U_K^{d'})\subseteq \Nm(U_K^R)=U_F^d$.
This proves all three inclusions in case (ii).

Finally consider case (iii).  Here $m'=m\le t+1$, so $q'=q$ and
$d'=d\le t$.  Below the break, norm maps $U_K^i$ into $U_F^i$ for
every $i\le t$; at the last possible step $m=t+1$ one uses the
critical inclusion $\Nm(U_K^{t+1})\subseteq U_F^{t+1}$.  Applying these
statements at $i=q$, $i=m$, and $i=d$ proves all three inclusions.

For the final assertion, write $b'=b(1+z)$ with
$z\in\p_K^{d'}$.  Then $1+z\in U_K^{d'}$, so
\eqref{eq:stationary-parameter-ambiguity} gives
$\Nm(b')/\Nm(b)\in U_F^d$, which is equivalent to
\eqref{eq:norm-of-stationary-class-well-defined}.
\end{proof}

For $y\in\p_K^{d'+\varepsilon'}$, let $\bar y$ denote its class
modulo $\p_K^{m'}$.

\begin{proposition}[Truncated norm homomorphism and adjoint]
\label{lem:truncated-norm-adjoint}
In each of cases (i)--(iii), the assignment
\[
 P_{K/F}(\bar y)=\overline{\Nm_{K/F}(1+y)-1}
\]
is a well-defined additive homomorphism
\[
 P_{K/F}:\p_K^{d'+\varepsilon'}/\p_K^{m'}
 \longrightarrow\p_F^{d+\varepsilon}/\p_F^m.
\]
Let $\gamma_F$ and $\gamma_K$ be admissible denominators for
$(\chi_F,\psi_F)$ and $(\chi_K,\psi_K)$, respectively.  The pairings
\[
 (\bar x,\bar y)\longmapsto\psi_F(xy/\gamma_F),
 \qquad
 (\bar x',\bar y')\longmapsto\psi_K(x'y'/\gamma_K)
\]
between $\cO_F/\p_F^d$ and $\p_F^{d+\varepsilon}/\p_F^m$, and between
$\cO_K/\p_K^{d'}$ and $\p_K^{d'+\varepsilon'}/\p_K^{m'}$, are perfect.
There is a unique additive adjoint
\[
 P^*_{K/F;\gamma_F,\gamma_K}:
 \cO_F/\p_F^d\longrightarrow\cO_K/\p_K^{d'}
\]
with the following property: for $x\in\cO_F$, an element $z\in\cO_K$
represents $P^*_{K/F;\gamma_F,\gamma_K}(\bar x)$ if and only if
\begin{equation}\label{eq:Pstar-def}
 \psi_F\!\left(\frac{x(\Nm_{K/F}(1+y)-1)}{\gamma_F}\right)
 =\psi_K(zy/\gamma_K)
 \qquad(y\in\p_K^{d'+\varepsilon'}).
\end{equation}
If
\[
 \Nm_{K/F}(1+y)-1\equiv\Tr_{K/F}(y)\pmod{\p_F^m}
 \qquad(y\in\p_K^{d'+\varepsilon'})
\]
and $\gamma_K/\gamma_F\in\cO_K^\times$, then
\[
 P^*_{K/F;\gamma_F,\gamma_K}(\bar x)
 =\overline{(\gamma_K/\gamma_F)x}
 \qquad(x\in\cO_F).
\]
Under the same trace-congruence hypothesis, if $\gamma_K=\gamma_F$,
this is the map $\cO_F/\p_F^d\to\cO_K/\p_K^{d'}$ induced by
$F\hookrightarrow K$.
\end{proposition}

\begin{proof}
Write $\Nm=\Nm_{K/F}$ and put
\[
 q=d+\varepsilon,\qquad q'=d'+\varepsilon',
 \qquad f(y)=\Nm(1+y)-1\quad(y\in\p_K^{q'}).
\]
Equation~\eqref{eq:stationary-depth-first-inclusion} gives
$f(y)\in\p_F^q$.  If $y'-y\in\p_K^{m'}$, then
$(1+y')/(1+y)\in U_K^{m'}$, so
\eqref{eq:stationary-depth-second-inclusion} gives
$f(y')\equiv f(y)\pmod{\p_F^m}$.  Thus $f$ induces the stated map.

For $y,z\in\p_K^{q'}$, one has $yz\in\p_K^{m'}$ because $2q'\ge m'$.
Multiplicativity of the norm gives the identity
$1+f(y+z+yz)=(1+f(y))(1+f(z))$.
The preceding congruence permits replacing $f(y+z+yz)$ by $f(y+z)$
modulo $\p_F^m$.  Also $f(y)f(z)\in\p_F^{2q}\subseteq\p_F^m$.
Hence $f(y+z)\equiv f(y)+f(z)\pmod{\p_F^m}$, proving additivity of
$P_{K/F}$.

Lemma~\ref{lem:finite-duality} gives perfectness of both pairings.
For $x\in\cO_F$, the left side of \eqref{eq:Pstar-def} defines a
character of $\p_K^{q'}/\p_K^{m'}$, depending only on
$\bar x\in\cO_F/\p_F^d$.  Perfectness over $K$ gives a unique class
$\bar z\in\cO_K/\p_K^{d'}$ representing this character.  Character
multiplication and uniqueness make $\bar x\mapsto\bar z$ additive.

Under the stated trace congruence, the left side of
\eqref{eq:Pstar-def} equals
\[
 \psi_F\!\left(\frac{x\Tr_{K/F}(y)}{\gamma_F}\right)
 =\psi_K(xy/\gamma_F)
 =\psi_K\!\left(\frac{((\gamma_K/\gamma_F)x)y}{\gamma_K}\right).
\]
Since $(\gamma_K/\gamma_F)x\in\cO_K$, uniqueness gives
$P^*_{K/F;\gamma_F,\gamma_K}(\bar x)=\overline{(\gamma_K/\gamma_F)x}$
and its common-denominator specialization.
\end{proof}

The truncated norm and its adjoint are introduced in
\cite[pp.~62--64]{Langlands}.

\begin{corollary}[Stationary class under norm pullback]
\label{cor:stationary-class-under-norm}
Under the hypotheses and with the denominators of
Proposition~\ref{lem:truncated-norm-adjoint}, let
\[
 \beta_F=\operatorname{St}_{\gamma_F,d+\varepsilon}(\chi_F)
       \in\cO_F/\p_F^d,
 \qquad
 \beta_K=\operatorname{St}_{\gamma_K,d'+\varepsilon'}(\chi_K)
       \in\cO_K/\p_K^{d'}.
\]
Then
\begin{equation}\label{eq:stationary-class-adjoint}
 \beta_K=P^*_{K/F;\gamma_F,\gamma_K}(\beta_F).
\end{equation}
\end{corollary}

\begin{proof}
Let $b_F\in\cO_F^\times$ represent $\beta_F$, and let $z\in\cO_K$
represent $P^*_{K/F;\gamma_F,\gamma_K}(\beta_F)$.  For
$y\in\p_K^{d'+\varepsilon'}$, Lemma~\ref{lem:stationary-depth-norm}
and \eqref{eq:Pstar-def} give
\begin{align*}
 \chi_K(1+y)
 &=\chi_F(\Nm_{K/F}(1+y))\\
 &=\psi_F\!\left(\frac{b_F(\Nm_{K/F}(1+y)-1)}{\gamma_F}\right)\\
 &=\psi_K(zy/\gamma_K).
\end{align*}
Uniqueness in Lemma~\ref{lem:stationary-class-existence}, applied over
$K$, gives $\bar z=\beta_K$.
\end{proof}

\begin{lemma}[Integer powers and Teichm\"uller representatives]
\label{lem:teichmueller-stationary}
Let $K/F$ be wildly ramified cyclic of prime degree
$p=\operatorname{char}k_F$, with break $t$, and put
$T=t+1$ and $s_0=\lceil T/2\rceil$.  Let $\tau$ be a character of
$F^\times$ of conductor $T$.  For $j\in\F_p$, let $j_{\mathbf Z}$
be its representative in $\{0,\ldots,p-1\}$, viewed in $F$, and let
$[j]\in\cO_F$ be its Teichm\"uller lift.  If $A\in F$ satisfies
\[
 \tau(1+x)=\psi_F(Ax)\qquad(x\in\p_F^{s_0}),
\]
then
\begin{equation*}
 \psi_F(j_{\mathbf Z}Ax)=\psi_F([j]Ax)
 \qquad(x\in\p_F^{s_0}).
\end{equation*}
Thus $[j]A$ and $j_{\mathbf Z}A$ give the same linearization of
$\tau^j$ on $\p_F^{s_0}$.  If $p$ is odd, then
\begin{equation}\label{eq:teichmueller-successor-depth}
 [j]-j_{\mathbf Z}\in\p_F^{s_0}.
\end{equation}
The Teichm\"uller lifts satisfy
\begin{equation}\label{eq:teichmueller-polynomials}
 \prod_{j\in\F_p}(X-[j])=X^p-X,
 \qquad
 \prod_{j\in\F_p}(X+[j])=X^p+(-1)^pX,
\end{equation}
and
\begin{equation}\label{eq:teichmueller-sum}
 \sum_{j\in\F_p}[j]=
 \begin{cases}0,&p\text{ odd},\\1,&p=2.\end{cases}
\end{equation}
\end{lemma}

\begin{proof}
In equal characteristic, the prime-field element $j_{\mathbf Z}$ is
already its Teichm\"uller lift.  Suppose that $F$ has characteristic zero
and put $e=v_F(p)$.  Since $p=\Tr_{K/F}(1)$, the trace-ideal formula
\eqref{eq:trace-ideal} with $a=0$ gives
\[
 e\ge \left\lfloor\frac{(p-1)T}{p}\right\rfloor
   \ge \left\lfloor\frac T2\right\rfloor.
\]
Moreover $[j]-j_{\mathbf Z}\in p\cO_F$.  Hence for
$x\in\p_F^{s_0}$,
\[
 ([j]-j_{\mathbf Z})x\in
 \p_F^{e+s_0}\subseteq\p_F^T.
\]
For $x\in\p_F^T$, the assumed linearization gives
$\psi_F(Ax)=\tau(1+x)=1$.  This proves
the asserted equality.

For even $T$, the bound on $e$ proves
\eqref{eq:teichmueller-successor-depth}.  If $p$ is odd and
$T=2r+1>1$, then $r\ge1$ and
$p(r+1)\le(p-1)(2r+1)$.
Hence
\[
 e\ge\left\lfloor\frac{(p-1)T}{p}\right\rfloor
 \ge r+1=\frac{T+1}{2}.
\]
Since $[j]-j_{\mathbf Z}\in p\cO_F$, this proves
\eqref{eq:teichmueller-successor-depth}.

The elements $[j]$ are the $p$ roots of $X^p-X$ in $F$.  This proves the first polynomial
identity; replacing $X$ by $-X$ gives the second.  The coefficient of
$X^{p-1}$ gives the sum formula.
\end{proof}

\begin{remark}[Notation for powers of the norm character]
In the wild cases below, $\tau$ denotes a generator of $S(K/F)$.
An exponent $j$ in $\tau^j$ denotes the integer
representative of a class in $\F_p$.  In formulas in $F$ or $K$, the
same symbol denotes its Teichm\"uller lift; in residue-field formulas it
denotes its class in $\F_p$.
\end{remark}

\begin{lemma}[Norm representatives at subcritical precision]
\label{lem:norm-stationary-representatives}
Let $K/F$ be ramified cyclic of prime degree with break $t$.  Let
$h\in\mathbf Z$, let $0\le s\le t$, and let $c\in F^\times$ have
$v_F(c)=h$.  There is $\widetilde{c}\in K^\times$ with $v_K(\widetilde{c})=h$ such that
\begin{equation*}
 c-\Nm_{K/F}(\widetilde{c})\in\p_F^{h+s}.
\end{equation*}
\end{lemma}

\begin{proof}
The first isomorphism in \eqref{eq:norm-below-break} supplies $\widetilde{c}$ with
$c/\Nm_{K/F}(\widetilde{c})\in U_F^t$.  Total ramification gives
$v_K(\widetilde{c})=v_F(\Nm_{K/F}(\widetilde{c}))=h$.  Since $c$ and $\Nm_{K/F}(\widetilde{c})$ have the same
valuation and their quotient lies in $U_F^t$, their difference belongs
to $\p_F^{h+t}\subseteq\p_F^{h+s}$.
\end{proof}

\begin{lemma}[Norm-character conversion]\label{lem:norm-character-conversion}
Suppose $K/F$ is wildly ramified cyclic of prime degree $p$, put $T=t+1$,
and let $\tau$ generate $S(K/F)$.  Put $s_0=\lceil T/2\rceil$, and
let $\alpha\in F$ and $\gamma_F\in F^\times$ satisfy
\begin{equation}\label{eq:tau-high-linearization}
 \tau(1+x)=\psi_F(\alpha x/\gamma_F)
 \qquad(x\in\p_F^{s_0}).
\end{equation}
Then, for $y\in\p_K^{s_0}$,
\[
 \psi_F\!\left(\frac{\alpha \Nm_{K/F}(y)}{\gamma_F}\right)
 =\psi_F\!\left(-\frac{\alpha\Tr_{K/F}(y)}{\gamma_F}\right).
\]
\end{lemma}

\begin{proof}
Put $D=D_{K/F}=(p-1)T$.  Equation~\eqref{eq:wild-norm-truncation},
with $q=s_0$ and $r=T$, gives
\[
 \Nm_{K/F}(1+y)\equiv1+\Tr_{K/F}(y)+\Nm_{K/F}(y)\pmod{\p_F^T},
\]
because $2s_0\ge T$.  The trace-ideal formula \eqref{eq:trace-ideal}
and total ramification give, respectively,
\[
 v_F(\Tr_{K/F}(y))\ge
 \left\lfloor\frac{s_0+D}{p}\right\rfloor\ge s_0,
 \qquad v_F(\Nm_{K/F}(y))=v_K(y)\ge s_0.
\]
Since $\tau$ is trivial on norms and on $U_F^T$, applying
\eqref{eq:tau-high-linearization} to the displayed norm congruence
gives the asserted identity.
\end{proof}

\subsection{High-conductor parameters}

\begin{proposition}[High-conductor stationary-class table]
\label{prop:high-parameter-table}
Let $\gamma_F$ be admissible for $(\chi_F,\psi_F)$.  Whenever
$\beta$ occurs below, it denotes a representative in $\cO_F^\times$ of
$\operatorname{St}_{\gamma_F,d+\varepsilon}(\chi_F)$.  In each of the cases (a)--(d) below,
$\gamma_K=\gamma_F$, viewed as an element of $K$, is admissible for
$(\chi_K,\psi_K)$.  Set
$P^*_{K/F}=P^*_{K/F;\gamma_F,\gamma_K}$.
\begin{enumerate}[label=(\alph*)]
\item If $K/F$ is unramified, then
$P^*_{K/F}(\bar\beta)=\bar\beta$ in $\cO_K/\p_K^{d'}$.

\item If $K/F$ is tamely ramified quadratic, then
$P^*_{K/F}(\bar\beta)=\bar\beta$ in $\cO_K/\p_K^{d'}$.

\item If $K/F$ is ramified of odd prime degree, put $T=t+1$.
If $d\ge T$, equivalently $m=2d+\varepsilon\ge2T$, then
$P^*_{K/F}(\bar\beta)=\bar\beta$ in $\cO_K/\p_K^{d'}$.

\item Suppose $K/F$ is wildly ramified of odd prime degree $p$. Retain the minimal-conductor choice of $\chi_F$.
 In the range $T\le m\le2T-1$,
let $\tau$ generate $S(K/F)$.  Choose $\widetilde{\alpha}\in K^\times$ with $v_K(\widetilde{\alpha})=m-T$ such that
$\alpha=\Nm_{K/F}(\widetilde{\alpha})$
satisfies \eqref{eq:tau-high-linearization} on
$\p_F^{\lceil T/2\rceil}$.
Choose $\widetilde{\beta}\in\cO_K^\times$ such that
$\beta=\Nm_{K/F}(\widetilde{\beta})$
represents $\operatorname{St}_{\gamma_F,d+\varepsilon}(\chi_F)$.
For these choices, let $\beta_K\in\cO_K^\times$ and, for
$j\in\F_p$, $\beta_j\in\cO_F^\times$ represent the stationary classes of
$\chi_K$ and $\tau^j\chi_F$ relative to $\gamma_K$ and $\gamma_F$,
respectively.  Then
\begin{equation}\label{eq:Pstar-intermediate}
 \beta_K\equiv\beta-\frac{\widetilde{\beta}\alpha}{\widetilde{\alpha}}
 \pmod{\p_K^{d'}}.
\end{equation}
For $j\in\F_p$, with its Teichm\"uller representative used in $F$ and $K$,
\begin{equation}\label{eq:twist-beta-intermediate}
 \beta_j\equiv \Nm_{K/F}(\widetilde{\beta}+j\widetilde{\alpha})
 \equiv\beta+j\alpha\pmod{\p_F^d}.
\end{equation}
Moreover,
\[
 \Nm_{K/F}(\beta_K)\equiv
 \prod_{j\in\F_p}(\beta+j\alpha)
 =\beta^p-\beta\alpha^{p-1}\pmod{\p_F^d}.
\]
\end{enumerate}
\end{proposition}

\begin{proof}
Write $\Nm=\Nm_{K/F}$ and put $D=D_{K/F}$.
We first verify admissibility of $\gamma_K=\gamma_F$.
In the unramified case,
$m'=m$ and $n_K(\psi_K)=n_F(\psi_F)$, so an admissible
$\gamma_F$ remains admissible in $K$.  In each ramified case, Proposition~\ref{prop:pullback-conductor}(i),
with \eqref{eq:minimal-critical} at $m=t+1$, and
\eqref{eq:additive-conductor-trace} give
\[
 m'=\ell m-D,
 \qquad n_K(\psi_K)=\ell n_F(\psi_F)+D,
\]
and therefore
\[
 m'+n_K(\psi_K)=\ell\bigl(m+n_F(\psi_F)\bigr)
 =v_K(\gamma_F).
\]
Thus $\gamma_K=\gamma_F$ is admissible for $\chi_K$ in all four parts.

For (a), $d'=d$ and
$\varepsilon'=\varepsilon$.  If
$y\in\p_K^{d+\varepsilon}$, the $j$th elementary symmetric term in
$\Nm_{K/F}(1+y)$ has valuation at least $j(d+\varepsilon)$; hence every term
with $j\geq2$ lies in $\p_F^m$, and
\begin{equation*}
 \Nm_{K/F}(1+y)-1\equiv\Tr_{K/F}(y)\pmod{\p_F^m}.
\end{equation*}
Proposition~\ref{lem:truncated-norm-adjoint} therefore identifies the
adjoint with the map induced by $F\hookrightarrow K$ on the quotient
rings. Applying it to $\bar\beta$ proves (a).

For (b), $m'=2m-1$, so $d'+\varepsilon'=m$.  For $x\in\p_K^m$,
the quadratic norm formula and $v_F(\Nm(x))=v_K(x)\ge m$ give
$\Nm(1+x)-1\equiv\Tr_{K/F}(x)\pmod{\p_F^m}$.
Proposition~\ref{lem:truncated-norm-adjoint} again gives
$P^*_{K/F}(\bar\beta)=\bar\beta$.

For (c), put $T=t+1$ and $q=d'+\varepsilon'=\lceil m'/2\rceil$, and separate
the tame and wild cases.  If $K/F$ is tamely ramified of odd degree
$\ell$, then
$m'-1=\ell(m-1)$.
For $x\in\p_K^q$, let $s_j(x)$ be the $j$th elementary symmetric
function of its $\ell$ conjugates.  For $2\le j\le\ell-1$, every
summand defining $s_j(x)$ has $K$-valuation at least $jq$.  Since $s_j(x)\in F$ and
$v_K|_F=\ell v_F$, it follows that
\[
 v_F(s_j(x))\ge\left\lceil\frac{jq}{\ell}\right\rceil
 \ge\left\lceil\frac{2q}{\ell}\right\rceil\ge m.
\]
Also $v_F(\Nm(x))=v_K(x)\ge q\ge m$.  Thus all nonlinear norm terms
vanish modulo $\p_F^m$.

If $K/F$ is wildly ramified, its degree is the residue characteristic
$p$, $D=(p-1)T$, and
$m'=p(m-T)+T$.
The assumption $d\ge T$ gives $q\ge m$.  Moreover
$2q+D\ge m'+D=pm$, so \eqref{eq:symmetric-bound} gives
$v_F(s_j(x))\ge m$ for $2\le j\le p-1$.  Again the norm term has
valuation at least $m$.  Thus in either case
\[
 \Nm(1+x)-1\equiv\Tr_{K/F}(x)\pmod{\p_F^m}
 \qquad(x\in\p_K^q).
\]
Proposition~\ref{lem:truncated-norm-adjoint} gives
$P^*_{K/F}(\bar\beta)=\bar\beta$, proving (c).

For (d), put $T=t+1$, $v=m-T$, $r_0=\lfloor T/2\rfloor$, and
$s_0=\lceil T/2\rceil$.  Proposition~\ref{prop:pullback-conductor}(i)
and \eqref{eq:minimal-critical} at $m=T$ give
$m'=pm-D=T+pv=2d'+\varepsilon$,
so $\varepsilon'=\varepsilon$.

Relative to $\gamma_F$, the stationary class of $\tau$ on
$\p_F^{s_0}$ lies in $\p_F^v/\p_F^{v+r_0}$.
Since $r_0\le t$, Lemma~\ref{lem:norm-stationary-representatives}
with $(h,s)=(v,r_0)$ gives $\alpha=\Nm(\widetilde{\alpha})$ satisfying
\eqref{eq:tau-high-linearization}.
Likewise $d\le t$, so the same lemma with $(h,s)=(0,d)$ gives
$\beta=\Nm(\widetilde{\beta})$ representing
$\operatorname{St}_{\gamma_F,d+\varepsilon}(\chi_F)$. By definition of this stationary class,
\[
 \chi_F(1+x)=\psi_F(\beta x/\gamma_F)
 \qquad(x\in\p_F^{d+\varepsilon}).
\]
For $j\in\F_p^\times$, equation~\eqref{eq:minimal-orbit-conductors}
gives $m_F(\tau^j\chi_F)=m$ because $m\ge T$; for $j=0$ the
character is $\chi_F$.

Put $q=\lceil m'/2\rceil$.  Since
$2q+D\ge pm$, equation~\eqref{eq:wild-norm-truncation} gives
\begin{equation*}
 \Nm(1+x)-1\equiv\Tr(x)+\Nm(x)\pmod{\p_F^m}
 \qquad(x\in\p_K^q).
\end{equation*}
Since
\[
 q-v=\left\lceil\frac{T+(p-2)v}{2}\right\rceil
 \ge\left\lceil\frac{T+v}{2}\right\rceil=d+\varepsilon\ge s_0,
\]
the element $y=\widetilde{\beta}x/\widetilde{\alpha}$ belongs to $\p_K^{s_0}$.  Since
$ \Nm\!\left(\frac{\widetilde{\beta}x}{\widetilde{\alpha}}\right)
 =\frac{\beta}{\alpha}\Nm(x)$
by the definitions of $\alpha$ and $\beta$,
Lemma~\ref{lem:norm-character-conversion}, applied to
$y=\widetilde{\beta}x/\widetilde{\alpha}$, gives
\[
 \psi_F\!\left(\frac{\beta \Nm(x)}{\gamma_F}\right)
 =\psi_F\!\left(-\frac{\alpha\Tr(\widetilde{\beta}x/\widetilde{\alpha})}
                         {\gamma_F}\right).
\]
The congruence for $\Nm(1+x)-1$ above therefore gives
\begin{align*}
 \psi_F\!\left(
   \frac{\beta(\Nm(1+x)-1)}{\gamma_F}\right)
  &=\psi_F\!\left(
   \frac{\Tr\bigl((\beta-\widetilde{\beta}\alpha/\widetilde{\alpha})x\bigr)}
        {\gamma_F}\right)\\
 &=\psi_K\!\left(
   \frac{(\beta-\widetilde{\beta}\alpha/\widetilde{\alpha})x}{\gamma_K}\right),
\end{align*}
where $\gamma_K=\gamma_F$.  By the defining adjoint identity
\eqref{eq:Pstar-def},
$\beta-\widetilde{\beta}\alpha/\widetilde{\alpha}$ represents
$P^*_{K/F}(\bar\beta)$.  Equation~\eqref{eq:stationary-class-adjoint}
then gives \eqref{eq:Pstar-intermediate}.

Since $d+\varepsilon\ge s_0$, equations
\eqref{eq:Lamprecht-linearization}, \eqref{eq:tau-high-linearization}, and
Lemma~\ref{lem:teichmueller-stationary} give
\[
 (\tau^j\chi_F)(1+x)=\psi_F((\beta+j\alpha)x/\gamma_F)
 \qquad(x\in\p_F^{d+\varepsilon}).
\]
Equation~\eqref{eq:minimal-orbit-conductors} gives conductor $m$
for $\tau^j\chi_F$.
Using $\alpha=\Nm_{K/F}(\widetilde{\alpha})$ and factoring
$\Nm(\widetilde{\beta}+j\widetilde{\alpha})=\beta \Nm(1+j\widetilde{\alpha}/\widetilde{\beta})$, equation~\eqref{eq:symmetric-bound} gives every intermediate term
valuation at
least $d$, since
$iv+(p-1)T\ge pd\qquad(1\le i\le p-1)$.
The two endpoint terms are $\beta$ and $j\alpha$, because $j^p=j$ for
the chosen Teichm\"uller lift.  This proves
\eqref{eq:twist-beta-intermediate} as a class congruence.

For the product congruence, choose the representative
\[
 \beta_K^0=\beta-\frac{\widetilde{\beta}\alpha}{\widetilde{\alpha}}
\]
of the class in \eqref{eq:Pstar-intermediate}.  By
\eqref{eq:norm-of-stationary-class-well-defined}, it suffices to compute
$\Nm(\beta_K^0)$ modulo $\p_F^d$.  Put
$y=\widetilde{\beta}\alpha/\widetilde{\alpha}$.  Then
$v_K(y)=(p-1)v$ and $\Nm(y)=\beta\alpha^{p-1}$.  In the expansion of
$\Nm(\beta-y)$ every intermediate term has valuation at least $d$,
because
$i(p-1)v+(p-1)T\ge pd\qquad(1\le i\le p-1)$.
The endpoints give $\beta^p-\beta\alpha^{p-1}$, which equals
$\prod_{j\in\F_p}(\beta+j\alpha)$.  This proves the asserted
congruence for $\Nm_{K/F}(\beta_K)$.
\end{proof}

\begin{proposition}[Compatible wild-quadratic stationary classes and representatives]
\label{prop:wild-quadratic-high-parameters}
Suppose $K/F$ is wildly ramified quadratic and put $T=t+1$. Retain the minimal-conductor choice of $\chi_F$.
Assume $m=m_F(\chi_F)\ge T$.  Let $\tau$ be the nontrivial character in
$S(K/F)$, and let $\gamma_F$ be admissible for $(\chi_F,\psi_F)$.
Then $\gamma_K=\gamma_F$, viewed in $K$, is admissible for
$(\chi_K,\psi_K)$.
There exist $\widetilde{\alpha}\in K^\times$ with $v_K(\widetilde{\alpha})=m-T$,
$\widetilde{\beta}\in\cO_K^\times$, and $\kappa\in U_F^t$ such that the elements
\[
 \alpha=\Nm_{K/F}(\widetilde{\alpha}),\qquad \beta=\kappa \Nm_{K/F}(\widetilde{\beta})
\]
satisfy
\[
 \begin{aligned}
 \tau(1+x)&=\psi_F(\alpha \kappa x/\gamma_F)
      &&(x\in\p_F^{\lceil T/2\rceil}),\\
 \chi_F(1+x)&=\psi_F(\beta x/\gamma_F)
      &&(x\in\p_F^{d+\varepsilon}).
 \end{aligned}
\]
Let $\beta_K\in\cO_K^\times$ and $\beta_{\tau\chi_F}\in\cO_F^\times$
represent the stationary classes of $\chi_K$ and $\tau\chi_F$ for
$\gamma_K$ and $\gamma_F$, respectively.  Then
\begin{equation}\label{eq:Pstar-wild-two}
 \beta_K\equiv\beta-\frac{\widetilde{\beta}\alpha \kappa}{\widetilde{\alpha}}
 \pmod{\p_K^{d'}},
 \qquad
 \beta_{\tau\chi_F}\equiv\beta+\alpha \kappa\pmod{\p_F^d},
\end{equation}
and
\[
 \Nm_{K/F}(\beta_K)\equiv\beta(\beta+\alpha \kappa)
 \pmod{\p_F^d}.
\]
\end{proposition}

\begin{proof}
Write $\Nm=\Nm_{K/F}$.
Put $v=m-T\ge0$.  Proposition~\ref{prop:pullback-conductor}(i),
with \eqref{eq:minimal-critical} when $m=T$, gives
$m'=2m-T=2d'+\varepsilon'$.
Since $D_{K/F}=T$, equation~\eqref{eq:additive-conductor-trace} gives
$m'+n_K(\psi_K)=2(m+n_F(\psi_F))=v_K(\gamma_F)$.
Thus $\gamma_K=\gamma_F$ is admissible.  Set
\[
 r=\left\lfloor\frac T2\right\rfloor,
 \qquad s=\left\lceil\frac T2\right\rceil.
\]
Then
\begin{equation*}
 d'+\varepsilon'=m-r,
 \qquad d'+\varepsilon'-v=s,
 \qquad r+s=T.
\end{equation*}

Choose any representative $\beta$ of the stationary class of $\chi_F$.
By the second isomorphism in \eqref{eq:norm-below-break}, with $r=0$,
choose $\widetilde{\beta}\in\cO_K^\times$ such that
$\Nm(\widetilde{\beta})/\beta\in U_F^t$, and put
\[
 \kappa=\frac{\beta}{\Nm(\widetilde{\beta})}\in U_F^t.
\]
Let $c_\tau$ be any representative of the stationary coefficient class
of $\tau$ relative to $\gamma_F$ on $\p_F^s$.  That class lies in
$\p_F^v/\p_F^{v+r}$.
Apply Lemma~\ref{lem:norm-stationary-representatives} to the element
$c_\tau \kappa^{-1}$ with $(h,s)=(v,r)$.  It supplies
$\widetilde{\alpha}\in K^\times$ with $v_K(\widetilde{\alpha})=v$ such that, for
$\alpha=\Nm(\widetilde{\alpha})$,
$\alpha \kappa\equiv c_\tau\pmod{\p_F^{v+r}}$.
Thus $\alpha \kappa$ is a permitted representative and
\begin{equation}\label{eq:wild-two-tau-linearization}
 \tau(1+x)=
 \psi_F\!\left(\frac{\alpha \kappa x}{\gamma_F}\right)
 \qquad(x\in\p_F^s).
\end{equation}
This constructs the required simultaneous choices.

Lemma~\ref{lem:norm-character-conversion}, applied to
\eqref{eq:wild-two-tau-linearization} with coefficient $\alpha \kappa$
and denominator $\gamma_F$, gives, for $y\in\p_K^s$,
\begin{equation*}
 \psi_F\!\left(\frac{\alpha \kappa \Nm(y)}{\gamma_F}\right)
 =\psi_F\!\left(-\frac{\alpha \kappa\Tr(y)}{\gamma_F}\right).
\end{equation*}

Let $x\in\p_K^{d'+\varepsilon'}$ and put
$y=\widetilde{\beta}x/\widetilde{\alpha}$.  Since $d'+\varepsilon'-v=s$, $y\in\p_K^s$, and the identities
$\beta=\kappa \Nm(\widetilde{\beta})$ and $\alpha=\Nm(\widetilde{\alpha})$ give
$\beta \Nm(x)=\alpha \kappa \Nm(y)$.
Applying the displayed norm--trace character identity to
$y=\widetilde{\beta}x/\widetilde{\alpha}$, and using the quadratic identity
$\Nm(1+x)-1=\Tr_{K/F}(x)+\Nm(x)$, yields
\begin{align*}
 \psi_F\!\left(\frac{\beta(\Nm(1+x)-1)}{\gamma_F}\right)
 &=\psi_F\!\left(
   \frac{\beta\Tr(x)-
    \alpha \kappa\Tr(\widetilde{\beta}x/\widetilde{\alpha})}{\gamma_F}\right)\\
 &=\psi_K\!\left(
   \frac{(\beta-\widetilde{\beta}\alpha \kappa/\widetilde{\alpha})x}{\gamma_F}
   \right).
\end{align*}
By \eqref{eq:stationary-class-adjoint}, this proves the first
congruence in \eqref{eq:Pstar-wild-two}.

Since $d+\varepsilon\ge s$, multiplying
\eqref{eq:wild-two-tau-linearization} by the linearization of $\chi_F$
gives
\[
 (\tau\chi_F)(1+x)
 =\psi_F((\beta+\alpha \kappa)x/\gamma_F)
 \qquad(x\in\p_F^{d+\varepsilon}).
\]
Since $m\ge T$, equation~\eqref{eq:minimal-orbit-conductors}
gives $m_F(\tau\chi_F)=m$.  This proves the second congruence in
\eqref{eq:Pstar-wild-two}.

Finally choose the representative
\[
 \beta_K^0=\beta-\frac{\widetilde{\beta}\alpha \kappa}{\widetilde{\alpha}}
\]
of the stationary class of $\chi_K$, and put
$c=\widetilde{\beta}\alpha \kappa/\widetilde{\alpha}$.  By
\eqref{eq:norm-of-stationary-class-well-defined}, it is enough to
compute $\Nm(\beta_K^0)$ modulo $\p_F^d$.  The quadratic norm
formula gives
\[
 \Nm(\beta-c)=\beta^2-\beta\Tr(c)+\Nm(c),
 \qquad \Nm(c)=\beta\alpha \kappa.
\]
Now $v_K(c)=v$, and the trace-ideal formula \eqref{eq:trace-ideal}, with different exponent
$T$, gives
\[
 v_F(\Tr(c))\ge
 \left\lfloor\frac{v+T}{2}\right\rfloor=d.
\]
Consequently
\[
 \Nm(\beta_K)\equiv \Nm(\beta_K^0)
 \equiv\beta^2+\beta\alpha \kappa
 =\beta(\beta+\alpha \kappa)\pmod{\p_F^d},
\]
which proves the asserted congruence for $\Nm_{K/F}(\beta_K)$.
\end{proof}

\subsection{Low-conductor parameters}

Let $K/F$ be wildly ramified cyclic of prime degree
$p=\operatorname{char}k_F$, with break $t$, and put $T=t+1$.
Let $\tau$ be a nontrivial element of $S(K/F)$.  Retain the minimal-conductor choice of $\chi_F$.  Assume
\[
 2\le m=2d+\varepsilon\le T,
 \qquad v=T-m.
\]
By \eqref{eq:minimal-critical}, $m_K(\chi_K)=m$.  Put
\[
 r_0=\left\lfloor\frac T2\right\rfloor,
 \qquad s_0=\left\lceil\frac T2\right\rceil.
\]
\begin{proposition}[Low-conductor stationary-class table]
\label{prop:low-parameter-table}
Choose
\[
 \Gamma_\tau\cO_F=\p_F^{T+n_F(\psi_F)},\qquad
 \widetilde{\xi}\cO_K=\p_K^v,\qquad
 \xi=\Nm_{K/F}(\widetilde{\xi}),
\]
and put
\begin{equation}\label{eq:gammas-low}
 \gamma_F=\Gamma_\tau/\xi,\qquad
 \gamma_K=\Gamma_\tau/\widetilde{\xi}.
\end{equation}
These denominators are admissible for $\chi_F$ and $\chi_K$,
respectively. There exist $\widetilde\alpha,\widetilde\beta\in U_K^0$
whose norms
\begin{equation}\label{eq:low-norm-representatives}
 \alpha=\Nm_{K/F}(\widetilde\alpha),\qquad
 \beta=\Nm_{K/F}(\widetilde\beta)
\end{equation}
satisfy
\begin{equation}\label{eq:low-mu}
\begin{aligned}
 \tau(1+x)&=\psi_F(\alpha x/\Gamma_\tau)
 &&(x\in\p_F^{s_0}),\\
 \chi_F(1+x)&=\psi_F(\beta x/\gamma_F)
             =\psi_F(\xi\beta x/\Gamma_\tau)
 &&(x\in\p_F^{d+\varepsilon}).
\end{aligned}
\end{equation}
For these choices, $j\alpha+\xi\beta\notin\p_F$ for $j\ne0$ when
$m=T$.
Let $\beta_K\in\cO_K^\times$ represent the stationary class of
$\chi_K$ relative to $\gamma_K$.  Then
\[
 \beta_K\equiv\frac{\beta\xi}{\widetilde{\xi}}
       -\frac{\widetilde{\beta}\alpha}{\widetilde{\alpha}}\pmod{\p_K^d}.
\]
Relative to $\Gamma_\tau$, the stationary class of $\chi_F$ is
\[
 \operatorname{St}_{\Gamma_\tau,d+\varepsilon}(\chi_F)
 =\xi\beta+\p_F^{v+d}\in\p_F^v/\p_F^{v+d},
 \qquad \xi\beta=\Nm_{K/F}(\widetilde{\xi}\widetilde{\beta}).
\]
For $j\in\F_p^\times$, let $\beta_j\in\cO_F^\times$
represent $\operatorname{St}_{\Gamma_\tau,s_0}(\tau^j\chi_F)$.
With the Teichm\"uller representative used for $j$ in $F$ and $K$,
\begin{equation}\label{eq:low-beta-twists}
 \beta_j\equiv \Nm_{K/F}(j\widetilde{\alpha}+\widetilde{\xi}\widetilde{\beta})
 \equiv j\alpha+\xi\beta\pmod{\p_F^{r_0}}
 \qquad(j\in\F_p^\times).
\end{equation}
Moreover,
\[
 \Nm_{K/F}(\beta_K)\equiv
 \beta\prod_{j\in\F_p^\times}(\xi\beta+j\alpha)
 \pmod{\p_F^d}.
\]
\end{proposition}

\begin{proof}
Write $\Nm=\Nm_{K/F}$.
The denominators in \eqref{eq:gammas-low} are admissible, since
\[
 \begin{aligned}
 v_F(\gamma_F)&=T+n_F(\psi_F)-v=m+n_F(\psi_F),\\
 v_K(\gamma_K)&=p(T+n_F(\psi_F))-v=m+n_K(\psi_K),
 \end{aligned}
\]
where the second equality uses \eqref{eq:additive-conductor-trace}.
Relative to $\Gamma_\tau$, the stationary class of $\tau$ on
$\p_F^{s_0}$ lies in $\cO_F/\p_F^{r_0}$; relative to $\gamma_F$,
that of $\chi_F$ lies in $\cO_F/\p_F^d$. Since $r_0,d\le t$,
Lemma~\ref{lem:norm-stationary-representatives} gives the exact norm
representatives in \eqref{eq:low-norm-representatives} with the
linearizations \eqref{eq:low-mu}. At $m=T$, the asserted
noncancellation follows from
\eqref{eq:minimal-orbit-no-leading-cancellation} with
$\gamma=\Gamma_\tau$.

Put $q=d+\varepsilon=\lceil m/2\rceil$.  Since
$2q+(p-1)T\ge pm$, equation~\eqref{eq:wild-norm-truncation} gives
\begin{equation*}
 \Nm(1+x)-1\equiv\Tr(x)+\Nm(x)\pmod{\p_F^m}
 \qquad(x\in\p_K^q).
\end{equation*}
For $y=\widetilde{\xi}\widetilde{\beta}x/\widetilde{\alpha}$ one has
$v_K(y)\ge q+v=T-\floor{m/2}\ge s_0$.
Lemma~\ref{lem:norm-character-conversion}, applied to the
linearization of $\tau$ in \eqref{eq:low-mu} with denominator
$\Gamma_\tau$, gives
\[
 \psi_F\!\left(-\frac{\alpha\Tr(y)}{\Gamma_\tau}\right)
 =\psi_F\!\left(\frac{\alpha \Nm(y)}{\Gamma_\tau}\right).
\]
By \eqref{eq:low-norm-representatives},
$\alpha \Nm(y)=\xi\beta \Nm(x)$, and hence
\begin{equation*}
 \psi_F\!\left(-\frac{\alpha\Tr(y)}{\Gamma_\tau}\right)
 =\psi_F\!\left(\frac{\xi\beta \Nm(x)}{\Gamma_\tau}\right).
\end{equation*}
Put $\beta_K^0=\beta\xi/\widetilde{\xi}-\widetilde{\beta}\alpha/\widetilde{\alpha}$.  Then
\begin{align*}
 \psi_K\!\left(\frac{\beta_K^0x}{\gamma_K}\right)
 &=\psi_F\!\left(
   \frac{\Tr_{K/F}(\widetilde{\xi}\beta_K^0x)}{\Gamma_\tau}\right)\\
 &=\psi_F\!\left(
   \frac{\xi\beta\Tr(x)
         -\alpha\Tr(\widetilde{\xi}\widetilde{\beta}x/\widetilde{\alpha})}{\Gamma_\tau}\right).
\end{align*}
Substituting the identity for $\psi_F(-\alpha\Tr(y)/\Gamma_\tau)$ and the congruence for $\Nm(1+x)-1$ above gives
\[
 \psi_K(\beta_K^0x/\gamma_K)
 =\psi_F\!\left(\frac{\beta(\Nm(1+x)-1)}{\gamma_F}\right)
 =\chi_K(1+x).
\]
Corollary~\ref{cor:stationary-class-under-norm} now gives the stated
congruence for $\beta_K$.

The formula for the stationary class of $\chi_F$ follows from
$\gamma_F=\Gamma_\tau/\xi$ and $v_F(\xi)=v$.  For $j\ne0$, put
$w=\widetilde{\xi}\widetilde{\beta}/(j\widetilde{\alpha})$, so $v_K(w)=v$.  Equation~\eqref{eq:symmetric-bound} gives
\[
 v_F(s_i(w))\ge
 \left\lfloor\frac{iv+(p-1)T}{p}\right\rfloor\ge r_0
 \qquad(1\le i\le p-1),
\]
and therefore
$\Nm(j\widetilde{\alpha}+\widetilde{\xi}\widetilde{\beta})
 \equiv j\alpha+\xi\beta\pmod{\p_F^{r_0}}$.
Since $s_0\ge d+\varepsilon$, the two linearizations in
\eqref{eq:low-mu} and Lemma~\ref{lem:teichmueller-stationary} give
\[
 (\tau^j\chi_F)(1+x)=\psi_F((j\alpha+\xi\beta)x/\Gamma_\tau)
 \qquad(x\in\p_F^{s_0}).
\]
For $j\ne0$, the character has conductor $T$ by
Lemma~\ref{lem:minimal-orbit-stationary}, so this proves
\eqref{eq:low-beta-twists}.

Put
\[
 u=\frac{\widetilde{\xi}\widetilde{\beta}}{\widetilde{\alpha}},
 \qquad \mathcal N=\Nm(u)=\frac{\xi\beta}{\alpha},
\]
so the representative already constructed is
\[
 \beta_K^0=\frac{\alpha}{\widetilde{\xi}}(\mathcal N-u).
\]
By \eqref{eq:norm-of-stationary-class-well-defined}, $\Nm(\beta_K)$ modulo
$\p_F^d$ may be computed from this representative.  Thus
\[
 \Nm(\beta_K)\equiv \Nm(\beta_K^0)
 =\frac{\alpha^p}{\xi}\Nm(\mathcal N-u)\pmod{\p_F^d}.
\]
In the expansion of $\Nm(\mathcal N-u)$ the two endpoints are
$\mathcal N^p+(-1)^p\mathcal N$.  For $1\le i\le p-1$, the intermediate term
$\mathcal N^{p-i}s_i(u)$, after multiplication by
$\alpha^p/\xi$, has valuation at least
\[
 L_i=(p-i-1)v+
 \left\lfloor\frac{iv+(p-1)T}{p}\right\rfloor.
\]
Because $T$ is an integer,
\[
 \left\lfloor\frac{iv+(p-1)T}{p}\right\rfloor
 =T-\left\lceil\frac{T-iv}{p}\right\rceil.
\]
Moreover $i\ge1$ and $v=T-m\ge0$, so
$T-iv\le T-v=m=2d+\varepsilon$.  Since $p\ge2$,
\[
 \left\lceil\frac{T-iv}{p}\right\rceil
 \le\left\lceil\frac m2\right\rceil=d+\varepsilon.
\]
Consequently
\[
 \begin{aligned}
 L_i
 &\ge(p-i-1)v+T-d-\varepsilon\\
 &=(p-i)v+d\\
 &\ge d.
 \end{aligned}
\]
Thus every intermediate term vanishes modulo $\p_F^d$, and only the
endpoints survive.  The Teichm\"uller product identity
\eqref{eq:teichmueller-polynomials} gives
\[
 \Nm(\beta_K)\equiv
 \frac{\alpha^p}{\xi}\bigl(\mathcal N^p+(-1)^p\mathcal N\bigr)
 =\beta\prod_{j\ne0}(\xi\beta+j\alpha)
 \pmod{\p_F^d}.
\]
This proves the asserted congruence for $\Nm_{K/F}(\beta_K)$.
\end{proof}

\section{Unramified and tame cases}

Let $K/F$ be cyclic of prime degree $\ell$, let $\chi_F$ be a
quasi-character of $F^\times$, and let $\psi_F$ be a nontrivial
additive character of $F$.  Put
\[
 \chi_K=\chi_F\circ \Nm_{K/F},\qquad
 \psi_K=\psi_F\circ\Tr_{K/F}.
\]
Retain the minimal-conductor choice of $\chi_F$.

\begin{lemma}[Ramified conductor-zero comparison]
\label{lem:ramified-conductor-zero}
If $K/F$ is ramified cyclic of prime degree and $\chi_F$ is
unramified, then the First Main identity \eqref{eq:first-main} holds.
\end{lemma}
\begin{proof}
Let $t$ be the lower break and put $T=t+1$ and $n=n_F(\psi_F)$.
Choose a uniformizer
$\varpi_K$ with $\varpi_F=\Nm_{K/F}(\varpi_K)$.
Equation~\eqref{eq:unramified-character-local-constant} gives
$\Delta_K(\chi_K)=\chi_F(\varpi_F)^{n_K(\psi_K)}$.
Each nontrivial $\mu\in S(K/F)$ has conductor $T$, whereas the
trivial character has conductor zero. Thus \eqref{eq:unramified-twist}
gives
\begin{equation}\label{eq:ramified-conductor-zero}
\begin{aligned}
 \prod_{\mu\in S(K/F)}\Delta_F(\chi_F\mu)
 &=\chi_F(\varpi_F)^{n+(\ell-1)(T+n)}
                   \prod_{\mu\in S(K/F)}\Delta_F(\mu)\\
 &=\Delta_K(\chi_K)\prod_{\mu\in S(K/F)}\Delta_F(\mu),
\end{aligned}
\end{equation}
since $n+(\ell-1)(T+n)=\ell n+(\ell-1)T=n_K(\psi_K)$ by
\eqref{eq:additive-conductor-trace}. No evaluation of the
norm-character product is needed.
\end{proof}

\begin{lemma}[Common elementary-factor power identity]
\label{lem:common-elementary-power}
If both multiplicative conductors exceed $1$, a common admissible
denominator $\gamma\in F^\times$ and a common stationary representative
$\beta\in\cO_F^\times$ give
\begin{equation}\label{eq:common-elementary-power}
 \chi_K(\gamma)\Delta_{2,K}(\chi_K)
 =\bigl(\chi_F(\gamma)\Delta_{2,F}(\chi_F)\bigr)^\ell.
\end{equation}
Here both $\Delta_2$ factors use these choices.
\end{lemma}
\begin{proof}
Since $\Nm_{K/F}(a)=a^\ell$ and $\Tr_{K/F}(a)=\ell a$ for $a\in F$,
substitution in \eqref{eq:stationary-factor-definitions} proves
\eqref{eq:common-elementary-power}.
\end{proof}

\begin{proposition}[Unramified case]\label{prop:unramified-new}
The First Main identity \eqref{eq:first-main} holds when $K/F$ is unramified.
\end{proposition}

\begin{proof}
Write $\Nm=\Nm_{K/F}$.
Put
\[
 k=k_F,\qquad k'=k_K,\qquad
 n=n_F(\psi_F),\qquad m=m_F(\chi_F).
\]
Choose a uniformizer $\varpi$ of $F$.  By
Lemma~\ref{lem:unramified-compatibility}, it is also a uniformizer of
$K$, the residue extension $k'/k$ has degree $\ell$, and
$n_K(\psi_K)=n,\qquad m_K(\chi_K)=m$.

\smallskip
\noindent\emph{Step 1: reduction to a power identity.}
Lemma~\ref{lem:unramified-compatibility}(e) identifies
$F^\times/\Nm(K^\times)$ with the cyclic group generated by the class of
$\varpi$, of order $\ell$.  Hence every $\mu\in S(K/F)$ is unramified,
and the values $\mu(\varpi)$ run through all $\ell$th roots of unity.
Equations~\eqref{eq:unramified-twist} and
\eqref{eq:unramified-character-local-constant} give
\[
 \begin{aligned}
 \Delta_F(\mu\chi_F,\psi_F)
 &=\mu(\varpi)^{m+n}\Delta_F(\chi_F,\psi_F),\\
 \Delta_F(\mu,\psi_F)&=\mu(\varpi)^n.
 \end{aligned}
\]
Both formulas remain valid when $m=0$.  The values $\mu(\varpi)$ are
the roots of $X^\ell-1$, whose product is $(-1)^{\ell-1}$; hence
$\prod_{\mu\in S(K/F)}\mu(\varpi)=(-1)^{\ell-1}$.
The two local-constant formulas above give the quotient of
the product on the right side of \eqref{eq:first-main} by the
norm-character product on its left as
$(-1)^{m(\ell-1)}\Delta_F(\chi_F,\psi_F)^\ell$.
Thus \eqref{eq:first-main} is equivalent to
\begin{equation}\label{eq:unramified-power-new}
 \Delta_K(\chi_K,\psi_K)=
 (-1)^{m(\ell-1)}\Delta_F(\chi_F,\psi_F)^\ell.
\end{equation}

\smallskip
\noindent\emph{Step 2: conductor $m=0$.}
Formula \eqref{eq:unramified-character-local-constant} gives
\[
 \Delta_F(\chi_F,\psi_F)=\chi_F(\varpi)^n,
 \qquad
 \Delta_K(\chi_K,\psi_K)=\chi_K(\varpi)^n.
\]
By Lemma~\ref{lem:unramified-compatibility}(c),
$\Nm(\varpi)=\varpi^\ell$, so
$\chi_K(\varpi)=\chi_F(\varpi)^\ell$.  This proves
\eqref{eq:unramified-power-new} for $m=0$.

\smallskip
\noindent\emph{Step 3: conductor $m=1$.}
For $\bar u\in k^\times$ and $\bar x\in k$, choose lifts
$u\in\cO_F^\times$ and $x\in\cO_F$.  Define
\[
 \bar\chi(\bar u)=\chi_F(u),
 \qquad
 \bar\psi(\bar x)=\psi_F(x/\varpi^{n+1}).
\]
These are the residual multiplicative and additive characters on $k$.  By
\eqref{eq:unramified-residue-trace-norm}, the residual characters
on $k'$ are
\begin{equation*}
 \bar\chi_K=\bar\chi\circ \Nm_{k'/k},
 \qquad
 \bar\psi_K=\bar\psi\circ\Tr_{k'/k}.
\end{equation*}
The conductor-one formula \eqref{eq:conductor-one-local-gauss} gives
\[
 \Delta_F(\chi_F,\psi_F)
 =-\chi_F(\varpi)^{n+1}
   \ph\!\left(\tau_k(\bar\chi,\bar\psi)\right)
\]
and
\[
 \Delta_K(\chi_K,\psi_K)
 =-\chi_K(\varpi)^{n+1}
   \ph\!\left(\tau_{k'}(\bar\chi_K,\bar\psi_K)\right).
\]
Equations~\eqref{eq:HD-lift} and
\eqref{eq:unramified-residue-trace-norm} give
\[
 \tau_{k'}(\bar\chi_K,\bar\psi_K)
 =\tau_k(\bar\chi,\bar\psi)^\ell.
\]
Since $\ph(z^\ell)=\ph(z)^\ell$ for $z\ne0$, and since
$\chi_K(\varpi)=\chi_F(\varpi)^\ell$, this yields
\[
 \Delta_K(\chi_K,\psi_K)
 =(-1)^{\ell-1}\Delta_F(\chi_F,\psi_F)^\ell,
\]
which is \eqref{eq:unramified-power-new} because $m=1$.

\smallskip
\noindent\emph{Step 4: conductor $m>1$.}
Write $m=2d+\varepsilon$, with $\varepsilon\in\{0,1\}$, and put
$\gamma=\varpi^{m+n}$.  Proposition~\ref{prop:high-parameter-table}(a)
provides a common stationary representative $\beta\in\cO_F^\times$
for $\chi_F$ and $\chi_K$ with this common admissible denominator.
If $\varepsilon=0$, both residual factors are $1$ by
\eqref{eq:stationary-factor-definitions}.  Equations
\eqref{eq:common-elementary-power} and \eqref{eq:Delta-factorization-new}
then prove \eqref{eq:unramified-power-new}, whose sign is $1$ for even $m$.

It remains to treat $\varepsilon=1$.  Then $m=2d+1$ and $d\geq1$.
Take $\delta=\varpi^d$ and form $H_F:k\to\C^\times$ and
$H_K:k'\to\C^\times$ by \eqref{eq:Lamprecht-Hasse-function}, using
$\gamma$, $\beta$, and $\delta$ for both characters.  The residual
additive character \eqref{eq:Lamprecht-residual-character} is
\begin{equation*}
 \psi_0(\bar z)=\psi_F(\beta\delta^2z/\gamma),
 \qquad z\in\cO_F\text{ a lift of }\bar z\in k.
\end{equation*}

For $x\in\cO_K$, let $s_j(x)$ be the $j$th elementary symmetric
function of its $F$-conjugates.  Thus
$s_1(x)=T(x):=\Tr_{K/F}(x)$ and
$s_2(x)=\sum_{\{\sigma,\tau\}}x^\sigma x^\tau$,
the sum being over unordered pairs of distinct $F$-embeddings of $K$.
The norm expansion is
\[
 \Nm(1+\delta x)
 =1+\delta T(x)+\delta^2s_2(x)
   +\sum_{j=3}^{\ell}\delta^js_j(x),
\]
where the last sum is empty if $\ell=2$.  Each term in that sum lies in
$\p_F^{3d}\subseteq\p_F^{2d+1}$, and the same is true of the cross term
$\delta^3T(x)s_2(x)$.  Hence
\begin{equation*}
 \Nm(1+\delta x)
 \equiv(1+\delta T(x))(1+\delta^2s_2(x))
 \pmod{\p_F^{m}}.
\end{equation*}
Because $\delta^2s_2(x)\in\p_F^{2d}\subseteq\p_F^{d+1}$,
equation \eqref{eq:Lamprecht-linearization} and the definition of $\psi_0$ give
\begin{equation*}
 \chi_F(1+\delta^2s_2(x))^{-1}
 =\psi_F(-\beta\delta^2s_2(x)/\gamma)
 =\psi_0(-\overline{s_2(x)}).
\end{equation*}
Substituting the preceding norm congruence and character evaluation in
\eqref{eq:Lamprecht-Hasse-function} over $K$ gives
\begin{equation*}
 H_K(\bar x)
 =H_F\!\left(\Tr_{k'/k}(\bar x)\right)
  \psi_0\!\left(-\omega_{k'/k}(\bar x)\right).
\end{equation*}
Here \eqref{eq:unramified-residue-trace-norm} identifies
$\overline{T(x)}$ with the residue-field trace.  The same reduction of
the characteristic polynomial, now using its second coefficient,
identifies $\overline{s_2(x)}$ with the function $\omega_{k'/k}$ in
\eqref{eq:Hasse-lift-definition}.
Taking phases in \eqref{eq:Hasse-lift-sum}, using the displayed formula for $H_K$ and \eqref{eq:stationary-factor-definitions}, gives
\begin{equation*}
 \Delta_{3,K}(\chi_K)
 =(-1)^{\ell-1}\Delta_{3,F}(\chi_F)^\ell.
\end{equation*}
This residue-field identity and equations~\eqref{eq:common-elementary-power} and \eqref{eq:Delta-factorization-new} prove
\eqref{eq:unramified-power-new}, since $m$ is odd.

\end{proof}

\begin{proposition}[Tame conductor-one case]
\label{prop:tame-conductor-one-comparison}
Let $K/F$ be tamely ramified cyclic of prime degree $\ell$.
Retain the minimal-conductor choice of $\chi_F$.
Suppose $m_F(\chi_F)=1$.  Then the First Main identity \eqref{eq:first-main} holds.
\end{proposition}

\begin{proof}
Since $t=0$ and $m_F(\chi_F)=1$, equation~\eqref{eq:minimal-critical}
gives $m_K(\chi_K)=1$.

Choose $\varpi_K$ and put $\varpi_F=\Nm_{K/F}(\varpi_K)$.  Let
$n=n_F(\psi_F)$.  Since the tame different has exponent $\ell-1$,
$n_K(\psi_K)=\ell n+\ell-1$.
The element
$\gamma=\varpi_F^{n+1}$
has $K$-valuation $\ell(n+1)=n_K(\psi_K)+1$, so it is an admissible
element for every conductor-one factor over both fields.  Let $k$ be
the common residue field.  For $x\in k$, choose a lift
$\widetilde x\in\cO_F$ and define
\[
 \bar\psi(x)=\psi_F(\widetilde x/\gamma)\quad(x\in k),
 \qquad
 \bar\chi(x)=\chi_F(\widetilde x)\quad(x\in k^\times).
\]
  On residue classes the norm and
trace are
\[
 \overline{\Nm_{K/F}(u)}=\bar u^{\,\ell},
 \qquad
 \overline{\Tr_{K/F}(u)}=\ell\bar u.
\]
Hence the residual characters attached to $(\chi_K,\psi_K)$ are
\begin{equation*}
 \bar\chi_K=\bar\chi^{\,\ell},
 \qquad
 \bar\psi_K(x)=\bar\psi(\ell x).
\end{equation*}

Choose a generator $\mu_1$ of $S(K/F)$, put $\mu_j=\mu_1^j$ for
$0\le j<\ell$, and let $\eta$ be the residue character of $\mu_1$.  Then $\eta$ has order
$\ell$ and the residue character of $\mu_j$ is $\eta^j$.  Moreover
$\mu_j(\varpi_F)=1$, because $\varpi_F=\Nm_{K/F}(\varpi_K)$.
For $j\ne0$, equation~\eqref{eq:minimal-orbit-conductors} of
Lemma~\ref{lem:minimal-orbit-stationary} gives
$m_F(\mu_j\chi_F)=1$; the case $j=0$ is the given conductor of $\chi_F$.

The finite-realization formula \eqref{eq:delta-finite-realization},
with admissible element $\gamma$, definition~\eqref{eq:tau-definition},
and the residual-character identities above give
\begin{equation*}
\begin{aligned}
 \Delta_K(\chi_K,\psi_K)
 &=-\chi_F(\gamma)^\ell
   \ph\!\left(\tau_k(\bar\chi^{\,\ell},
                              \bar\psi(\ell\,\cdot))\right),\\
 \Delta_F(\mu_j\chi_F,\psi_F)
 &=-\chi_F(\gamma)
   \ph\!\left(\tau_k(\bar\chi\eta^j,\bar\psi)\right),\\
 \Delta_F(\mu_j,\psi_F)
 &=-\ph\!\left(\tau_k(\eta^j,\bar\psi)\right)
 \quad(1\le j<\ell).
\end{aligned}
\end{equation*}
By \eqref{eq:trivial-local-constant}, $\Delta_F(1,\psi_F)=1$.  The powers of $\chi_F(\gamma)$ and
the signs are the same on the two sides of \eqref{eq:first-main} and
cancel.

The substitution $x\mapsto\ell^{-1}x$ in
\eqref{eq:tau-definition} gives
\begin{equation*}
 \tau_k(\bar\chi^{\,\ell},\bar\psi(\ell\,\cdot))
 =\bar\chi(\ell^\ell)
   \tau_k(\bar\chi^{\,\ell},\bar\psi).
\end{equation*}
Taking phases in the Hasse--Davenport multiplication formula
\eqref{eq:HD-product}, with $\chi=\bar\chi$ and $\psi=\bar\psi$,
and substituting the three local-constant evaluations above and the last scaling identity proves the proposition.
\end{proof}

\begin{proposition}[Tamely ramified odd degree]\label{prop:tame-odd-new}
The First Main identity \eqref{eq:first-main} holds for a tamely ramified
cyclic extension of odd prime degree.
\end{proposition}

\begin{proof}
Here $t=0$ and every nontrivial norm character has conductor $1$.
Put $m=m_F(\chi_F)$. For $m=0$, use
\eqref{eq:ramified-conductor-zero} with $T=1$.

For $m=1$, Proposition~\ref{prop:tame-conductor-one-comparison} applies.
Now let $m=2d+\varepsilon>1$, with $\varepsilon\in\{0,1\}$.
Since $t=0$ and $d\ge1$, Proposition~\ref{prop:high-parameter-table}(c)
provides a common admissible denominator $\gamma\in F^\times$ and a
common stationary representative $\beta\in\cO_F^\times$ for
$\chi_F$ and $\chi_K$.  Equation~\eqref{eq:stable-twist-new} gives
\[
 \prod_{\mu\in S(K/F)}\Delta_F(\mu\chi_F,\psi_F)
 =\Delta_F(\chi_F,\psi_F)^\ell,
\]
because the product of all characters in the odd-order group $S(K/F)$
is trivial.  By \eqref{eq:S-product-one}, it remains to prove
$\Delta_K(\chi_K,\psi_K)=\Delta_F(\chi_F,\psi_F)^\ell$.
Apply \eqref{eq:common-elementary-power} to these choices of
$\gamma$ and $\beta$.
By \eqref{eq:pullback-high}, $m_K(\chi_K)=\ell(m-1)+1$, which has the
same parity as $m$.
If $m$ is even, both residual factors are $1$ by
\eqref{eq:stationary-factor-definitions}.  If $m$ is odd,
\eqref{eq:ramified-Hasse-comparison} gives
$\Delta_{3,K}(\chi_K)=\Delta_{3,F}(\chi_F)^\ell$.
In either case, \eqref{eq:Delta-factorization-new} proves the required
identity.
\end{proof}

\begin{proposition}[Tamely ramified quadratic case]\label{prop:tame-two-new}
The First Main identity \eqref{eq:first-main} holds for a tamely ramified
quadratic extension.
\end{proposition}

\begin{proof}
Put $m=m_F(\chi_F)$, $n=n_F(\psi_F)$, and $k=k_F=k_K$, and write
$S(K/F)=\{1,\mu\}$.
For $m=0$, use \eqref{eq:ramified-conductor-zero} with $\ell=2$ and $T=1$.
For $m=1$, Proposition~\ref{prop:tame-conductor-one-comparison} applies.

Assume $m=2d+\varepsilon\ge2$, with $\varepsilon\in\{0,1\}$.
The residue characteristic is odd.  Let $\sigma$ be the nontrivial
element of $\Gal(K/F)$.  Choose a uniformizer $\varpi_K$ with
$\sigma(\varpi_K)=-\varpi_K$ and put
$\varpi_F=\Nm_{K/F}(\varpi_K)=-\varpi_K^2$.  Then $\mu(\varpi_F)=1$, and
$\mu|_{\cO_F^\times}$ is induced by the quadratic character $\nu$ of
$k^\times$. Put $\gamma=\varpi_F^{m+n}$, which is admissible for $\chi_F$.
Proposition~\ref{prop:high-parameter-table}(b) gives its admissibility
for $\chi_K$ and a common stationary representative
$\beta\in\cO_F^\times$ for $\chi_F$ and $\chi_K$.
By \eqref{eq:pullback-high}, $m_K(\chi_K)=2m-1$.
Since $m_F(\mu)=1\le d$, Proposition~\ref{lem:stable-twist-new}
applies to $\theta=\chi_F$ and twisting character $\mu$.
The equalities $\mu(\gamma)=1$ and
$\mu(\beta)^{-1}=\nu(\bar\beta)$ reduce
\eqref{eq:stable-twist-new} to
\begin{equation*}
 \Delta_F(\mu\chi_F)=\nu(\bar\beta)\Delta_F(\chi_F).
\end{equation*}
Using the preceding stable-twist identity and equations~\eqref{eq:common-elementary-power} (with $\ell=2$), \eqref{eq:Delta-factorization-new}, and \eqref{eq:trivial-local-constant}, we reduce \eqref{eq:first-main} to
\begin{equation}\label{eq:tame-quadratic-residual}
 \Delta_{3,K}(\chi_K)\Delta_F(\mu)
 =\nu(\bar\beta)\Delta_{3,F}(\chi_F)^2.
\end{equation}
For $x\in k$, choose a lift $\widetilde x\in\cO_F$ and define
$\psi_0(x)=\psi_F(\widetilde x/\varpi_F^{n+1})$.
Put $Q=Q_{\psi_0}(1,0)$.  Equations
\eqref{eq:conductor-one-local-gauss} and
\eqref{eq:quadratic-gauss-rescaling}, with $\mu(\varpi_F)=1$, give
$\Delta_F(\mu)=\nu(2)Q$.
By \eqref{eq:quadratic-phase-new}, $Q^2=\nu(-1)$.

If $m=2d$, then $m_K(\chi_K)=4d-1$.  Taking
$\delta_K=\varpi_K^{2d-1}$, let $x\in k$ and choose a lift
$\widetilde x\in\cO_F$.  Since $\sigma(\varpi_K)=-\varpi_K$ and
$2d-1$ is odd,
$\Tr_{K/F}(\delta_K\widetilde x)=0$,
so the additive factor in $H_K(x)$ is $1$.  Moreover
\[
 \Nm_{K/F}(1+\delta_K\widetilde x)
 =1+\varpi_F^{2d-1}\widetilde x^{\,2}.
\]
Hence, using $\chi_K=\chi_F\circ \Nm_{K/F}$ and the stationary
linearization of $\chi_F$,
\begin{align*}
 H_K(x)
 &=\chi_F\!\left(1+\varpi_F^{2d-1}\widetilde x^{\,2}\right)^{-1}\\
 &=\psi_F\!\left(
   -\frac{\beta\varpi_F^{2d-1}\widetilde x^{\,2}}{\gamma}\right)\\
 &=\psi_F\!\left(
   -\frac{\beta\widetilde x^{\,2}}{\varpi_F^{n+1}}\right)
 =\psi_0(-\bar\beta x^2),
\end{align*}
because $\gamma=\varpi_F^{2d+n}$.  Equation
\eqref{eq:quadratic-phase-new} gives
\[
 \Delta_{3,K}=Q_{\psi_0}(-2\bar\beta,0)=\nu(-2\bar\beta)Q.
\]
Using $\Delta_F(\mu)=\nu(2)Q$ and $Q^2=\nu(-1)$,
\[
 \Delta_{3,K}\Delta_F(\mu)
 =\nu(-2\bar\beta)\nu(2)\nu(-1)=\nu(\bar\beta).
\]
Since $\Delta_{3,F}(\chi_F)=1$ by
\eqref{eq:stationary-factor-definitions}, this proves
\eqref{eq:tame-quadratic-residual}.

If $m=2d+1$, take $\delta_F=\varpi_F^d$.  Then
\eqref{eq:Hasse-functional-equation} gives
\[
 \frac{H_F(x+y)}{H_F(x)H_F(y)}
 =\psi_0(\bar\beta xy)
 \qquad(x,y\in k),
\]
because $\beta\delta_F^2/\gamma=\beta/\varpi_F^{n+1}$.
The function $x\mapsto H_F(x)\psi_0(-\bar\beta x^2/2)$ is therefore
an additive character of $k$.  Since $\psi_0$ is nontrivial, there is
a unique $c\in k$ such that
\[
 H_F(x)=\psi_0(\bar\beta x^2/2+cx).
\]
Here $m_K(\chi_K)=4d+1$.  With
$\delta_K=\varpi_K^{2d}=(-1)^d\varpi_F^d$, substitution gives
\[
 H_K(x)=H_F((-1)^d x)^2
 =\psi_0(\bar\beta x^2+2(-1)^dcx).
\]
Equation~\eqref{eq:quadratic-phase-new} gives
\begin{equation*}
\begin{aligned}
 \Delta_{3,K}
 &=\nu(2\bar\beta)\psi_0(-c^2/\bar\beta)Q,\\
 \Delta_{3,F}^{\,2}
 &=\psi_0(-c^2/\bar\beta)\nu(-1).
\end{aligned}
\end{equation*}
The two phases above and $\Delta_F(\mu)=\nu(2)Q$, with $Q^2=\nu(-1)$, prove
\eqref{eq:tame-quadratic-residual}.
\end{proof}

\section{Wild odd-prime extensions: the stable range}

\begin{proposition}[Stable wild odd-prime comparison]\label{prop:wild-stable-new}
Let $K/F$ be wildly ramified of odd prime degree $p$, with break $t$.
If $m_F(\chi_F)=2d+\varepsilon>1$ and $d\geq t+1$, then the First Main
identity \eqref{eq:first-main} holds.
\end{proposition}

\begin{proof}
Every nontrivial norm character has conductor $t+1\leq d$.  The stable
twist formula \eqref{eq:stable-twist-new} and \eqref{eq:S-product-one} reduce \eqref{eq:first-main} to
\[
 \Delta_K(\chi_K,\psi_K)=\Delta_F(\chi_F,\psi_F)^p.
\]
Proposition~\ref{prop:high-parameter-table}(c) supplies a common
admissible denominator and stationary representative, so
\eqref{eq:common-elementary-power} gives the $p$th-power identity
for the elementary factors.  Here
$m'+n_K(\psi_K)=p(m+n_F(\psi_F))$ by \eqref{eq:pullback-high} and
\eqref{eq:additive-conductor-trace}, and $m'\equiv m\pmod2$ by
\eqref{eq:pullback-high}.  For even $m$, both residual factors are
$1$ by \eqref{eq:stationary-factor-definitions}; for odd $m$, apply
\eqref{eq:ramified-Hasse-comparison}.
\end{proof}

\section{Wild extensions: conductors zero and one}\label{sec:wild-endpoints}
Let $K/F$ be wildly ramified cyclic of prime degree $p$, with lower
break $t$. Put $n=n_F(\psi_F)$ and choose uniformizers with
$\varpi_F=\Nm_{K/F}(\varpi_K)$.  Then by \eqref{eq:additive-conductor-trace},
\begin{equation*}
 n_K(\psi_K)=pn+(p-1)(t+1).
\end{equation*}

\subsection{Conductor zero}
For unramified $\chi_F$, apply Lemma~\ref{lem:ramified-conductor-zero}
with $\ell=p$.

\subsection{Conductor one}
Set $k=k_F=k_K$ and $\gamma_F=\varpi_F^{n+1}$.
By Corollary~\ref{cor:additive-independence} and
Lemma~\ref{lem:residual-normalization}, we may scale $\psi_F$ by a unit
so that the residual character \eqref{eq:residual-additive-character} is
\[
 \psi_0(\bar x)=\psi_F(x/\gamma_F)
 =\exp\!\left(\frac{2\pi i}{p}\Tr_{k/\F_p}(\bar x)\right),
\]
where $x\in\cO_F$ is any lift of $\bar x\in k$.
Frobenius invariance of $\Tr_{k/\F_p}$ gives
$\psi_0(\bar x^p)=\psi_0(\bar x)$.

Let $f(T)$ be the minimal polynomial of $\varpi_K$, put
$u_f=f'(\varpi_K)/\varpi_K^{(p-1)(t+1)}\in\cO_K^\times$,
let $\sigma$ generate $\Gal(K/F)$, and define $\lambda$ by
\eqref{eq:critical-ramification-coordinate}. Let $\tau$ generate $S(K/F)$.
Choose an admissible
$\gamma_\tau$ for $\tau$ and a representative $b_\tau$ of its
stationary numerator class such that
\[
 \tau(1+x)=\psi_F(b_\tau x/\gamma_\tau)
 \qquad(x\in\p_F^{\lceil(t+1)/2\rceil}),
\]
and put $c_\tau=\gamma_\tau/b_\tau$.
\begin{lemma}[Endpoint annihilator and discriminant congruence]
\label{lem:endpoint-discriminant-congruence}
With the notation above,
\[
 \left(\overline{\frac{c_\tau}{\varpi_F^t\gamma_F}}\right)^{p-1}
 =\lambda^{p(p-1)}.
\]
Moreover, one has
\[
 \bar u_f=(-1)^p\lambda^{p-1},\qquad
 \overline{\Nm_{K/F}(u_f)}=(-1)^p\lambda^{p(p-1)}.
\]
\end{lemma}

\begin{proof}
Put $P=P_\lambda$ as in \eqref{eq:critical-norm-polynomial-extracted}.
By Proposition~\ref{prop:critical-norm-polynomial}, the annihilator of
$\operatorname{im}P$ under the trace pairing is $\lambda^{-p}\F_p$.
Let $z\in k$, choose a lift $\widetilde z\in\cO_F$,
and put $x=\varpi_F^t\widetilde z$.  Since
$\tau(1+x)=\psi_F(x/c_\tau)$ and, for a lift
$y\in\cO_F$ of $\bar y\in k$,
$\psi_0(\bar y)=\psi_F(y/\gamma_F)$, the residual coefficient of
$\tau$ is
\[
 a_\tau=
 \overline{\frac{\varpi_F^t\gamma_F}{c_\tau}}.
\]
The character $\tau$ is trivial on the norm image.  Hence
$\psi_0(a_\tau P(z))=1$ for every $z\in k$, so
$a_\tau\in\lambda^{-p}\F_p$.  It is nonzero because $\tau$ has
conductor $t+1$; consequently
$a_\tau=\lambda^{-p}u$ for some $u\in\F_p^\times$.  Taking
inverses and then $(p-1)$st powers gives
\[
 \left(\overline{\frac{c_\tau}
 {\varpi_F^t\gamma_F}}\right)^{p-1}
 =(\lambda^p u^{-1})^{p-1}
 =\lambda^{p(p-1)},
\]
which proves the first assertion.

For the second assertion, write
\[
 f'(\varpi_K)=\prod_{j=1}^{p-1}
 (\varpi_K-\sigma^j\varpi_K).
\]
After division by $\varpi_K^{(p-1)(t+1)}$, the residues of the factors
are the nonzero elements $-j\lambda$ of the ramification line.
Their product is $(-1)^p\lambda^{p-1}$.  Taking the norm to $F$ raises
this residue to the $p$th power and gives the last identity.
\end{proof}

\begin{proposition}[Wild conductor-one comparison]
\label{prop:wild-conductor-one}
If $m_F(\chi_F)=1$, then the First Main identity
\eqref{eq:first-main} holds.
\end{proposition}
\begin{proof}
By \eqref{eq:minimal-critical}, $m_K(\chi_K)=1$.
Put
\begin{equation}\label{eq:endpoint-gamma-K}
 \gamma_K=\gamma_F\frac{f'(\varpi_K)}{\varpi_K^{p-1}}.
\end{equation}
The different formula $v_K(f'(\varpi_K))=(p-1)(t+1)$ and the formula for $n_K(\psi_K)$ above give
$v_K(\gamma_K)=1+n_K(\psi_K)$; hence $\gamma_K$ is admissible for
$(\chi_K,\psi_K)$.  For $0\le i\le p-1$, partial fractions for
$T^i/f(T)$ at infinity give
\begin{equation*}
 \Tr_{K/F}\!\left(\frac{\varpi_K^i}{f'(\varpi_K)}\right)
 =\begin{cases}0,&i<p-1,\\1,&i=p-1.\end{cases}
\end{equation*}
For $x\in\cO_K$, choose $a\in\cO_F$ with $\bar a=\bar x\in k$.
Then $x-a\in\p_K$, and the trace-ideal formula \eqref{eq:trace-ideal}
shows that
\[
 \Tr_{K/F}\!\left(\frac{(x-a)\varpi_K^{p-1}}
                         {f'(\varpi_K)}\right)\in\p_F.
\]
The identity for $i=p-1$ therefore gives
\[
 \Tr_{K/F}\!\left(\frac{x\varpi_K^{p-1}}{f'(\varpi_K)}\right)
 \equiv a\pmod{\p_F}.
\]
Since $v_F(\gamma_F)=n+1$ and $\psi_F$ is trivial on $\p_F^{-n}$,
\[
 \psi_K(x/\gamma_K)=\psi_F(a/\gamma_F)=\psi_0(\bar x).
\]
Thus the residual additive character determined by $\gamma_K$ is
$\psi_0$.  If $\bar\chi$ denotes the residue-field
character of $\chi_F$, then that of $\chi_K$ is
$x\mapsto\bar\chi(x^p)$.  Since $\psi_0(x^p)=\psi_0(x)$ and
$x\mapsto x^p$ permutes $k^\times$,
\[
 \sum_{x\in k^\times}\psi_0(x)\bar\chi(x^p)^{-1}
 =\sum_{x\in k^\times}\psi_0(x)\bar\chi(x)^{-1}.
\]
Then by \eqref{eq:delta-finite-realization},  
\begin{equation}\label{eq:endpoint-upper-ratio}
 \frac{\Delta_K(\chi_K,\psi_K)}
      {\Delta_F(\chi_F,\psi_F)}
 =\chi_F\!\left(\frac{\Nm_{K/F}(\gamma_K)}{\gamma_F}\right).
\end{equation}

Using \eqref{eq:endpoint-gamma-K} and $\Nm_{K/F}(\varpi_K)=\varpi_F$, we obtain
\[
 \frac{\Nm_{K/F}(\gamma_K)}{\gamma_F}
   =(\gamma_F\varpi_F^t)^{p-1}\Nm_{K/F}(u_f).
\]
Lemma~\ref{lem:endpoint-discriminant-congruence} gives the congruence
\begin{equation*}
 \frac{\Nm_{K/F}(\gamma_K)}{\gamma_F}
 \equiv(-1)^p c_\tau^{p-1}\pmod{U_F^1}.
\end{equation*}

Finally, Lemma~\ref{lem:stationary-class-calculus}(c) shows that
$jb_\tau$ represents the stationary numerator class of $\tau^j$ on
$\p_F^{\lceil(t+1)/2\rceil}$.  Its conductor remains $t+1$ because $\tau^j$ is a
nontrivial norm character.  Choose this representative.  The stable-twist formula  \eqref{eq:stable-twist-new},
applied with $\theta=\tau^j$ and $\nu=\chi_F$, gives
\[
 \Delta_F(\tau^j\chi_F,\psi_F)
 =\chi_F\!\left(\frac{c_\tau}{j}\right)
  \Delta_F(\tau^j,\psi_F).
\]
The product of the Teichm\"uller representatives of
$\F_p^\times$ is $-1$ for odd $p$ and is $1$ for $p=2$; hence
\[
 \prod_{j\in\F_p^\times}\frac{c_\tau}{j}
 =(-1)^p c_\tau^{p-1}.
\]
Combining this product with the congruence for $\Nm_{K/F}(\gamma_K)/\gamma_F$ above and \eqref{eq:endpoint-upper-ratio}, we obtain the First Main identity \eqref{eq:first-main} in the
conductor-one case.
\end{proof}

\section{Wild odd-prime extensions: the non-stable range}\label{sec:nonstable-odd}

Let $K/F$ be wildly ramified cyclic of prime degree $p>2$, with
lower break $t$ and common residue field $k$. Put $T=t+1$ and
$m=m_F(\chi_F)$, and assume $1<m<2T$.
Retain the minimal-conductor choice of $\chi_F$.  Recall from
\eqref{eq:finite-additive-parametrization} that
\[
 \psi_{\rm can}(z)=\exp\!\left(\frac{2\pi i}{p}
             \Tr_{k/\F_p}(z)\right).
\]
By Corollary~\ref{cor:additive-independence} and
Lemma~\ref{lem:residual-normalization}, normalize the residual
additive character once and for all by $\psi_0=\psi_{\rm can}$.
In particular $\psi_0(z^p)=\psi_0(z)$.  For the quadratic phase of
Lemma~\ref{lem:quadratic-phase-new}, its magnitude assertion gives,
when $A\ne0$ and $B\in k$,
\begin{equation*}
 Q_{\psi_0}(A,B)=
 |k|^{-1/2}\sum_{x\in k}\psi_0(Ax^2/2+Bx).
\end{equation*}
For later use put
$q_0=Q_{\psi_0}(1,0)=|k|^{-1/2}\sum_{x\in k}\psi_0(x^2/2)$.

\subsection{Critical-coordinate transport and phase cancellation}

\begin{lemma}[Powers of the basic quadratic phase]
\label{lem:q0-powers}
One has
\begin{equation}\label{eq:q0-powers}
 q_0^2=\nu(-1),\qquad
 q_0^{p-1}=\nu(-1),\qquad
 q_0^p=q_0\nu(-1).
\end{equation}
\end{lemma}

\begin{proof}
The first identity is \eqref{eq:quadratic-phase-new}.  In
particular $q_0^4=1$.  If $p\equiv1\pmod4$, then $-1$ is a square in
$\F_p\subset k$, so $\nu(-1)=1$, and $p-1$ is divisible by $4$;
hence $q_0^{p-1}=1=\nu(-1)$.  If $p\equiv3\pmod4$, then
$p-1\equiv2\pmod4$, and therefore
$q_0^{p-1}=q_0^2=\nu(-1)$.  Multiplication by $q_0$ gives the last
identity.
\end{proof}

Let $E$ be a nonarchimedean local field with residue field $k$,
let $\psi_E$ be a nontrivial additive character, and let $\theta$
have conductor $m_E(\theta)>1$.  Choose an admissible denominator
$\gamma$ and a stationary representative $\beta$.
Put $d=\lfloor m_E(\theta)/2\rfloor$, choose $\delta$ with
$v_E(\delta)=d$, and write $H_\theta$ for
\eqref{eq:Lamprecht-Hasse-function} with $\psi=\psi_E$.
If $m_E(\theta)=2d+1$, define
$A_\theta\in k^\times$ by
\[
 \psi_E(\beta\delta^2\widetilde z/\gamma)=\psi_0(A_\theta z)
 \qquad(z\in k),
\]
where $\widetilde z\in\cO_E$ lifts $z$.
Since $\psi_0=\psi_{\rm can}$, existence and uniqueness follow from
\eqref{eq:Lamprecht-residual-character} and the nondegeneracy of the
trace pairing $k\times k\to\F_p$.  Lemma~\ref{lem:quadratic-normal-form} therefore gives a unique
$B_\theta\in k$ with
\begin{equation}\label{eq:critical-pair-definition}
 H_\theta(x)=\psi_0(A_\theta x^2/2+B_\theta x)
 \qquad(x\in k).
\end{equation}
If $m_E(\theta)=2d$, equation~\eqref{eq:Lamprecht-linearization} gives
\[
 H_\theta(x)=1\qquad(x\in k).
\]
Thus \eqref{eq:stationary-factor-definitions} and the definition of $Q_{\psi_0}$ give
\begin{equation*}
 \Delta_{3,E}(\theta)=
 \begin{cases}
  Q_{\psi_0}(A_\theta,B_\theta),&m_E(\theta)\text{ odd},\\
  1,&m_E(\theta)\text{ even}.
 \end{cases}
\end{equation*}

Let $K/F$ be wildly ramified cyclic of odd prime degree $p$, with
lower break $t$, and put $T=t+1$ and $k=k_F=k_K$.
Choose a uniformizer $\varpi_K$, put $\varpi_F=\Nm_{K/F}(\varpi_K)$,
choose a generator $\sigma$ of $\Gal(K/F)$, and recall from
\eqref{eq:critical-ramification-coordinate} that
\[
 \lambda=
 \overline{\frac{\sigma(\varpi_K)-\varpi_K}
                  {\varpi_K^{t+1}}}\in k^\times.
\]
For $z\in k$, write $\widetilde z\in\cO_F$ for a lift.
Proposition~\ref{prop:critical-norm-polynomial} supplies the norm
polynomial \eqref{eq:critical-norm-polynomial-extracted} and trace
coefficient \eqref{eq:critical-norm-trace-coefficient} in these
coordinates.
For $w\in\cO_K$ this trace coefficient gives
\begin{equation}\label{eq:odd-break-trace-unit}
 \Tr_{K/F}(w\varpi_K^t\widetilde z)
 \equiv\varpi_F^t\widetilde{(-\lambda^{p-1}\bar w z)}
 \pmod{\p_F^{t+1}}\qquad(z\in k),
\end{equation}
because $w-\widetilde{\bar w}\in\p_K$ and
$\Tr_{K/F}(\p_K^{t+1})=\p_F^{t+1}$ by \eqref{eq:trace-ideal}.

\begin{lemma}[Normalization of the norm-character generator]
\label{lem:odd-break-line-normalization}
A generator $\tau$ of $S(K/F)$ can be chosen so that
\begin{equation*}
 \tau(1+\varpi_F^t\widetilde z)
 =\psi_0(\lambda^{-p}z)\qquad(z\in k).
\end{equation*}
Put $\eta_0=\lambda^{-p}$. If $T$ is odd, the Hasse function of $\tau$ with critical element
$\delta_\tau=\varpi_F^{(T-1)/2}$ has quadratic coefficient $\eta_0$.
\end{lemma}

\begin{proof}
By Corollary~\ref{cor:norm-characters}, a generator $\tau$ has
conductor $T=t+1$.  Hence its restriction to
$U_F^t/U_F^{t+1}\simeq(k,+)$ is a nontrivial additive character.
Since $\psi_0=\psi_{\rm can}$, the nondegeneracy of the trace pairing
gives a unique $a\in k^\times$ such that
$\tau(1+\varpi_F^t\widetilde z)=\psi_0(az)$.
By Proposition~\ref{prop:critical-norm-polynomial}, the annihilator
of the image in \eqref{eq:critical-norm-image} is
$\lambda^{-p}\F_p$.  Since $\tau$ is trivial on norms,
$a=j\lambda^{-p}$ for some $j\in\F_p^\times$.
Replacing $\tau$ by $\tau^{j^{-1}}$ gives the required normalization of $\tau$.

Assume $T$ is odd.  For the stationary representative $\beta_\tau$
and admissible denominator $\gamma_\tau$ used to form $H_\tau$,
the equality $\delta_\tau^2=\varpi_F^t$ gives
\[
 \frac{H_\tau(x+y)}{H_\tau(x)H_\tau(y)}
 =\psi_F\!\left(
   \frac{\beta_\tau\delta_\tau^2\widetilde x\widetilde y}
        {\gamma_\tau}\right)
 =\tau(1+\varpi_F^t\widetilde{x y})
 =\psi_0(\eta_0xy),
\]
by \eqref{eq:Hasse-functional-equation},
\eqref{eq:Lamprecht-linearization}, and the normalization of $\tau$ just chosen.  Thus the quadratic coefficient is
$\eta_0$.
\end{proof}

\begin{lemma}[Changes of critical coordinate]
\label{lem:critical-polynomial-coordinate}
Let $H_\theta$ have coefficient pair $(A_\theta,B_\theta)$ in
\eqref{eq:critical-pair-definition}.  If a critical element $\delta$ is
replaced by $\widetilde a\delta$, where $a\in k^\times$ and
$\widetilde a$ lifts $a$, the pair becomes
$(a^2A_\theta,aB_\theta)$.
Precomposing $H_\theta$ with $x\mapsto x^p$ replaces the pair by
$(A_\theta^{1/p},B_\theta^{1/p})$; precomposing with
$x\mapsto x^{1/p}$ replaces it by $(A_\theta^p,B_\theta^p)$.
Each substitution preserves $\sum_{x\in k}H_\theta(x)$.
\end{lemma}

\begin{proof}
Equation~\eqref{eq:Lamprecht-Hasse-function} replaces $H_\theta(x)$
by $H_\theta(ax)$ when $\delta$ is replaced by
$\widetilde a\delta$.  The other two rules follow from
\[
 \psi_0(Ax^{2p}/2+Bx^p)
 =\psi_0(A^{1/p}x^2/2+B^{1/p}x),
\]
since $\psi_0(z^p)=\psi_0(z)$.  Multiplication by $a$ and
Frobenius are bijections of $k$.
\end{proof}

\begin{lemma}[Symmetric terms of degree at least three at the $\chi_K$-critical depth]
\label{lem:odd-upper-higher-vanishing}
Assume that $K/F$ is wildly ramified cyclic of odd prime degree $p$,
with break $t$, and put $T=t+1$.  Suppose
$m=m_F(\chi_F)=2d+1>1$, set $a=T-m$, and suppose that
$m_K(\chi_K)$ is odd, say $m_K(\chi_K)=2d_K+1$, with
\[
 2d_K=
 \begin{cases}
  m-1,&m<T,\\
  t,&m=T,\\
  t+p(m-T),&m>T.
 \end{cases}
\]
Let $u\in K^\times$ satisfy $v_K(u)=a$, put $\mathcal N=\Nm_{K/F}(u)$, and let
$x\in\p_K^{d_K}$.  Then, for every $3\le j\le p-1$,
\begin{equation}\label{eq:odd-upper-higher-terms-vanish}
 s_j(ux)\in\p_F^T,
 \qquad
 \mathcal N\,s_j(x)\in\p_F^T.
\end{equation}
Moreover
\begin{equation}\label{eq:odd-upper-H-quadratic-reduction}
 \sum_{j=2}^{p}\bigl(s_j(ux)-\mathcal N\,s_j(x)\bigr)
 \equiv s_2(ux)-\mathcal N\,s_2(x)\pmod{\p_F^T}.
\end{equation}
\end{lemma}

\begin{proof}
Write $\Nm=\Nm_{K/F}$.
Since $v_F(\mathcal N)=v_K(u)=a$ by \eqref{eq:symmetric-bound-last},
Equation~\eqref{eq:symmetric-bound} gives
\begin{align*}
 v_F(s_j(ux))
 &\ge
 \left\lfloor
  \frac{j(d_K+a)+(p-1)T}{p}\right\rfloor,\\
 v_F(\mathcal N\,s_j(x))
 &\ge a+
 \left\lfloor
  \frac{jd_K+(p-1)T}{p}\right\rfloor.
\end{align*}
It is therefore enough to prove
$j(d_K+a)\ge T, \qquad jd_K+pa\ge T$.
If $p=3$, the range $3\le j\le p-1$ is empty and
$s_3(ux)=\Nm(u)\Nm(x)=\mathcal N\,s_3(x)$, so both assertions of the lemma follow
immediately.  We may therefore suppose $p\ge5$.  The formulas below
show in particular that $d_K+a>0$; hence both left sides increase with
$j$, and it suffices to take $j=3$.
For $m<T$, one has $a=T-m>0$ and $d_K=(t-a)/2$; for $m>T$, put
$q=m-T=-a>0$, so $d_K=(t+pq)/2$.  Substitution gives
\[
\begin{array}{c|c|c}
 &3(d_K+a)-T&3d_K+pa-T\\ \hline
 m<T&\dfrac{t+3a-2}{2}&\dfrac{t+(2p-3)a-2}{2}\\[2mm]
 m=T&\dfrac{t-2}{2}&\dfrac{t-2}{2}\\[2mm]
 m>T&\dfrac{t+3(p-2)q-2}{2}&\dfrac{t+pq-2}{2}.
\end{array}
\]
Every entry is nonnegative.  When $m=T$, one has $t=2d\ge2$;
when $m<T$, use $t\ge1$ and $a\ge1$; when $m>T$, use
$t\ge1$, $q\ge1$, and $p\ge3$.  This proves
the two required inequalities and hence
\eqref{eq:odd-upper-higher-terms-vanish}.

Finally,
$s_p(ux)=\Nm(ux)=\Nm(u)\Nm(x)=\mathcal N\,s_p(x)$,
so the degree-$p$ terms cancel.  Removing them and the terms
of degrees $3,\ldots,p-1$ proves
\eqref{eq:odd-upper-H-quadratic-reduction}.  The remaining term is
homogeneous of degree two in $x$.
\end{proof}

\begin{proposition}[Odd-prime critical-coordinate transport]
\label{lem:residual-parameter-transport}
Let $\tau$ and $\eta_0$ be as in Lemma~\ref{lem:odd-break-line-normalization}.  Put
$m=m_F(\chi_F)$ and assume $2\le m<2T$.
When $m\le T$, choose the admissible denominators and norm representatives
as in \eqref{eq:gammas-low}--\eqref{eq:low-mu}. When $m>T$, choose
them as in Proposition~\ref{prop:high-parameter-table}(d).
Put $a=T-m$ and define $A\in F^\times$, $u\in K^\times$, and
$\mathcal N\in F^\times$ by
\begin{equation}\label{eq:A-u}
 \begin{gathered}
 (A,u,\mathcal N)=
 \begin{cases}
 (\alpha/\Gamma_\tau,\ \widetilde{\xi}\widetilde{\beta}/\widetilde{\alpha},\ \xi\beta/\alpha),&m\le T,\\[1mm]
 (\alpha/\gamma_F,\ \widetilde{\beta}/\widetilde{\alpha},\ \beta/\alpha),&m>T,
 \end{cases}\\
 \mathcal N=\Nm_{K/F}(u),\qquad v_K(u)=v_F(\mathcal N)=a.
 \end{gathered}
\end{equation}
Put $d_0=\overline{\varpi_F^{-a}\mathcal N}\in k^\times$.
Choose the stationary representative
\[
 \beta_K=
 \begin{cases}
  \displaystyle\frac{\beta\xi}{\widetilde{\xi}}
       -\frac{\widetilde{\beta}\alpha}{\widetilde{\alpha}},&m\le T,\\[2mm]
  \displaystyle\beta-\frac{\widetilde{\beta}\alpha}{\widetilde{\alpha}},&m>T.
 \end{cases}
\]
When $T$ is odd, put $\delta_\tau=\varpi_F^{(T-1)/2}$ and define
$B_\tau\in k$ by
\[
 \begin{aligned}
 H_\tau(x)&=\psi_F(A\delta_\tau\widetilde x)
              \tau(1+\delta_\tau\widetilde x)^{-1}\\
 &=\psi_0(\eta_0x^2/2+B_\tau x)\qquad(x\in k).
 \end{aligned}
\]

When $m$ is odd, let $(A_{\chi_F},B_{\chi_F})$ be the coefficient pair
of $H_{\chi_F}$ with representative $\beta$, denominator $\gamma_F$, and critical element
$\delta_\chi=\varpi_F^{(m-1)/2}$.  Then
$A_{\chi_F}=\eta_0d_0$.

The conductor $m_K(\chi_K)$ has the same parity as $m$.  When $m$ is odd, write $m_K(\chi_K)=2d_K+1$, take
$\delta_K=\varpi_K^{d_K}$, and define
\begin{equation}\label{eq:odd-upper-coordinate-reindexing}
 H_{\chi_K}(X)=
 \psi_K\!\left(
   \frac{\beta_K\delta_K\widetilde{X^{1/p}}}{\gamma_K}\right)
 \chi_K(1+\delta_K\widetilde{X^{1/p}})^{-1},
 \qquad X\in k.
\end{equation}
Let $(A_K,B_K)$ be the coefficient pair of this function in
\eqref{eq:critical-pair-definition}.  Then
\begin{equation}\label{eq:odd-transport-direct-pair}
 (A_K,B_K)=
 \begin{cases}
 (A_{\chi_F},B_{\chi_F}),&m<T,\\[1mm]
 (-\eta_0(d_0^p-d_0),\ B_{\chi_F}-d_0B_\tau),&m=T,\\[1mm]
 (-d_0^{p-1}A_{\chi_F},0),&m>T.
 \end{cases}
\end{equation}
\end{proposition}

\begin{proof}
Write $\Nm=\Nm_{K/F}$.  For $z\in k$, write $\widetilde z\in\cO_F$
for a lift, also viewed in $\cO_K$.  Put
\[
 u_0=\varpi_K^{-a}u,\qquad
 \mathcal N_0=\varpi_F^{-a}\mathcal N,\qquad
 \zeta=\bar u_0.
\]
Propositions~\ref{prop:low-parameter-table} and
\ref{prop:high-parameter-table}(d), in the respective ranges, give
$\Nm(u)=\mathcal N$ and $v_K(u)=v_F(\mathcal N)=a$.  Hence $u_0,\mathcal N_0$ are units and
$d_0=\overline{\mathcal N_0}$.  Since the extension is totally ramified, the
residue norm is Frobenius; hence
$d_0=\overline{\mathcal N_0}=\bar u_0^{\,p}=\zeta^p$.
There is no additional residue scalar when $\mathcal N$ is viewed as an element
of $K$.  Indeed, if $\sigma$ generates $\Gal(K/F)$, then
\[
 \frac{\varpi_F}{\varpi_K^p}
 =\prod_{i=0}^{p-1}\frac{\sigma^i(\varpi_K)}{\varpi_K}
 \equiv1\pmod{\p_K},
\]
because the extension is wildly ramified.  Hence, after division by
$\varpi_K^{pa}$, the initial residue of
$\mathcal N=\varpi_F^a \mathcal N_0$ is $d_0$ for every integer $a$, including $a<0$.

For the
representatives in Propositions~\ref{prop:high-parameter-table} and
\ref{prop:low-parameter-table}, the stationary ratios are
\begin{equation}\label{eq:odd-three-stationary-ratios}
 \frac{\beta(\tau)}{\gamma_\tau}=A,
 \qquad
 \frac{\beta(\chi_F)}{\gamma_F}=A \mathcal N,
 \qquad
 \frac{\beta_K}{\gamma_K}=A(\mathcal N-u).
\end{equation}
The normalization of $\tau$ in
Lemma~\ref{lem:odd-break-line-normalization} gives
\begin{equation*}
 \psi_F(A\varpi_F^t\widetilde z)
 =\psi_0(\eta_0z)\qquad(z\in k).
\end{equation*}

Suppose first that $m_F(\chi_F)=2d+1$.  Take
$\delta_\chi=\varpi_F^d$.  Since $a+2d=t$, the polar quotient of the
$\chi_F$ critical function is, by \eqref{eq:Hasse-functional-equation} and
\eqref{eq:odd-three-stationary-ratios},
\begin{align*}
 (X,Y)&\longmapsto
 \psi_F(A \mathcal N \delta_\chi^2\widetilde X\widetilde Y)\\
 &=\psi_F(A\varpi_F^t \mathcal N_0
                 \widetilde X\widetilde Y)
 =\psi_0(\eta_0\overline{\mathcal N_0}XY).
\end{align*}
Since $d_0=\overline{\mathcal N_0}$, this gives
$A_{\chi_F}=\eta_0d_0$.

To compute $A_K$, use the trace coefficient
\eqref{eq:odd-break-trace-unit}.  The parity of $m_K(\chi_K)$ is the parity of
$m_F(\chi_F)$: in the low range the conductors are equal, while in the
high range
$m_K(\chi_K)-1=t+p(m-T)\equiv m-1\pmod2$.
Assume this parity is odd and write $m_K(\chi_K)=2d_K+1$.  Put
$r=v_K(\mathcal N-u)$.  The exact conductor of $\chi_K$ implies $r=0$ at the
boundary; away from the boundary the two terms have distinct
valuations.  Hence
\begin{equation*}
 \begin{array}{c|c|c}
  &r&2d_K\\ \hline
  m<T&T-m&m-1\\
  m=T&0&t\\
  m>T&p(T-m)&t+p(m-T).
 \end{array}
\end{equation*}
In every row,
$r+2d_K=t$.
Take $\delta_K=\varpi_K^{d_K}$ and put
$w_0=\varpi_K^{-r}(\mathcal N-u)\in\cO_K^\times$.  In the natural residue coordinate $Z$ for $\chi_K$, define
\[
 H_K^{\mathrm{nat}}(Z)=
 \psi_K\!\left(\frac{\beta_K\delta_K\widetilde Z}{\gamma_K}\right)
 \chi_K(1+\delta_K\widetilde Z)^{-1}.
\]
For $Z,W\in k$, equations~\eqref{eq:Hasse-functional-equation} and
\eqref{eq:odd-three-stationary-ratios} give
\begin{align*}
 \frac{H_K^{\mathrm{nat}}(Z+W)}
      {H_K^{\mathrm{nat}}(Z)H_K^{\mathrm{nat}}(W)}
 &=\psi_F\!\left(
   A\Tr_{K/F}\bigl((\mathcal N-u)\delta_K^2
                  \widetilde Z\widetilde W\bigr)\right).
\end{align*}
Since $r+2d_K=t$, the argument of the trace is
$w_0\varpi_K^t\widetilde Z\widetilde W$.  Formula
\eqref{eq:odd-break-trace-unit} and the fact that
$x\mapsto\psi_F(A x)$ is trivial on $\p_F^{t+1}$ therefore
give
\begin{align*}
 \frac{H_K^{\mathrm{nat}}(Z+W)}
      {H_K^{\mathrm{nat}}(Z)H_K^{\mathrm{nat}}(W)}
 &=\psi_0\!\left(
   -\eta_0\lambda^{p-1}\bar w_0ZW\right)\notag\\
 &=\psi_0\!\left(
   \lambda^{-1}\operatorname{in}(u-\mathcal N)ZW\right),
\end{align*}
where
$\operatorname{in}(y)=\overline{\varpi_K^{-v_K(y)}y}$.
The last equality uses $\eta_0=\lambda^{-p}$ and
$\operatorname{in}(u-\mathcal N)=-\bar w_0$.

By \eqref{eq:odd-upper-coordinate-reindexing},
$H_{\chi_K}(X)=H_K^{\mathrm{nat}}(X^{1/p})$.  Since Frobenius
permutes $k$,
\[
 \sum_{X\in k}H_{\chi_K}(X)
 =\sum_{Z\in k}H_K^{\mathrm{nat}}(Z).
\]
By
Lemma~\ref{lem:critical-polynomial-coordinate}, applied to the expression
\eqref{eq:critical-pair-definition}, the quadratic coefficient in this coordinate is the $p$th power of the coefficient
in the displayed polar-quotient calculation.  Thus
\begin{equation}\label{eq:upper-A-leading}
 \frac{A_K}{\eta_0}=\operatorname{in}(u-\mathcal N)^p.
\end{equation}
The initial term of $u-\mathcal N$ is determined by the two valuations.  If
$m<T$, then $a>0$ and $v_K(u)=a<pa=v_K(\mathcal N)$, so it is $\zeta$.  If
$m=T$, both terms are units and it is $\zeta-d_0$.  If $m>T$, then
$a<0$ and $v_K(\mathcal N)=pa<a=v_K(u)$, so it is $-d_0$.  Consequently
\[
 A_K=\eta_0
 \begin{cases}
  \zeta^p,&m<T,\\
  (\zeta-d_0)^p,&m=T,\\
  (-d_0)^p,&m>T,
 \end{cases}
 =
 \begin{cases}
  A_{\chi_F},&m<T,\\
  -\eta_0(d_0^p-d_0),&m=T,\\
  -d_0^{p-1}A_{\chi_F},&m>T,
 \end{cases}
\]
using $d_0=\zeta^p$,
$A_{\chi_F}=\eta_0d_0$, and the fact that $p$ is odd.

It remains to determine the affine coefficient.  We do this from an
identity of critical functions, retaining the
stationary ratios \eqref{eq:odd-three-stationary-ratios}.

Let
\[
 P(y)=\Nm(1+y)-1,
 \qquad
 R(y)=P(y)-\Tr(y)=\sum_{j=2}^{p}s_j(y).
\]
For $X\in k$, let $Z=X^{1/p}$, choose a lift $\widetilde Z$, and put
$x=\delta_K\widetilde Z$.  Thus $X$ is the displayed coordinate for $\chi_K$ in \eqref{eq:odd-upper-coordinate-reindexing}.  Equation~\eqref{eq:stationary-parameter-ambiguity} gives
$P(x)\in\p_F^{\lfloor m_F(\chi_F)/2\rfloor}$.
In all three ranges $a+d_K\ge\lfloor T/2\rfloor$; since
$\lfloor T/2\rfloor\le t$, Theorem~\ref{thm:graded-norm}(b) gives
$P(ux)\in\p_F^{\lfloor T/2\rfloor}$.
For each odd conductor, use the fixed critical elements
$\delta_\chi=\varpi_F^{(m_F(\chi_F)-1)/2}$ and
$\delta_\tau=\varpi_F^{(T-1)/2}$, respectively, and define
\[
 r_\chi(X)=\overline{P(x)/\delta_\chi},\qquad
 r_\tau(X)=\overline{P(ux)/\delta_\tau}.
\]
If $T$ is even, set $r_\tau=0$.  Then
\begin{equation}\label{eq:odd-exact-critical-transport}
 H_{\chi_K}(X)
 =H_{\chi_F}(r_\chi(X))\,
  H_\tau(r_\tau(X))^{-1}\,
  \psi_F\!\left(A\{R(ux)-\mathcal N R(x)\}\right).
\end{equation}
To verify this identity, first note that
$1+P(ux)=\Nm(1+ux)$ and hence
$\tau(1+P(ux))=1$.  Therefore the right-hand side of
\eqref{eq:odd-exact-critical-transport} is
\begin{align*}
 &\psi_F(A\mathcal N P(x))\chi_F(1+P(x))^{-1}
  \psi_F(-A P(ux))
  \psi_F\!\left(A\{R(ux)-\mathcal N R(x)\}\right)\\
 &\qquad=
 \psi_F\!\left(A\{\mathcal N\Tr(x)-\Tr(ux)\}\right)
 \chi_F(\Nm(1+x))^{-1}\\
 &\qquad=
 \psi_K\!\left(\frac{\beta_Kx}{\gamma_K}\right)
 \chi_K(1+x)^{-1}
 =H_{\chi_K}(X),
\end{align*}
where the second equality uses $P=\Tr+R$ and the third uses
\eqref{eq:odd-three-stationary-ratios}.

We compute $r_\chi(X)$ and $r_\tau(X)$.  Only odd conductors need be
considered.  If $m<T$, write $m=2d+1$, so $2d<t$ and
$m_K(\chi_K)=m$.  Here $\delta_K=\varpi_K^d$ and
$\delta_\chi=\varpi_F^d$.  For
$x=\delta_K\widetilde{X^{1/p}}$, the norm term is
$\Nm(x)=\delta_\chi\widetilde X \pmod{\p_F^{d+1}}$,
whereas Equation~\eqref{eq:symmetric-bound} gives
$v_F(s_j(x))\ge d+1$ for $1\le j<p$.  Thus $r_\chi(X)=X$ in
the coordinate \eqref{eq:odd-upper-coordinate-reindexing}.  If the
conductor $T=t+1$ of $\tau$ is odd, put $d_\tau=t/2$.  Since
$v_K(u)=T-m$,
$v_K(ux)=d+T-m=t-d\ge d_\tau+1$,
and the same symmetric-function estimate puts every term of $P(ux)$
in $\p_F^{d_\tau+1}$.  For even $T$, $H_\tau=1$. In either case $r_\tau(X)=0$.

If $m=T$ and $T$ is odd, then $m=T=2d+1$ and
$v_K(u)=0$.  In this case
\[
 \delta_K=\varpi_K^d,
 \qquad \delta_\chi=\delta_\tau=\varpi_F^d.
\]
Modulo $\p_F^{d+1}$,
\[
 P(x)\equiv \Nm(x)=\delta_\chi\widetilde X,
 \qquad
 P(ux)\equiv \Nm(u)\Nm(x)=\delta_\tau\widetilde{d_0X}.
\]
Hence $r_\chi(X)=X$ and $r_\tau(X)=d_0X$ with no further coordinate
choice.

Finally suppose $m>T$ and write again $m=2d+1$.  The critical depth
for $\chi_K$ is
$d_K=\frac{t+p(m-T)}2$.
The norm term satisfies $\Nm(x)\in\p_F^{d+1}$ because
$d_K-d=\frac{(p-1)(m-T)}2\ge1$.
For the remaining symmetric terms, the smallest lower bound occurs at
$j=1$, and
$2\{d_K+(p-1)T-p(d+1)\}=(p-1)t-2\ge0$.
Hence Equation~\eqref{eq:symmetric-bound} gives
$P(x)\in\p_F^{d+1}$, so $r_\chi(X)=0$.  If $T$ is odd, write
$t=2d_\tau$.  Then
\[
 v_K(ux)=d_K+T-m
 =d_\tau+(p-2)(d-d_\tau)\ge d_\tau+1,
\]
and the symmetric-coefficient estimate gives $P(ux)\in\p_F^{d_\tau+1}$.
Thus $r_\tau(X)=0$; the assertion is automatic when $T$ is even.
We have proved
\begin{equation*}
 \begin{array}{c|c|c}
  &r_\chi(X)&r_\tau(X)\\ \hline
  m<T&X&0\\
  m=T&X&d_0X\\
  m>T&0&0.
 \end{array}
\end{equation*}

The polynomial $R(ux)-\mathcal N R(x)$ has no terms of degree zero or one.
Equations~\eqref{eq:odd-upper-higher-terms-vanish} and
\eqref{eq:odd-upper-H-quadratic-reduction} show that, modulo
$\p_F^T$, its degree-$p$ terms cancel and every term of degree
$3,\ldots,p-1$ vanishes.  Since
$z\mapsto\psi_F(A z)$ is trivial on $\p_F^T$, the correction
factor in \eqref{eq:odd-exact-critical-transport} is therefore even:
replacing $X$ by $-X$ replaces
$x=\delta_K\widetilde{X^{1/p}}$ by $-x$ and leaves the surviving
quadratic term unchanged. This factor is a quotient of the critical
functions in that identity, so it has a quadratic-affine normal form
in $X$. Its evenness in odd characteristic forces the affine
coefficient to be zero.
With $A_K$ given by \eqref{eq:upper-A-leading}, substitution of the
calculated values of $r_\chi(X)$ and $r_\tau(X)$ into
\eqref{eq:odd-exact-critical-transport} gives the linear-character factor
\begin{equation*}
 L_{\chi_K}(X)=
 \begin{cases}
  \psi_0(B_{\chi_F}X),&m<T,\\[1mm]
  \psi_0\!\left((B_{\chi_F}-d_0B_\tau)X\right),&m=T,\\[1mm]
  1,&m>T.
 \end{cases}
\end{equation*}
Together with the quadratic coefficients already computed, this proves
\eqref{eq:odd-transport-direct-pair}.
\end{proof}

\begin{lemma}[Integer stationary representatives]
\label{lem:integer-stationary-representatives}
Retain the minimal-conductor choice of $\chi_F$ and the
parameter choices of Proposition~\ref{lem:residual-parameter-transport}.
For $j\in\F_p^\times$, the following are admissible denominators
and stationary representatives:
\begin{equation}\label{eq:odd-stationary-representative-table}
\begin{array}{c|c|c|c}
 \text{range}&\theta&\gamma_\theta&\beta_\theta\\ \hline
 m>T&\tau^j&\gamma_F/\alpha&[j]\\
 m>T&\tau^j\chi_F&\gamma_F&\beta+[j]\alpha\\
 m\le T&\tau^j&\Gamma_\tau&[j]\alpha\\
 m\le T&\tau^j\chi_F&\Gamma_\tau&\xi\beta+[j]\alpha.
\end{array}
\end{equation}
Replacing $[j]$ by its integer representative
$j_{\mathbf Z}\in\{1,\ldots,p-1\}$ leaves the stationary class unchanged.
For odd conductor, it also leaves the critical function
\eqref{eq:Lamprecht-Hasse-function} unchanged, with any fixed critical element.
\end{lemma}

\begin{proof}
The $[j]$-representatives of $\tau^j\chi_F$ are given by
\eqref{eq:twist-beta-intermediate} and \eqref{eq:low-beta-twists}.
For $\tau^j$, apply Lemma~\ref{lem:teichmueller-stationary} to
\eqref{eq:tau-high-linearization} or \eqref{eq:low-mu}, respectively.
By Corollary~\ref{cor:norm-characters} and
\eqref{eq:minimal-orbit-conductors}, $q:=m_F(\theta)\ge T$;
the displayed denominators satisfy $v_F(\gamma_\theta)=q+n_F(\psi_F)$.
Let $h$ be the old numerator minus the new one.  In every row,
$h/\gamma_\theta=A([j]-j_{\mathbf Z})$, where
$A$ is given by \eqref{eq:A-u} and
$v_F(A)=-T-n_F(\psi_F)$.  Thus
\eqref{eq:teichmueller-successor-depth} gives
\[
 v_F(h)\ge q-\lfloor T/2\rfloor\ge\lceil q/2\rceil.
\]
The two numerators therefore represent the same stationary class.
For odd $q$, any critical element has valuation $(q-1)/2$, so
\[
 v_F(h\delta_\theta/\gamma_\theta)
 \ge\frac{q+1}{2}+\frac{q-1}{2}-q-n_F(\psi_F)=-n_F(\psi_F).
\]
By \eqref{eq:Lamprecht-Hasse-function}, the quotient of the two
critical functions is $\psi_F(h\delta_\theta\widetilde z/\gamma_\theta)=1$
for every $z\in k$, since $\psi_F$ is trivial on $\p_F^{-n_F(\psi_F)}$.
\end{proof}

\begin{proposition}[Odd-prime coefficient table]
\label{prop:odd-coefficient-table}
Retain the minimal-conductor choice of $\chi_F$ and the
notation of Proposition~\ref{lem:residual-parameter-transport}.
For $j\in\F_p$, let $(A_j,B_j)$ be the coefficient pair of
$\tau^j\chi_F$, using the integer sum representative in
\eqref{eq:odd-stationary-representative-table} for $j\ne0$.
Each pair is taken only when its conductor is odd.  Use
$\delta_\tau=\varpi_F^{(T-1)/2}$ when $m\le T$ and $j\ne0$,
and $\delta_\chi=\varpi_F^{(m-1)/2}$ otherwise.
\begin{equation}\label{eq:lower-coeff-table}
 (A_j,B_j)=
 \begin{cases}
 (j\eta_0,jB_\tau),&m<T,\ j\ne0,\\[1mm]
 (A_{\chi_F},B_{\chi_F}),&m<T,\ j=0,\\[1mm]
 (\eta_0(j+d_0),\ B_{\chi_F}+jB_\tau),&m=T,\\[1mm]
 (A_{\chi_F},B_{\chi_F}),&m>T.
 \end{cases}
\end{equation}
The coefficient pair for $\chi_K$ is given by
\eqref{eq:odd-transport-direct-pair}.
\end{proposition}

\begin{proof}
Use the sum representatives in \eqref{eq:odd-stationary-representative-table}.
Lemma~\ref{lem:integer-stationary-representatives} permits
replacing $[j]$ by $j_{\mathbf Z}$ without changing a critical function.
Put
\[
 \Gamma=\begin{cases}\Gamma_\tau,&m\le T,\\\gamma_F,&m>T,\end{cases}
 \qquad \beta_\tau=\Gamma A,\quad
 \beta_\chi=\Gamma A\mathcal N,
\]
where $A$ is defined in \eqref{eq:A-u}.
By \eqref{eq:odd-three-stationary-ratios}, these are the coefficients
of $\tau$ and $\chi_F$ relative to $\Gamma$.  The selected coefficient
of $\tau^j\chi_F$ is $j_{\mathbf Z}\beta_\tau+\beta_\chi$.
For $x\in k$, let $\widetilde x\in\cO_F$ be a lift.

Suppose $m<T$ and $j\ne0$.  A coefficient pair is needed only when
$T$ is odd.  Put $d_\tau=(T-1)/2$ and use
$\delta_\tau=\varpi_F^{d_\tau}$, as in
Lemma~\ref{lem:odd-break-line-normalization}.  Since
$m_F(\chi_F)\le T-1$,
\[
 v_F(\delta_\tau)=\frac{T-1}{2}
 \ge\left\lceil\frac{m_F(\chi_F)}2\right\rceil.
\]
Thus the stationary linearization of $\chi_F$ is valid at
$\delta_\tau\widetilde x$.  With the sum representative just chosen,
\begin{align*}
 \frac{H_{\tau^j\chi_F}(x)}{H_{\tau^j}(x)}
 &=\psi_F\!\left(
   \frac{\beta_\chi \delta_\tau\widetilde x}{\Gamma}
  \right)
   \chi_F(1+\delta_\tau\widetilde x)^{-1}\\
 &=1.
\end{align*}
Hence the two critical functions agree.  Since
$H_{\tau^j}=H_\tau^j$, its pair is
$(j\eta_0,jB_\tau)$.  The case $j=0$ is $\chi_F$ itself and
has pair $(A_{\chi_F},B_{\chi_F})$.

Suppose $m>T$.  A coefficient pair is needed only when
$m_F(\chi_F)$ is odd.  Put
$d_\chi=(m_F(\chi_F)-1)/2$ and use
$\delta_\chi=\varpi_F^{d_\chi}$, as in
Proposition~\ref{lem:residual-parameter-transport}.  Since
$m_F(\chi_F)>T$,
\[
 v_F(\delta_\chi)=\frac{m_F(\chi_F)-1}{2}
 \ge\left\lceil\frac T2\right\rceil.
\]
The stationary linearization of every $\tau^j$ is therefore valid at
$\delta_\chi\widetilde x$, and for every $j\in\F_p$,
\begin{align*}
 \frac{H_{\tau^j\chi_F}(x)}{H_{\chi_F}(x)}
 &=\psi_F\!\left(
   \frac{j_{\mathbf Z}\beta_\tau \delta_\chi\widetilde x}{\Gamma}
  \right)
   \tau^j(1+\delta_\chi\widetilde x)^{-1}\\
 &=1.
\end{align*}
Thus every critical function has the pair
$(A_{\chi_F},B_{\chi_F})$.

At $m=T$, only odd $T$ requires coefficient pairs.  Both conductors
are $T$, and $\Gamma=\Gamma_\tau$ is admissible for that conductor.  Here
$\beta_\tau=\alpha$ and $\beta_\chi=\xi\beta$ are units, so
\eqref{eq:minimal-orbit-no-leading-cancellation} gives
$j_{\mathbf Z}\alpha+\xi\beta\notin\p_F$.
Use this sum representative and the common critical element
$\delta_\tau=\delta_\chi=\varpi_F^{(T-1)/2}$.  The
critical function of the product is the pointwise product of
the two critical functions, so the coefficient pairs add:
$j(\eta_0,B_\tau) +(A_{\chi_F},B_{\chi_F})$.
This proves the three cases for $(A_j,B_j)$.
\end{proof}

\begin{proposition}[Odd-prime finite phase cancellation]
\label{prop:odd-phase-cancel}
Retain the minimal-conductor choice of $\chi_F$.
The residual factors satisfy
\begin{equation}\label{eq:Rodd}
 \mathcal R_{\mathrm{odd}}:=
 \frac{\Delta_{3,K}(\chi_K)
       \prod_{j\in\F_p^\times}\Delta_{3,F}(\tau^j)}
      {\prod_{j\in\F_p}\Delta_{3,F}(\tau^j\chi_F)}=1.
\end{equation}
\end{proposition}

\begin{proof}
Write $R_j=\Delta_{3,F}(\tau^j)$,
$T_j=\Delta_{3,F}(\tau^j\chi_F)$, and $U=\Delta_{3,K}(\chi_K)$.
For odd conductor, evaluate $Q_{\psi_0}(A_\theta,B_\theta)$ by
\eqref{eq:quadratic-phase-new}; for even conductor the factor is $1$.

If $m<T$, \eqref{eq:lower-coeff-table} gives $T_j=R_j$ for
$j\ne0$, while the $j=0$ pair and the pair for $\chi_K$ agree, so $T_0=U$.

Suppose $m>T$.  If $m_F(\chi_F)$ is even, then every $T_j$ and $U$ is
$1$.  If it is odd, \eqref{eq:lower-coeff-table} gives the same phase
$Q_\chi=Q_{\psi_0}(A_{\chi_F},B_{\chi_F})$ for all $T_j$.  By \eqref{eq:quadratic-phase-new}, the additive factor has $p$th power $1$, because its
value lies in the group of $p$th roots of unity.  Hence \eqref{eq:q0-powers} gives
\[
 Q_\chi^p=q_0^p\nu(A_{\chi_F})^p
 =q_0\nu(-1)\nu(A_{\chi_F})
 =q_0\nu(-A_{\chi_F})=U,
\]
because $d_0^{p-1}$ is a square and the affine coefficient of
$\chi_K$ is zero by \eqref{eq:odd-transport-direct-pair}.
If the conductor $T=t+1$ of $\tau$ is even, then
all $R_j$ are $1$.  If it is odd, then
\begin{equation*}
\begin{aligned}
 \prod_{j\ne0}R_j
 &=q_0^{p-1}\nu(\eta_0)^{p-1}
   \nu\!\left(\prod_{j\ne0}j\right)
   \psi_0\!\left(-\frac{B_\tau^2}{2\eta_0}
                         \sum_{j\ne0}j\right)\\
 &=q_0^{p-1}\nu(-1)=1.
\end{aligned}
\end{equation*}
Here $\sum_{j\ne0}j=0$, $\prod_{j\ne0}j=-1$, and
$q_0^{p-1}=\nu(-1)$ by \eqref{eq:q0-powers}.  Hence
$U\prod R_j=\prod T_j$ in every parity
pattern above the break.

Finally suppose $m=T$.  All four kinds of character have conductor
$T$.  If $T$ is even, every residual phase is $1$.  Suppose $T$ is odd.
Put
\[
 c=A_{\chi_F}=\eta_0d_0,\qquad
 \rho=d_0^{-1},\qquad
 \sigma=\frac{B_{\chi_F}}{A_{\chi_F}},\qquad
 \upsilon=\frac{B_\tau}{A_{\chi_F}}.
\]
Here $c\ne0$.  At $m=T$, the stationary coefficients relative to
$\Gamma_\tau$ are $\alpha$ and $\xi\beta=\alpha\mathcal N$, with
$\alpha\in\cO_F^\times$ and $\overline{\mathcal N}=d_0$.
Equation~\eqref{eq:minimal-orbit-no-leading-cancellation} gives
$j+d_0\ne0$ for every $j\in\F_p$, hence $1+j\rho\ne0$.
The boundary rows of \eqref{eq:lower-coeff-table} and
\eqref{eq:odd-transport-direct-pair}, together with the pair
$(j\eta_0,jB_\tau)$ for $\tau^j$, are
\[
 \begin{aligned}
 (A_j,B_j)&=c(1+j\rho,\sigma+j\upsilon),&&j\in\F_p,\\
 (A_{\tau^j},B_{\tau^j})&=c(j\rho,j\upsilon),&&j\in\F_p^\times,\\
 (A_K,B_K)&=c(1-\rho^{1-p},\sigma-\rho^{-1}\upsilon).
 \end{aligned}
\]
Proposition~\ref{lem:affine-pencil} therefore gives
$U\prod_{j\ne0}R_j=\prod_jT_j$ directly.
Hence \eqref{eq:Rodd} also holds at $m=T$.
\end{proof}

\subsection{Evaluation of the stationary factors}

Retain the parameters $A,u,\mathcal N$ and $a=T-m$ from
\eqref{eq:A-u}.
Represent $\F_p^\times$ by its Teichm\"uller lifts.  Define
\begin{equation}\label{eq:correction-units}
 z_0=\frac{\Nm_{K/F}(1-\mathcal N/u)}{1-\mathcal N^{p-1}},
 \qquad
 z_j=\frac{\Nm_{K/F}(u+j)}{\mathcal N+j}\quad(j\ne0).
\end{equation}
If $a\ne0$, the denominators are nonzero by valuation.
If $a=0$, then $m_F(\chi_F)=T$ and $\xi\beta=\alpha \mathcal N$, with
$\alpha\in\cO_F^\times$.  Applying
\eqref{eq:minimal-orbit-no-leading-cancellation} to the coefficients
$\alpha$ and $\xi\beta$ gives
\begin{equation}\label{eq:odd-denominator-units}
 \overline{\mathcal N}+j\ne0\ (j\in\F_p),\qquad
 \mathcal N+j\in\cO_F^\times,\qquad 1-\mathcal N^{p-1}\in\cO_F^\times
 \quad(a=0).
\end{equation}
\begin{lemma}[Correction identities and depths]
\label{lem:correction-depth}\label{lem:parameter-ratios}
\label{prop:correction-linearization}
One has
\begin{equation}\label{eq:parameter-ratios}
 \Nm_{K/F}(\mathcal N-u)
 =\left(\prod_{j\in\F_p}(\mathcal N+j)\right)z_0.
\end{equation}
Moreover,
\begin{equation}\label{eq:correction-depths}
\begin{aligned}
 z_j-1&\in
 \p_F^{\left\lfloor(|a|+(p-1)T)/p\right\rfloor}
 \subseteq\p_F^{\lceil T/2\rceil} &&(j\ne0),\\
 z_0-1&\in
 \p_F^{\left\lfloor(p-1)(T+|a|)/p\right\rfloor}
 \subseteq\p_F^{\lceil m/2\rceil}.
\end{aligned}
\end{equation}
Thus $z_j-1$ lies in the linearization domain of $\tau$ and
$z_0-1$ lies in the linearization domain of $\chi_F$.
Consequently,
\begin{equation*}
 \chi_F(z_0)=\psi_F\bigl(A\mathcal N(z_0-1)\bigr),
 \qquad
 \tau^j(z_j)=\psi_F\bigl(Aj(z_j-1)\bigr)
 \quad(j\ne0).
\end{equation*}
\end{lemma}

\begin{proof}
Since $p$ is odd,
\[
 \Nm_{K/F}(1-\mathcal N/u)
 =-\frac{\Nm_{K/F}(\mathcal N-u)}{\mathcal N},
 \qquad
 1-\mathcal N^{p-1}
 =-\frac{\mathcal N^p-\mathcal N}{\mathcal N}.
\]
Hence
\[
 z_0=\frac{\Nm_{K/F}(\mathcal N-u)}{\mathcal N^p-\mathcal N}.
\]
Now \eqref{eq:teichmueller-polynomials} gives
$\mathcal N^p-\mathcal N=\prod_{j\in\F_p}(\mathcal N+j)$.

Write $\Nm=\Nm_{K/F}$ and, for $c\ge0$, put
$R(c)=\left\lfloor\frac{c+(p-1)T}{p}\right\rfloor$.
For $w\in K$ with $1+\Nm(w)\ne0$, put
$Z(w)=\Nm(1+w)/(1+\Nm(w))$.  If $v_K(w)=c>0$, then
$Z(w)-1=\frac{s_1(w)+\cdots+s_{p-1}(w)}{1+\Nm(w)}$,
so \eqref{eq:symmetric-bound} gives
$v_F(Z(w)-1)\ge R(c)$.  If $c<0$, the identity
$Z(w)=Z(w^{-1})$ reduces to the positive case.  If $c=0$, the
unit assertions \eqref{eq:odd-denominator-units} show that
$1+\Nm(w)$ is a unit in each occurrence of $Z(w)$: for $w=u/j$ this is equivalent to
$\overline{\mathcal N}+j\ne0$, and for $w=-\mathcal N/u$ to
$1-\overline{\mathcal N}^{p-1}\ne0$.
Equation~\eqref{eq:symmetric-bound} then gives
\[
 v_F\bigl(s_i(w)\bigr)\ge
 \floor{\frac{(p-1)T}{p}}=R(0)
 \qquad(1\le i\le p-1),
\]
so the same conclusion holds when $c=0$.

Now $z_j=Z(u/j)$, while $z_0=Z(-\mathcal N/u)$ and
$v_K(-\mathcal N/u)=(p-1)a$ by \eqref{eq:A-u}.  This proves
the first inclusions.

We use the elementary integer inequality
\begin{equation*}
 \floor{\frac{(p-1)M}{p}}\ge\ceil{\frac M2}
 \qquad(p\ge3,\ M\ge2).
\end{equation*}
If $M$ is even, the real inequality
$(p-1)M/p\ge M/2$ already has an integral right side.  If $M$ is odd,
then $M\ge3$, and
\[
 \frac{(p-1)M}{p}\ge\frac{M+1}{2}
 \quad\Longleftrightarrow\quad
 (p-2)M\ge p,
\]
which follows from $M\ge3$ and $p\ge3$.

Since $R(|a|)\ge R(0)$, applying
$\lfloor(p-1)M/p\rfloor\ge\lceil M/2\rceil$ with $M=T$ gives
$R(|a|)\ge\ceil{T/2}$. For $z_0$, since $m=T-a$, the same inequality gives
\[
 R((p-1)|a|)
 =\floor{\frac{(p-1)(T+|a|)}p}
 \ge\ceil{\frac{T+|a|}{2}}
 \ge\ceil{\frac m2}.
\]
This proves \eqref{eq:correction-depths}.

Finally, the linearizations of $\chi_F$ and $\tau$ have coefficients
$A\mathcal N$ and $A$, respectively. Their domains contain $z_0-1$
and $z_j-1$ by \eqref{eq:correction-depths}. Applying the linearizations
and raising the latter value to the $j$th power proves the character
identities.
\end{proof}

\begin{proposition}[Assembly of the stationary quotient]
\label{prop:exact-odd-assembly}
The error quotient \eqref{eq:error-term}
\[
 \Err_{K/F}(\chi_F)
 =\mathcal R_{\mathrm{odd}}\,\psi_F(-AX),
\]
where
\begin{equation}\label{eq:X-def}
 X=\Tr_{K/F}(u)+\mathcal N(z_0-1)+\sum_{j\ne0}j(z_j-1).
\end{equation}
\end{proposition}

\begin{proof}
Use the factorization \eqref{eq:Delta-factorization-new}, the sum
representatives in \eqref{eq:odd-stationary-representative-table}, and
the critical functions in
Proposition~\ref{prop:odd-coefficient-table}.

In the high range $m>T$, the admissible-element and
stationary-parameter pairs, with their sources, are
\[
 \begin{array}{c|c|c}
 \text{character}&(\gamma_\theta,\beta_\theta)&\text{source}\\ \hline
 \chi_K&(\gamma_F,\ \alpha(\mathcal N-u))
   &\text{Proposition~\ref{prop:high-parameter-table}(d), }\eqref{eq:Pstar-intermediate}\\
 \tau^j&(\gamma_F/\alpha,\ j)\quad(j\ne0)
   &\text{Lemma~\ref{lem:integer-stationary-representatives}}\\
 \tau^j\chi_F&(\gamma_F,\ \alpha(\mathcal N+j))\quad(j\in\F_p)
   &\eqref{eq:twist-beta-intermediate},\ \text{Lemma~\ref{lem:integer-stationary-representatives}}.
 \end{array}
\]
In the low range $m\le T$, the pairs are
\[
 \begin{array}{c|c|c}
 \text{character}&(\gamma_\theta,\beta_\theta)&\text{source}\\ \hline
 \chi_K&(\Gamma_\tau/\widetilde{\xi},\ (\alpha/\widetilde{\xi})(\mathcal N-u))
   &\text{Proposition~\ref{prop:low-parameter-table}}\\
 \tau^j&(\Gamma_\tau,\ \alpha j)\quad(j\ne0)
   &\text{Lemma~\ref{lem:integer-stationary-representatives}}\\
 \chi_F&(\Gamma_\tau/\xi,\ (\alpha/\xi)\mathcal N)
   &\eqref{eq:gammas-low},\ \eqref{eq:low-mu}\\
 \tau^j\chi_F&(\Gamma_\tau,\ \alpha(\mathcal N+j))\quad(j\ne0)
   &\eqref{eq:low-beta-twists},\ \text{Lemma~\ref{lem:integer-stationary-representatives}}.
 \end{array}
\]
The factors contributed by the admissible elements cancel.  In
the high range their quotient is
\[
 \frac{\chi_K(\gamma_F)
       \prod_{j\ne0}\tau^j(\gamma_F/\alpha)}
      {\prod_j(\tau^j\chi_F)(\gamma_F)}=1,
\]
and in the low range it is
\[
 \frac{\chi_K(\Gamma_\tau/\widetilde{\xi})\prod_{j\ne0}\tau^j(\Gamma_\tau)}
      {\chi_F(\Gamma_\tau/\xi)
       \prod_{j\ne0}(\tau^j\chi_F)(\Gamma_\tau)}=1.
\]
Here $\Nm_{K/F}(\widetilde{\xi})=\xi$ and $\alpha$ is a norm.

It remains to collect the factors $\theta(\beta)^{-1}$ from
\eqref{eq:stationary-factor-definitions}.  In the high range they give
\[
 \chi_K\!\left(\alpha(\mathcal N-u)\right)^{-1}
 \prod_{j\ne0}\tau^j(j)^{-1}
 \prod_{j\in\F_p}(\tau^j\chi_F)\!\left(\alpha(\mathcal N+j)\right).
\]
Since $\chi_K=\chi_F\circ \Nm_{K/F}$, $\Nm_{K/F}(\alpha)=\alpha^p$, and
$\tau(\alpha)=1$ because $\alpha$ is a norm, this becomes
\[
H= \chi_F\!\left(\frac{\prod_{j\in\F_p}(\mathcal N+j)}
 {\Nm_{K/F}(\mathcal N-u)}\right)
 \prod_{j\ne0}\tau^j\!\left(\frac{\mathcal N+j}{j}\right).
\]
The low-range contribution is
\[
 \chi_K\!\left(\frac{\alpha}{\widetilde{\xi}}(\mathcal N-u)\right)^{-1}
 \prod_{j\ne0}\tau^j(\alpha j)^{-1}
 \chi_F\!\left(\frac{\alpha}{\xi}\mathcal N\right)
 \prod_{j\ne0}(\tau^j\chi_F)\!\left(\alpha(\mathcal N+j)\right).
\]
Using $\Nm_{K/F}(\widetilde{\xi})=\xi$, $\chi_K=\chi_F\circ \Nm_{K/F}$, and
$\tau(\alpha)=1$, this reduces to the same expression $H$.   
By \eqref{eq:parameter-ratios}, the first ratio is $z_0^{-1}$.
Moreover
\[
 \frac{\mathcal N+j}{j}=\Nm_{K/F}\!\left(\frac{u+j}{j}\right)z_j^{-1},
\]
so the norm character kills the norm and
$H=\chi_F(z_0)^{-1}\prod_{j\ne0}\tau^j(z_j)^{-1}$.

The additive stationary values have signed sum
\[
 \Tr_{K/F}(\mathcal N-u)+\sum_{j\ne0}j-
 \sum_{j\in\F_p}(\mathcal N+j)=-\Tr_{K/F}(u),
\]
using the Teichm\"uller sum \eqref{eq:teichmueller-sum}.  Their additive factor is
therefore $\psi_F(-A\Tr_{K/F}(u))$.  Together with the expression for
$H$ and the residual quotient $\mathcal R_{\mathrm{odd}}$, this gives
\[
 \Err_{K/F}(\chi_F)
 =\mathcal R_{\mathrm{odd}}\,
 \left[
 \psi_F(A\Tr_{K/F}(u))\chi_F(z_0)
 \prod_{j\ne0}\tau^j(z_j)
 \right]^{-1}.
\]
Substituting Lemma~\ref{prop:correction-linearization} gives
$\Err_{K/F}(\chi_F)=\mathcal R_{\mathrm{odd}}\psi_F(-AX)$, with
$X$ as in \eqref{eq:X-def}.
\end{proof}

\subsection{The norm--trace formula and its valuation}

Let $s_r=s_r(u)$ be the elementary symmetric functions of the
conjugates of $u$, with $s_1=\Tr_{K/F}(u)$ and $s_p=\mathcal N$.  Put
\begin{equation*}
 B=s_1+\sum_{r=1}^{p-2}(-\mathcal N)^r s_{p-r}.
\end{equation*}

\begin{proposition}[Exact norm--trace formula]\label{prop:X-closed}
The element $X$ in \eqref{eq:X-def} satisfies  
\begin{equation}\label{eq:X-closed}
 X=\frac{pB}{1-\mathcal N^{p-1}}.
\end{equation}
\end{proposition}

\begin{proof}
Put $x_0=z_0-1$ and $x_j=z_j-1$ for $j\ne0$.
For $j\ne0$, multiplication of the conjugates of $u+j$ gives
\[
 \Nm_{K/F}(u+j)=\sum_{r=0}^{p}j^{p-r}s_r,
 \qquad s_0=1,\quad s_p=\mathcal N.
\]
Hence
\[
 x_j=\frac{\Nm_{K/F}(u+j)-(\mathcal N+j)}{\mathcal N+j}
 =\frac{\sum_{r=1}^{p-1}j^{p-r}s_r}{\mathcal N+j}.
\]
Because $j^{p-1}=1$,
\begin{equation*}
 \frac1{\mathcal N+j}=
 \frac{\sum_{q=0}^{p-2}(-\mathcal N)^qj^{p-2-q}}
      {1-\mathcal N^{p-1}}.
\end{equation*}
Indeed, multiplying the numerator by $\mathcal N+j$ telescopes to
$1-\mathcal N^{p-1}$.  Using this expansion to sum over the
$(p-1)$ Teichm\"uller roots, together with
\[
 \sum_{j\ne0}j^r=
 \begin{cases}p-1,&p-1\mid r,\\0,&p-1\nmid r,
 \end{cases}
\]
shows that only the indices $r+q=1$ and $r+q=p$ survive.
By the definition of $B$, this gives
\begin{equation*}
 \sum_{j\ne0}jx_j=\frac{(p-1)B}{1-\mathcal N^{p-1}}.
\end{equation*}

On the other hand,
\[
 \Nm_{K/F}(\mathcal N-u)-(\mathcal N^p-\mathcal N)=\sum_{r=1}^{p-1}(-1)^rs_r\mathcal N^{p-r}.
\]
By \eqref{eq:teichmueller-polynomials},
$z_0=\Nm_{K/F}(\mathcal N-u)/(\mathcal N^p-\mathcal N)$, so
\begin{align*}
 \mathcal N x_0
 &=-\frac{\sum_{r=1}^{p-1}(-1)^rs_r\mathcal N^{p-r}}
          {1-\mathcal N^{p-1}}\\
 &=\frac{s_1\mathcal N^{p-1}+\sum_{q=1}^{p-2}(-\mathcal N)^qs_{p-q}}
          {1-\mathcal N^{p-1}}\\
 &=\frac{B}{1-\mathcal N^{p-1}}-s_1.
\end{align*}
Adding $s_1$, this identity, and the formula for $\sum_{j\ne0}jx_j$ above proves
\eqref{eq:X-closed}.  This is an identity in the field $F$ itself.  In
equal characteristic the same finite sums have $p=0$ and $p-1=-1$ in
$F$, so no lifting to characteristic zero is being used.
\end{proof}

In mixed characteristic, the closed formula \eqref{eq:X-closed}
appears in \cite[Chapter~III, proof of Lemma~8, pp.~89--90]{Dwork}.

\begin{proposition}[Vanishing or depth of the norm--trace error]
\label{prop:X-depth}
Retain the minimal-conductor choice of $\chi_F$.
One has
\begin{equation}\label{eq:X-depth-final}
 X\in\p_F^T,
 \qquad \psi_F(A\p_F^T)=1.
\end{equation}
If $\operatorname{char}F=p$, then $X=0$.  In all cases
$\psi_F(-AX)=1$.
\end{proposition}

\begin{proof}
The denominator $1-\mathcal N^{p-1}$ is nonzero, as established
following \eqref{eq:correction-units}.

Suppose first that $\operatorname{char}F=p$.  In \eqref{eq:X-closed}, the integer $p$ is zero in $F$.  Therefore
$X=\frac{pB}{1-\mathcal N^{p-1}}=0$,
which proves the first assertion in equal characteristic.

Now assume $\operatorname{char}F=0$.  Put $e=v_F(p)$ and $e_0=\lfloor(p-1)T/p\rfloor=T-\lceil T/p\rceil$.
Lemma~\ref{lem:teichmueller-stationary} gives $e\ge e_0$; for odd $p$,
its proof gives $e_0\ge\lceil T/2\rceil$, hence $2e_0\ge T$.

If $a\ge0$, the denominator in \eqref{eq:X-closed} is a unit, as
shown by \eqref{eq:odd-denominator-units} when $a=0$ and by
$v_F(\mathcal N)>0$ when $a>0$.  By \eqref{eq:symmetric-bound}, every $s_j(u)$ with
$1\le j<p$ has valuation at least $e_0$, and the powers of $\mathcal N$ do not decrease it.
Hence $v_F(pB)\ge e+e_0\ge2e_0\ge T$.

If $a<0$, put $c=-a>0$. Since
$v_F(1-\mathcal N^{p-1})=(p-1)a$, division by this denominator is
multiplication, up to a unit, by $\mathcal N^{1-p}$.
Apply \eqref{eq:symmetric-bound} directly with $v_K(u)=-c$. For the
$s_1$ term it gives
\[
\begin{aligned}
 v_F\bigl(p\mathcal N^{1-p}s_1(u)\bigr)
 &\ge e+\floor{\frac{(p(p-1)-1)c+(p-1)T}{p}}\\
 &\ge e+e_0\ge T.
\end{aligned}
\]
For $2\le j<p$, the other terms satisfy
\[
\begin{aligned}
 v_F\bigl(p\mathcal N^{1-j}s_j(u)\bigr)
 &\ge e+\floor{\frac{((p-1)j-p)c+(p-1)T}{p}}\\
 &\ge e+e_0\ge T,
\end{aligned}
\]
since $(p-1)j-p\ge p-2>0$.
Thus $X\in\p_F^T$ also in mixed characteristic.

Finally $\tau$ has conductor $T$ and
$\tau(1+x)=\psi_F(Ax)$ for $x\in\p_F^{\lceil T/2\rceil}$.
Restricting to $x\in\p_F^T$ and using $\tau|_{U_F^T}=1$ gives
$\psi_F(A\p_F^T)=1$.  The conclusion follows.
\end{proof}

\begin{corollary}[Wild odd-prime non-stable case]
\label{cor:wild-odd}
For odd $p$, the First Main identity \eqref{eq:first-main} holds in every
non-stable wildly ramified case.
\end{corollary}

\begin{proof}
Propositions~\ref{prop:odd-phase-cancel} and~\ref{prop:exact-odd-assembly}, together with \eqref{eq:X-depth-final}, give the result.
\end{proof}

\section{Wild quadratic extensions}\label{sec:wild-quadratic}

Let $K/F$ be wildly ramified quadratic with lower break $t$,
$S(K/F)=\{1,\tau\}$, and common residue field $k$ of characteristic
$p=2$. Retain the minimal-conductor choice of $\chi_F$.
Put $m=m_F(\chi_F)>1$, $T=t+1$, and $a=T-m$.

\subsection{Stationary parameters and the norm correction}

Use the stationary parameters of Proposition~\ref{prop:low-parameter-table}
for $m\le T$ and Proposition~\ref{prop:wild-quadratic-high-parameters}
for $m>T$, and put
\begin{equation}\label{eq:quadratic-parameters}
\begin{gathered}
 (A,u,\mathcal N)=
 \begin{cases}
 (\alpha/\Gamma_\tau,\ \widetilde{\xi}\widetilde{\beta}/\widetilde{\alpha},\ \xi\beta/\alpha),&m\le T,\\[1mm]
 (\alpha \kappa/\gamma_F,\ \widetilde{\beta}/\widetilde{\alpha},\ \beta/(\alpha \kappa)),&m>T,
 \end{cases}\\
 \mathcal N=\Nm_{K/F}(u),\qquad v_K(u)=v_F(\mathcal N)=a.
\end{gathered}
\end{equation}
Equations~\eqref{eq:low-mu}, \eqref{eq:wild-two-tau-linearization}, and
\eqref{eq:Lamprecht-linearization} give
\begin{equation}\label{eq:quadratic-linearization-domains}
\begin{aligned}
 \tau(1+z)&=\psi_F(Az)
 &&\bigl(z\in\p_F^{\lceil T/2\rceil}\bigr),\\
 \chi_F(1+z)&=\psi_F(A\mathcal N z)
 &&\bigl(z\in\p_F^{\lceil m/2\rceil}\bigr).
\end{aligned}
\end{equation}
Define the correction units by
\begin{equation}\label{eq:quad-corrections}
 z_0=\frac{\Nm_{K/F}(1-\mathcal N/u)}{1+\mathcal N},
 \qquad z_1=\frac{\Nm_{K/F}(u+1)}{\mathcal N+1}.
\end{equation}
If $a\ne0$, the denominator is nonzero by valuation.  If $a=0$, then
$m=T$, $\xi\beta=\alpha\mathcal N$, and $\alpha$ is a unit, so
\eqref{eq:minimal-orbit-no-leading-cancellation} gives
$1+\mathcal N\in\cO_F^\times$.

\begin{lemma}[Exact quadratic norm correction]
\label{lem:quad-X}
Put $s=\Tr_{K/F}(u)$ and
$X=s+\mathcal N(z_0-1)+(z_1-1)$.  Then
\begin{equation}\label{eq:quad-zxy}
 z_0-1=-\frac{s}{1+\mathcal N},\qquad
 z_1-1=\frac{s}{1+\mathcal N},\qquad
 X=\frac{2s}{1+\mathcal N}.
\end{equation}
\end{lemma}

\begin{proof}
Since $\mathcal N/u$ is the nontrivial conjugate of $u$, the quadratic
norm identity gives
\[
 \Nm_{K/F}(1+u)=1+s+\mathcal N,\qquad
 \Nm_{K/F}(1-\mathcal N/u)=1-s+\mathcal N.
\]
Substitution in \eqref{eq:quad-corrections} proves the first two
identities, and their substitution in the definition of $X$ proves
the third.
\end{proof}

The formulas \eqref{eq:quad-zxy} occur in Dwork's $p$-adic calculation
\cite[Chapter~III, proof of Lemma~$8'$, p.~93]{Dwork}.

\begin{lemma}[Quadratic graded trace]
\label{lem:quadratic-graded-trace}
Let $K/F$ be wildly ramified quadratic and let
$T=v_K(\mathfrak D_{K/F})=t+1$.  If $r\ge0$ and $r+T$ is odd,
trace induces a $k$-linear isomorphism
\[
 \Tr_{K/F}:\p_K^r/\p_K^{r+1}
 \xrightarrow{\ \sim\ }
 \p_F^{(r+T-1)/2}/\p_F^{(r+T+1)/2}.
\]
In particular:
\begin{enumerate}[label=(\alph*)]
\item If $T$ is odd and $d_\tau=(T-1)/2$, then
$v_F(2)=d_\tau$.  For every $w\in\cO_K^\times$ and every lift
$\widetilde{\bar w}\in\cO_F$ of its residue,
\begin{equation}\label{eq:quadratic-unit-trace-class}
 \Tr_{K/F}(w)\equiv2\widetilde{\bar w}\pmod{\p_F^{d_\tau+1}}.
\end{equation}
\item If $m=2d+1>T$ and $r=m-T$, then trace identifies
$\p_K^r/\p_K^{r+1}$ with $\p_F^d/\p_F^{d+1}$.
\end{enumerate}
\end{lemma}

\begin{proof}
For a quadratic extension, \eqref{eq:trace-ideal} reads
$\Tr(\p_K^j)=\p_F^{\floor{(j+T)/2}}$.
If $r+T=2q+1$, then
\[
 \Tr(\p_K^r)=\p_F^q,
 \qquad
 \Tr(\p_K^{r+1})=\p_F^{q+1}.
\]
Thus the displayed map induced by trace is surjective.  Both source and target are
one-dimensional $k$-vector spaces, and the map is $k$-linear because
the residue extension is trivial.  It is therefore an isomorphism.

If $T$ is odd, apply the isomorphism with $r=0$.  The class of $1$ is
nonzero and maps to the class of $\Tr(1)=2$, proving
$v_F(2)=d_\tau$.  Since $w-\widetilde{\bar w}\in\p_K$, the trace-ideal formula \eqref{eq:trace-ideal}
at $j=1$ gives
\[
 \Tr(w)-2\widetilde{\bar w}\in\Tr(\p_K)
 =\p_F^{(T+1)/2}=\p_F^{d_\tau+1},
\]
which proves (a).  In particular, a wildly ramified quadratic
extension with $T$ odd is necessarily of mixed characteristic; in
equal characteristic two the equality $v_F(2)=d_\tau$ is impossible.
If $m=2d+1$ and $r=m-T$, then $r+T=m$ and $q=d$, proving (b).
\end{proof}

The lemma also follows from \eqref{eq:trace-ideal} and
Proposition~\ref{prop:critical-norm-polynomial} by specializing to $p=2$;
the proof above takes the more elementary route through trace ideals
and $\Tr_{K/F}(1)=2$.

\subsection{The four critical functions}

Put $\psi_0(z)=(-1)^{\Tr_{k/\F_2}(z)}$ and fix a normalized
refinement $\mathfrak q$ as in Proposition~\ref{lem:normalized-refinement}.
Use the family $h_\lambda$ and phases $H(\lambda)$ defined there.

\begin{proposition}[Wild-quadratic critical functions]
\label{prop:quadratic-critical-functions}
Retain the minimal-conductor choice of $\chi_F$.
For
\[
 \theta\in\{\tau,\chi_F,\tau\chi_F,\chi_K\},
 \qquad
 E_\theta=
 \begin{cases}
  F,&\theta\ne\chi_K,\\
  K,&\theta=\chi_K,
 \end{cases}
 \qquad
 \psi_\theta=
 \begin{cases}
  \psi_F,&\theta\ne\chi_K,\\
  \psi_K,&\theta=\chi_K,
 \end{cases}
\]
if $m_{E_\theta}(\theta)=2d_\theta+1$ is odd, choose an admissible denominator $\gamma_\theta$ and a stationary
representative $\beta_\theta$ from the classes given by
\eqref{eq:low-mu} and Proposition~\ref{prop:low-parameter-table} when
$m\le T$, and by \eqref{eq:wild-two-tau-linearization} and
Proposition~\ref{prop:wild-quadratic-high-parameters} when $m>T$.
Choose a critical element
$\delta_\theta$ with
$\delta_\theta\cO_{E_\theta}=\p_{E_\theta}^{d_\theta}$.
The choices can be made so that, for every $z\in k$ and every lift
$\widetilde z\in\cO_F\subset\cO_K$,
\[
 \psi_\theta\!\left(
   \frac{\beta_\theta\delta_\theta^2\widetilde z}
        {\gamma_\theta}\right)=\psi_0(z).
\]
The critical function
$ H_\theta(z)=
 \psi_\theta\!\left(
   \frac{\delta_\theta\beta_\theta\widetilde z}
        {\gamma_\theta}\right)
 \theta(1+\delta_\theta\widetilde z)^{-1}$
then satisfies 
\[
 H_\theta(r+s)=H_\theta(r)H_\theta(s)\psi_0(rs)
 \qquad(r,s\in k).
\]
For even conductor put $H_\theta=1$, in accordance with
\eqref{eq:Lamprecht-linearization}.

For each $\theta\in\{\tau,\chi_F,\tau\chi_F,\chi_K\}$ of odd conductor,
let $\lambda_\theta\in k$ be the unique coefficient such that
\[
 H_\theta=h_{\lambda_\theta},
\]
with $h_\lambda$ as in \eqref{eq:q-linear-family}.
At $m=T$, let $\rho\in k^\times\setminus\{1\}$ be the residue of
$u$ from \eqref{eq:quadratic-parameters}, and let $\sqrt\rho$ denote its
unique square root in $k$.
These coefficients satisfy
\begin{equation}\label{eq:quadratic-gamma-table}
\begin{aligned}
 \lambda_{\tau\chi_F}&=
 \begin{cases}
  \lambda_{\tau},&m<T,\ T\text{ odd},\\[1mm]
  \dfrac{\lambda_{\tau}+\rho\lambda_{\chi_F}+\sqrt\rho}{1+\rho},
     &m=T\text{ odd},\\[3mm]
  \lambda_{\chi_F},&m>T,\ m\text{ odd},
 \end{cases}\\[3mm]
 \lambda_{\chi_K}^2&=
 \begin{cases}
  \lambda_{\chi_F},&m<T,\ m\text{ odd},\\[1mm]
  \dfrac{\lambda_{\chi_F}+\rho\lambda_{\tau}+\rho+\sqrt\rho}{1+\rho},
     &m=T\text{ odd},\\[3mm]
  1+\lambda_{\tau},&m>T,\ T\text{ odd}.
 \end{cases}
\end{aligned}
\end{equation}
\end{proposition}

\begin{proof}
Write $\Nm=\Nm_{K/F}$.  Equations~\eqref{eq:minimal-orbit-conductors},
\eqref{eq:minimal-critical}, and \eqref{eq:pullback-high} give
\begin{equation}\label{eq:quadratic-conductor-table}
\begin{array}{c|cccc}
 &m_F(\tau)&m_F(\chi_F)&m_F(\tau\chi_F)&m_K(\chi_K)\\ \hline
 m\le T&T&m&T&m\\
 m>T&T&m&m&2m-T.
\end{array}
\end{equation}
All four conductors are even exactly when $m$ and $T$ are even.

Equation~\eqref{eq:stationary-class-adjoint},
Proposition~\ref{prop:low-parameter-table}, and
Proposition~\ref{prop:wild-quadratic-high-parameters} identify the
stationary classes in the two ranges.  To realize the choices asserted
in the statement, take representatives satisfying the following
equalities, with the parameters \eqref{eq:quadratic-parameters}:
\begin{equation}\label{eq:quadratic-unified-stationary-ratios}
 \frac{\beta(\chi_F)}{\gamma_F}=A\mathcal N,
 \qquad
 \frac{\beta_K}{\gamma_K}=A(\mathcal N-u),
 \qquad \mathcal N=\Nm(u).
\end{equation}
Use the sum representative, relative to a common denominator, for
$\tau\chi_F$.

For an odd-conductor character, form
\eqref{eq:Lamprecht-Hasse-function} with the representative just fixed.
Equations~\eqref{eq:Hasse-functional-equation} and
\eqref{eq:Lamprecht-residual-character} give a nondegenerate polar
character $\psi_0(cxy)$ with $c\in k^\times$.  Multiplying the
critical element by a unit with residue $c^{-1/2}$ gives the
normalization stated in the proposition, since Frobenius is bijective
on $k$.  For even conductor use $H_\theta=1$, as given by
\eqref{eq:Lamprecht-linearization}.
Write $H_\tau=h_{\lambda_\tau}$ when $T$ is odd and
$H_{\chi_F}=h_{\lambda_{\chi_F}}$ when $m$ is odd.
At $m=T$, $u$ is a unit and $\overline{\mathcal N}=\rho^2$.
Equation~\eqref{eq:minimal-orbit-no-leading-cancellation}, with $j=1$,
gives $1+\rho^2\ne0$, hence $1+\rho\ne0$.
When $m=T$ is odd, the stationary ratios $A$ and $A\mathcal N$
in the normalized $\delta_\tau$ coordinate give polar characters $\psi_0(xy)$ and
$\psi_0(\rho^2xy)$.  Thus, for a lift
$\widetilde\rho\in\cO_F^\times$, take
$\delta_\chi=\delta_\tau/\widetilde\rho$; the $\chi_F$-function in the
$\delta_\tau$ coordinate is then $h_{\lambda_{\chi_F}}(\rho z)$.
The product has polar character $\psi_0((1+\rho)^2xy)$, so its
normalized coordinate is obtained by replacing $z$ by $z/(1+\rho)$.

Suppose $m<T$ and $T$ is odd.  Put $d_\tau=(T-1)/2$ and retain the normalized critical element
$\delta_\tau$, of valuation $d_\tau$.  Since
$m<T$,
\[
 d_\tau\ge\left\lceil\frac m2\right\rceil,
\]
so the linearization of $\chi_F$ in \eqref{eq:Lamprecht-linearization}
applies to $x=\delta_\tau\widetilde z$.  Choose the sum representative for the
stationary class of $\tau\chi_F$.  Then, for $z\in k$,
\[
 \frac{H_{\tau\chi_F}(z)}{H_\tau(z)}
 =\psi_F(A\mathcal N \delta_\tau\widetilde z)
   \chi_F(1+\delta_\tau\widetilde z)^{-1}=1.
\]
Thus the critical functions of $\tau\chi_F$ and $\tau$ agree, and
$\lambda_{\tau\chi_F}=\lambda_{\tau}$.

Suppose $m>T$ and $m$ is odd.  Put $d_\chi=(m-1)/2$ and retain
the normalized critical element $\delta_\chi$, of valuation $d_\chi$.
Since $m>T$,
$d_\chi\ge\left\lceil\frac T2\right\rceil$,
so the linearization of $\tau$ in \eqref{eq:Lamprecht-linearization}
applies to $x=\delta_\chi\widetilde z$.  With the sum representative for
$\tau\chi_F$ one obtains
\[
 \frac{H_{\tau\chi_F}(z)}{H_{\chi_F}(z)}
 =\psi_F(A\delta_\chi\widetilde z)
   \tau(1+\delta_\chi\widetilde z)^{-1}=1.
\]
Hence the critical functions of $\tau\chi_F$ and $\chi_F$ agree, and
$\lambda_{\tau\chi_F}=\lambda_{\chi_F}$.

At the boundary $m=T$, the conductors of $\tau$, $\chi_F$, and
$\tau\chi_F$ are all $T$.  When $T$ is odd, the common-coordinate
convention $\delta_\chi=\delta_\tau/\widetilde\rho$ gives the critical functions
of $\tau$ and $\chi_F$
\begin{equation}\label{eq:quadratic-boundary-character-functions}
 \mathfrak q(z)\psi_0(\lambda_{\tau}z),\qquad
 \mathfrak q(\rho z)\psi_0(\rho\lambda_{\chi_F} z).
\end{equation}

We now assume that $m_K(\chi_K)$ is odd and compute the critical
function of $\chi_K$.  Let $\delta_K$ be a critical element for $\chi_K$.  For $Z\in k$, choose a
lift $\widetilde Z\in\cO_F$, put
\[
 x_Z=\delta_K\widetilde Z,
 \qquad P(y)=\Nm(1+y)-1=\Tr(y)+\Nm(y),
\]
and put $m_K(\chi_K)=2d_K+1$.  Equation
\eqref{eq:stationary-parameter-ambiguity} gives
$P(x_Z)\in\p_F^{\lfloor m/2\rfloor}$.  Also
$v_K(ux_Z)\ge a+d_K\ge\lfloor T/2\rfloor$, so
Theorem~\ref{thm:graded-norm}(b) gives
$P(ux_Z)\in\p_F^{\lfloor T/2\rfloor}$.
Define
\[
 p_{\rm nat}(Z)=\overline{P(x_Z)/\delta_\chi}
 \quad(m\text{ odd}),\qquad
 q_{\rm nat}(Z)=\overline{P(ux_Z)/\delta_\tau}
 \quad(T\text{ odd}).
\]
If $m$ is even, set $p_{\rm nat}=0$; if $T$ is even, set $q_{\rm nat}=0$.

Write $H_K^{\rm nat}(Z)$ for the critical function of $\chi_K$ in
this natural coordinate.  Its definition and
\eqref{eq:quadratic-unified-stationary-ratios} give
\begin{align*}
 H_K^{\rm nat}(Z)
 &=\psi_F\!\left(A\Tr((\mathcal N-u)x_Z)\right)
   \chi_F(1+P(x_Z))^{-1}\\
 &=H_{\chi_F}(p_{\rm nat}(Z))
   \psi_F\!\left(-A\{\Tr(ux_Z)+\mathcal N \Nm(x_Z)\}\right)\\
 &=H_{\chi_F}(p_{\rm nat}(Z))\psi_F(-AP(ux_Z)).
\end{align*}
For the second equality, substitute
\[
 H_{\chi_F}(p_{\rm nat}(Z))
 =\psi_F(A\mathcal N P(x_Z))\chi_F(1+P(x_Z))^{-1}
\]
from \eqref{eq:Lamprecht-Hasse-function} when $m$ is odd and
\eqref{eq:Lamprecht-linearization} when $m$ is even, and cancel
$\mathcal N\Tr(x_Z)$.  The last equality uses
$\mathcal N \Nm(x_Z)=\Nm(ux_Z)$.  Since
$1+P(ux_Z)=\Nm(1+ux_Z)$ and $\tau$ is trivial on norms,
\eqref{eq:Lamprecht-Hasse-function} and \eqref{eq:Lamprecht-linearization} similarly give
$H_\tau(q_{\rm nat}(Z))=\psi_F(AP(ux_Z))$.
For odd $T$, \eqref{eq:q-linear-family} and
\eqref{eq:q-phase-translation} give
$H_\tau(q_{\rm nat}(Z))^{-1}
 =H_\tau(q_{\rm nat}(Z))\psi_0(q_{\rm nat}(Z))$.
For even $T$, this also holds because $q_{\rm nat}(Z)=0$ and
$H_\tau\equiv1$.  Therefore
\begin{equation}\label{eq:quadratic-exact-transport-natural}
 H_K^{\rm nat}(Z)
 =H_{\chi_F}(p_{\rm nat}(Z))
  H_\tau(q_{\rm nat}(Z))
  \psi_0(q_{\rm nat}(Z)).
\end{equation}

The trace-ideal formula \eqref{eq:trace-ideal} and total ramification give
\begin{equation*}
 v_F(\Tr(\delta_K\widetilde Z))
 \ge\left\lfloor\frac{d_K+T}{2}\right\rfloor,
 \qquad v_F(\Nm(\delta_K\widetilde Z))\ge d_K.
\end{equation*}
The norm bound is an equality for $Z\ne0$.
When $m_F(\chi_F)=2d_\chi+1$, $d_K=d_\chi$, and
$\lfloor(d_K+T)/2\rfloor\ge d_\chi+1$, reduction of
$P(\delta_K\widetilde Z)=\Tr(\delta_K\widetilde Z)+\Nm(\delta_K\widetilde Z)$
modulo $\p_F^{d_\chi+1}$ gives
\begin{equation*}
 p_{\rm nat}(Z)=c_p Z^2,
 \qquad c_p=\overline{\frac{\Nm(\delta_K)}{\delta_\chi}}\in k^\times.
\end{equation*}
When $T=2d_\tau+1$, $a+d_K=d_\tau$, and
$\lfloor(a+d_K+T)/2\rfloor\ge d_\tau+1$, the same reduction of
$P(u\delta_K\widetilde Z)$, with $v_K(u)=a$ and $\Nm(u)=\mathcal N$, gives
\begin{equation*}
 q_{\rm nat}(Z)=c_q Z^2,
 \qquad c_q=\overline{\frac{\mathcal N \Nm(\delta_K)}{\delta_\tau}}\in k^\times.
\end{equation*}
Both displayed scalars are nonzero when defined.  Replacing $\delta_K$ by
$v_0\delta_K$, with $v_0\in\cO_K^\times$, multiplies both scalars by
$\bar v_0^{\,2}$.  Since Frobenius is an automorphism of $k$, we may
therefore normalize either nonzero scalar to $1$ without changing any
valuation or the ratio of the two scalars.

Put $a=T-m=v_K(u)$.  Suppose first that $m<T$.  Then $m_K(\chi_K)=m$, and
when $m$ is odd one has $d_K=d_\chi=(m-1)/2$.  The trace--norm bounds and the formula for $p_{\rm nat}$ above show that
$p_{\rm nat}(Z)=c_p Z^2$.  Scale $\delta_K$ so that
$c_p=1$.  If $T$ is odd, put $d_\tau=(T-1)/2$.  Since the
odd conductors $m<T$ differ by at least two,
$v_K(u\delta_K)=d_K+a=d_\tau+\frac a2\ge d_\tau+1$.
The norm and trace terms of $P(u\delta_K\widetilde Z)$ then lie in
$\p_F^{d_\tau+1}$, so $q_{\rm nat}(Z)=0$; if $T$ is even there is no
$\tau$-critical factor.

Suppose next that $m>T$ and that $m_K(\chi_K)$ is odd.  Then
$T$ is odd.  With $d_\tau=(T-1)/2$ and
$m_K(\chi_K)=2m-T$, one has
\[
 d_K=m-d_\tau-1,
 \qquad
 v_K(u\delta_K)=a+d_K=d_\tau.
\]
The trace of $u\delta_K\widetilde Z$ lies in
$\p_F^{d_\tau+1}$, while its norm has valuation at least $d_\tau$,
with equality for $Z\ne0$; hence
the formula for $q_{\rm nat}$ above applies.  Scale $\delta_K$ so that
$c_q=1$.  If $m$ is odd and $d_\chi=(m-1)/2$, then
$d_K-d_\chi=\frac{m-T}{2}\ge1$,
so $P(\delta_K\widetilde Z)\in\p_F^{d_\chi+1}$; if $m$
is even there is no such critical factor.

Finally suppose $m=T$.  Then $a=0$ and $m_K(\chi_K)=m=T$.  When $T$ is odd,
$d_K=d_\chi=d_\tau$.  Both initial terms are norm terms.  The formulas for $p_{\rm nat}$ and $q_{\rm nat}$ above give
\[
 p_{\rm nat}(Z)=c_p Z^2,
 \qquad
 q_{\rm nat}(Z)=c_q Z^2.
\]
The critical coordinates for $\chi_F$ and $\tau$ were fixed by
$\delta_\chi=\delta_\tau/\widetilde\rho$, and $\overline{\mathcal N}=\rho^2$.  Therefore
\begin{equation*}
 \frac{c_q}{c_p}
 =\overline{\mathcal N}\,\overline{\frac{\delta_\chi}{\delta_\tau}}
 =\rho^2\rho^{-1}=\rho.
\end{equation*}
Scale $\delta_K$ so that $c_p=1$.  This single scaling preserves
this ratio, and hence simultaneously
normalizes the two maps to
$p_{\rm nat}(Z)=Z^2$ and $q_{\rm nat}(Z)=\rho Z^2$.
Thus
\[
 (p_{\rm nat}(Z),q_{\rm nat}(Z))=
 \begin{cases}
  (Z^2,0),&m<T,\\
  (Z^2,\rho Z^2),&m=T,\\
  (0,Z^2),&m>T.
 \end{cases}
\]
Define
\[
 H_{\chi_K}(Z)=
 \begin{cases}
 H_K^{\rm nat}(Z),&m\ne T,\\[1mm]
 H_K^{\rm nat}(Z/\sqrt{1+\rho}),&m=T.
 \end{cases}
\]
In the boundary case, the change-of-critical-element formula following
\eqref{eq:Lamprecht-Hasse-function} realizes this substitution by
rescaling $\delta_K$ by a unit with residue $(1+\rho)^{-1/2}$.
Thus $H_{\chi_K}$ remains a critical function, and
$\sum_{Z\in k}H_{\chi_K}(Z)=\sum_{Z\in k}H_K^{\rm nat}(Z)$.

For $m\ne T$, substitution in \eqref{eq:quadratic-exact-transport-natural}
gives
\[
 H_{\chi_K}(Z)=
 \begin{cases}
 \mathfrak q(Z^2)\psi_0(\lambda_{\chi_F}Z^2),&m<T,\\
 \mathfrak q(Z^2)\psi_0((1+\lambda_\tau)Z^2),&m>T.
 \end{cases}
\]
At $m=T$, put $z=Z^2/(1+\rho)$. The same identity gives
\[
\begin{aligned}
 H_{\chi_K}(Z)
 &=\mathfrak q(z)\mathfrak q(\rho z)
     \psi_0((\lambda_{\chi_F}+\rho\lambda_\tau+\rho)z)\\
 &=\mathfrak q(Z^2)
     \psi_0\bigl((\lambda_{\chi_F}+\rho\lambda_\tau+\rho)z+\rho z^2\bigr)\\
 &=\mathfrak q(Z^2)\psi_0\!\left(
   \frac{\lambda_{\chi_F}+\rho\lambda_\tau+\rho+\sqrt\rho}{1+\rho}Z^2\right).
\end{aligned}
\]
The second equality uses \eqref{eq:q-refinement}; the third uses
$\psi_0(\rho z^2)=\psi_0(\sqrt\rho\,z)$.
Since $\mathfrak q(Z^2)=\mathfrak q(Z)$ and
$\psi_0(LZ^2)=\psi_0(\sqrt L\,Z)$ for every $L\in k$,
these expressions have polar character $\psi_0(XY)$ and give all
three formulas for $\lambda_{\chi_K}^2$ in
\eqref{eq:quadratic-gamma-table}.

When $m=T$ is odd, multiplying the two functions in
\eqref{eq:quadratic-boundary-character-functions} gives
\[
 \mathfrak q((1+\rho)z)
 \psi_0\bigl((\lambda_\tau+\rho\lambda_{\chi_F})z+\rho z^2\bigr).
\]
Using $\psi_0(\rho z^2)=\psi_0(\sqrt\rho\,z)$ and replacing $z$ by
$z/(1+\rho)$ gives the boundary formula for $\lambda_{\tau\chi_F}$.
The nonboundary formulas were proved above, completing
\eqref{eq:quadratic-gamma-table}.
\end{proof}

\begin{proposition}[Wild-quadratic residual quotient]
\label{prop:quadratic-residual-quotient}
Retain the setup and notation of
Proposition~\ref{prop:quadratic-critical-functions}.  Then
\begin{equation}\label{eq:quadratic-raw-Hasse-quotient}
\begin{aligned}
 \mathcal R_{\mathrm{even}}
 &:=\frac{\Delta_{3,K}(\chi_K)\Delta_{3,F}(\tau)}
      {\Delta_{3,F}(\chi_F)\Delta_{3,F}(\tau\chi_F)}\\
 &=\begin{cases}
  1,&m<T\text{ or }m\text{ even},\\[1mm]
  h_{\lambda_{\chi_F}}(1)^{-1},&m>T,\ m\text{ odd},\\[1mm]
  h_{\lambda_{\tau\chi_F}+1}\!\left(\dfrac{\rho}{1+\rho}\right),
     &m=T \text{ odd}.
 \end{cases}
\end{aligned}
\end{equation}
\end{proposition}

\begin{proof}
For odd conductor use the phase $H(\lambda)$ in
\eqref{eq:quadratic-normalized-sum} and the identities
\eqref{eq:q-phase-translation}; for even conductor the factor is
$1$ by \eqref{eq:stationary-factor-definitions}.

If $m<T$, the conductors in \eqref{eq:quadratic-conductor-table} and
the coefficients in \eqref{eq:quadratic-gamma-table} give
$\Delta_{3,K}(\chi_K)=\Delta_{3,F}(\chi_F)$ and
$\Delta_{3,F}(\tau\chi_F)=\Delta_{3,F}(\tau)$.
Thus $\mathcal R_{\mathrm{even}}=1$.

If $m>T$, the product
$\Delta_{3,K}(\chi_K)\Delta_{3,F}(\tau)$ is $1$: for even $T$ both
factors are $1$, and for odd $T$ it is
$H(1+\lambda_{\tau})H(\lambda_{\tau})=1$ by
\eqref{eq:quadratic-gamma-table}.
The denominator is $1$ for even $m$ and is
$H(\lambda_{\chi_F})^2=h_{\lambda_{\chi_F}}(1)$ for odd $m$.
This proves the asserted formulas when $m>T$.

If $m=T$ and $T$ is even, all four conductors are even, so again
$\mathcal R_{\mathrm{even}}=1$.

It remains to assume $m=T$ and $T$ odd.
All four critical factors are present.  Put
\[
 L=\frac{\lambda_{\chi_F}+\rho\lambda_{\tau}+\rho+\sqrt\rho}{1+\rho},
 \qquad A_\rho=\frac{\rho}{1+\rho}.
\]
Since $\lambda_{\chi_K}^2=L$, equation~\eqref{eq:q-phase-translation}
gives $H(\lambda_{\chi_K})=H(L)$. Hence
\begin{align*}
 \mathcal R_{\mathrm{even}}
 &=\frac{\overline{\mathfrak q(L)}
          \overline{\mathfrak q(\lambda_{\tau})}}
         {\overline{\mathfrak q(\lambda_{\chi_F})}
          \overline{\mathfrak q(\lambda_{\tau\chi_F})}}\\
 &=\mathfrak q(\lambda_{\chi_F})\mathfrak q(\lambda_{\tau\chi_F})
   \overline{\mathfrak q(L)}\overline{\mathfrak q(\lambda_{\tau})}.
\end{align*}
Set $X_0=\lambda_{\chi_F}+\lambda_{\tau\chi_F}$ and $Y_0=L+\lambda_{\tau}$.  In characteristic
two, substitution of the $m=T$ formulas in \eqref{eq:quadratic-gamma-table} gives
\begin{align*}
 X_0+Y_0
 &=\lambda_{\chi_F}+\lambda_{\tau}
   +\frac{\lambda_{\tau}+\rho\lambda_{\chi_F}+\sqrt\rho
          +\lambda_{\chi_F}+\rho\lambda_{\tau}+\rho+\sqrt\rho}{1+\rho}\\
 &=\frac{\rho}{1+\rho}=A_\rho.
\end{align*}
Thus $X_0=Y_0+A_\rho$.  The refinement law therefore gives
\begin{equation*}
 \mathcal R_{\mathrm{even}}=\mathfrak q(A_\rho)\psi_0(E),
 \qquad
 E=Y_0A_\rho+\lambda_{\chi_F}\lambda_{\tau\chi_F}+L\lambda_{\tau}.
\end{equation*}
Let $a_\rho=\sqrt\rho/(1+\sqrt\rho)$, so $a_\rho^2=A_\rho$.  An expansion in
characteristic two gives
\begin{equation*}
 E+\lambda_{\tau\chi_F}A_\rho+A_\rho
 =\bigl(a_\rho(\lambda_{\chi_F}+\lambda_{\tau})+A_\rho\bigr)^2
  +\bigl(a_\rho(\lambda_{\chi_F}+\lambda_{\tau})+A_\rho\bigr).
\end{equation*}
The right side of this identity has
absolute trace zero. Substituting this into
$\mathcal R_{\mathrm{even}}=\mathfrak q(A_\rho)\psi_0(E)$ gives
\[
 \mathcal R_{\mathrm{even}}=\mathfrak q(A_\rho)
      \psi_0(\lambda_{\tau\chi_F}A_\rho+A_\rho)=h_{\lambda_{\tau\chi_F}+1}(A_\rho).
\]

This proves \eqref{eq:quadratic-raw-Hasse-quotient}.
\end{proof}

\subsection{Evaluation of the stationary factors}

\begin{lemma}[Depth of the quadratic correction units]
\label{lem:quadratic-correction-depth}
The correction units in \eqref{eq:quad-corrections} satisfy
\begin{equation}\label{eq:quadratic-correction-depth-bound}
 v_F(z_0-1)=v_F(z_1-1)
 \ge\left\lfloor\frac{T+|a|}{2}\right\rfloor.
\end{equation}
\end{lemma}

\begin{proof}
Put $s=\Tr_{K/F}(u)$.  Since
$v_F(1+\mathcal N)=\min\{a,0\}$, equations~\eqref{eq:quad-zxy} and
\eqref{eq:trace-ideal} give
\begin{align*}
 v_F(z_0-1)=v_F(z_1-1)
 &=v_F(s)-\min\{a,0\}\\
 &\ge\left\lfloor\frac{a+T}{2}\right\rfloor-\min\{a,0\}
  =\left\lfloor\frac{T+|a|}{2}\right\rfloor.\qedhere
\end{align*}
\end{proof}

\begin{proposition}[Character values on the quadratic correction units]
\label{prop:quadratic-correction-values}
Retain the setup and notation of
Proposition~\ref{prop:quadratic-critical-functions}.   One has
\begin{equation}\label{eq:quadratic-correction-values}
\begin{aligned}
 &\bigl(\psi_F(A\mathcal N(z_0-1))\chi_F(z_0)^{-1},\,
        \psi_F(A(z_1-1))\tau(z_1)^{-1}\bigr)\\
 &\qquad=\begin{cases}
 (1,1),&m<T\text{ or }m\text{ even},\\[1mm]
 (h_{\lambda_{\chi_F}}(1),1),&m>T,\ m\text{ odd},\\[1mm]
 \left(h_{\lambda_{\chi_F}}\!\left(\dfrac{\rho^2}{1+\rho^2}\right),\,
       h_{\lambda_{\tau}}\!\left(\dfrac{\rho}{1+\rho^2}\right)\right),
       &m=T \text{ odd}.
 \end{cases}
\end{aligned}
\end{equation}
\end{proposition}

\begin{proof}
Write $\Nm=\Nm_{K/F}$ and $\Tr=\Tr_{K/F}$.
Put $x=z_0-1$ and $y=z_1-1$.
For a quasi-character $\theta$ of $F^\times$ of odd conductor
$2d_\theta+1$, write $H_\theta=h_{\lambda_\theta}$ using its normalized
critical element $\delta_\theta$.
Substitution of the lift $w/\delta_\theta$ in
\eqref{eq:Lamprecht-Hasse-function} gives
\begin{equation}\label{eq:quadratic-residual-correction}
 \theta(1+w)^{-1}
 =\psi_F\!\left(-\frac{\beta_\theta w}{\gamma_\theta}\right)
  h_{\lambda_\theta}\!\left(\overline{w/\delta_\theta}\right)
 \qquad(w\in\p_F^{d_\theta}).
\end{equation}
Since $m=T-a$, the bound \eqref{eq:quadratic-correction-depth-bound}
is at least both $\lfloor m/2\rfloor$ and $\lfloor T/2\rfloor$.
Thus \eqref{eq:quadratic-residual-correction} applies to
$(\theta,w)=(\chi_F,x)$ and $(\tau,y)$ when the respective conductor
is odd.  For even conductor, the bound reaches the linearization
domain in \eqref{eq:quadratic-linearization-domains}.
For odd conductor put
$x'=\overline{x/\delta_\chi}$ or $y'=\overline{y/\delta_\tau}$,
respectively.  The two residual factors are $h_{\lambda_{\chi_F}}(x')$ and
$h_{\lambda_{\tau}}(y')$; for even conductor the respective factor is $1$.

For a fixed stationary representative, normalized critical elements
give the same critical function.  Indeed, if $\delta_\theta'=v\delta_\theta$
with $v\in\cO_F^\times$, their normalization implies
$\psi_0((\bar v^2-1)z)=1$ for every $z\in k$.
Nondegeneracy of the trace pairing gives $\bar v^2=1$, hence $\bar v=1$.
By \eqref{eq:Lamprecht-Hasse-function},
$H_\theta'(z)=H_\theta(\bar v z)=H_\theta(z)$.

We will use the following trace--norm normalization. For $w\in K^\times$,
put $c_w=\Tr(w)^2/\Nm(w)$ and suppose $v_F(c_w)=T-1$.
Then $v_K((\Tr(w)/\Nm(w))w)=T-1>0$, and for $z\in\cO_F$,
\[
 \Nm\!\left(1+\frac{\Tr(w)}{\Nm(w)}wz\right)=1+c_w(z+z^2).
\]
The character $\tau$ is trivial on this norm. Since
$T-1\ge\ceil{T/2}$, its stationary linearization gives
$\psi_F(Ac_w(z+z^2))=1$.
Exactness of the conductor makes
$\bar z\mapsto\psi_F(Ac_wz)$ a nontrivial character of $k$.
It is trivial on the Artin--Schreier image, which is
$\ker\Tr_{k/\F_2}$ by \eqref{eq:scaled-frobenius-trace-image}.
Consequently $\psi_F(Ac_wz)=\psi_0(\bar z)$.

If $m<T$, the bound \eqref{eq:quadratic-correction-depth-bound} is
at least $(m+1)/2$ when $m$ is odd and at least $(T+1)/2$ when $T$
is odd.  Thus $x'=y'=0$ whenever defined, and both residual factors
are $1$.

Suppose $m>T$.  When $T$ is odd,
\eqref{eq:quadratic-correction-depth-bound} gives
$v_F(y)\ge\lfloor m/2\rfloor\ge(T+1)/2$, so $y'=0$.
When $T$ is even, \eqref{eq:quadratic-linearization-domains} applies
to $y$.  Thus the $\tau$ residual factor is always $1$.
For even $m$, the same linearization applies to $x$, proving the
asserted pair $(1,1)$.

Now suppose $m>T$ and $m$ is odd.  Put $c=m-T>0$ and $w=u^{-1}$.  Then $v_K(w)=c$ and
$c+T=m$ is odd.  Lemma~\ref{lem:quadratic-graded-trace}(b) shows that
\[
 \delta_\chi:=-\Tr(w)\in\p_F^{(m-1)/2}\setminus\p_F^{(m+1)/2}
\]
is a valid critical element for $\chi_F$.
Since $\Nm(w)=\mathcal N^{-1}$, one has
$c_w=\mathcal N\delta_\chi^2$ and
$v_F(c_w)=(T-m)+(m-1)=T-1$.
The trace--norm normalization above therefore gives
\[
 \psi_F(A\mathcal N\delta_\chi^2\widetilde z)=\psi_0(z)
 \qquad(z\in k),
\]
so $\delta_\chi=-\Tr(w)$ has the required normalized polar character.
Put $s=\Tr_{K/F}(u)$.  Since
$\frac{s}{\mathcal N}=\Tr(u^{-1})=\Tr(w)$
and $x=-s/(1+\mathcal N)$, one has the identity
\[
 \frac{x}{\delta_\chi}=\frac{\mathcal N}{1+\mathcal N}
 =1-\frac1{1+\mathcal N}.
\]
Here $v_F(\mathcal N)=T-m=-c$, so $v_F((1+\mathcal N)^{-1})=c>0$.  Therefore the
normalized correction class is
$x'=\overline{x/\delta_\chi}=1$.
Equation~\eqref{eq:quadratic-residual-correction} therefore gives
the residual factor $h_{\lambda_{\chi_F}}(1)$, proving the case $m>T$, $m$ odd.

If $m=T$ and $T$ is even, the bound
\eqref{eq:quadratic-correction-depth-bound} gives $x,y\in\p_F^{T/2}$,
so \eqref{eq:quadratic-linearization-domains} gives the pair $(1,1)$.

Finally suppose $m=T$ and $T$ is odd.  Put
$s=\Tr_{K/F}(u)$ and $d_\tau=(T-1)/2$.
Choose a lift $\widetilde\rho\in\cO_F^\times$ of $\rho$.
Lemma~\ref{lem:quadratic-graded-trace}(a) gives $v_F(2)=d_\tau$, so
$\delta_\tau=2,\qquad \delta_\chi=2/\widetilde\rho$
are critical elements for $\tau$ and $\chi_F$, respectively.
For $w=1$, one has $c_w=4$ and $v_F(c_w)=T-1$.
The same trace--norm normalization gives
\begin{equation}\label{eq:quadratic-boundary-polar-normalization}
 \psi_F(A\delta_\tau^2\widetilde z)=\psi_0(z)
 \qquad(z\in k).
\end{equation}
Since $\overline{\mathcal N}=\rho^2$,
\eqref{eq:quadratic-boundary-polar-normalization} gives
\[
 \psi_F(A\mathcal N\delta_\chi^2\widetilde z)
 =\psi_F\!\left(4A\frac{\mathcal N}{\widetilde\rho^{\,2}}
                    \widetilde z\right)=\psi_0(z).
\]
Thus these choices give the normalized functions $h_{\lambda_{\tau}}$ and
$h_{\lambda_{\chi_F}}$ already fixed in Proposition~\ref{prop:quadratic-critical-functions}.

Since $u-\widetilde\rho\in\p_K$, the unit-trace congruence
\eqref{eq:quadratic-unit-trace-class} gives
\[
 s-2\widetilde\rho=\Tr(u-\widetilde\rho)\in\p_F^{d_\tau+1}.
\]
Consequently
$\overline{s/\delta_\tau}=\rho$.
Combining $\overline{s/\delta_\tau}=\rho$,
\eqref{eq:quad-zxy}, and $\overline{\mathcal N}=\rho^2$ gives
\begin{equation}\label{eq:quadratic-boundary-correction-classes}
 x'=\overline{x/\delta_\chi}=\frac{\rho^2}{1+\rho^2},
 \qquad
 y'=\overline{y/\delta_\tau}=\frac{\rho}{1+\rho^2}.
\end{equation}
Substituting \eqref{eq:quadratic-boundary-correction-classes} in
\eqref{eq:quadratic-residual-correction} proves the remaining case of
\eqref{eq:quadratic-correction-values}.
\end{proof}

\begin{proposition}[Quadratic stationary assembly]
\label{prop:quadratic-complete-assembly}
Retain the minimal-conductor choice of $\chi_F$.
With $X$ from Lemma~\ref{lem:quad-X} and
$\mathcal R_{\mathrm{even}}$ from
\eqref{eq:quadratic-raw-Hasse-quotient}, the error quotient satisfies
\begin{equation}\label{eq:quadratic-assembly}
\begin{aligned}
 \Err_{K/F}(\chi_F)
 &=\mathcal R_{\mathrm{even}}\,\psi_F(-AX)\\[-1mm]
 &\quad\times\begin{cases}
  1,&m<T\text{ or }m\text{ even},\\[1mm]
  h_{\lambda_{\chi_F}}(1),&m>T,\ m\text{ odd},\\[1mm]
  h_{\lambda_{\chi_F}}\!\left(\dfrac{\rho^2}{1+\rho^2}\right)
  h_{\lambda_{\tau}}\!\left(\dfrac{\rho}{1+\rho^2}\right),
    &m=T \text{ odd}.
 \end{cases}
\end{aligned}
\end{equation}
\end{proposition}

\begin{proof}
Write $\Nm=\Nm_{K/F}$.  Use \eqref{eq:Delta-factorization-new},
\eqref{eq:trivial-local-constant}, and the following admissible denominators and stationary representatives.
Rescaling a denominator and its numerator uses
Lemma~\ref{lem:stationary-class-calculus}(b).
For $m>T$, the pairs and their sources are
\[
\begin{array}{c|c|c}
 \text{character}&(\gamma_\theta,\beta_\theta)&\text{source}\\ \hline
 \chi_K&(\gamma_F,\ \alpha \kappa(\mathcal N-u))
  &\eqref{eq:Pstar-wild-two}\\
 \tau&(\gamma_F/(\alpha \kappa),\ 1)
  &\eqref{eq:wild-two-tau-linearization}\\
 \chi_F&(\gamma_F,\ \alpha \kappa\mathcal N)
  &\eqref{eq:Lamprecht-linearization},\ \eqref{eq:quadratic-parameters}\\
 \tau\chi_F&(\gamma_F,\ \alpha \kappa(\mathcal N+1))
  &\eqref{eq:Pstar-wild-two}.
\end{array}
\]
For $m\le T$, they are
\[
\begin{array}{c|c|c}
 \text{character}&(\gamma_\theta,\beta_\theta)&\text{source}\\ \hline
 \chi_K&(\Gamma_\tau/\widetilde{\xi},\ (\alpha/\widetilde{\xi})(\mathcal N-u))
  &\text{Proposition~\ref{prop:low-parameter-table}}\\
 \tau&(\Gamma_\tau/\alpha,\ 1)&\eqref{eq:low-mu}\\
 \chi_F&(\Gamma_\tau/\xi,\ (\alpha/\xi)\mathcal N)
  &\eqref{eq:gammas-low},\ \eqref{eq:low-mu}\\
 \tau\chi_F&(\Gamma_\tau,\ \alpha(\mathcal N+1))
  &\eqref{eq:low-beta-twists}\ (j=1).
\end{array}
\]
The factors contributed by the admissible elements are, in the high
range,
\[
 \frac{\chi_K(\gamma_F)\tau(\gamma_F/(\alpha \kappa))}
      {\chi_F(\gamma_F)(\tau\chi_F)(\gamma_F)}
 =\tau(\alpha \kappa)^{-1},
\]
and in the low range,
\[
 \frac{\chi_K(\Gamma_\tau/\widetilde{\xi})\tau(\Gamma_\tau/\alpha)}
      {\chi_F(\Gamma_\tau/\xi)(\tau\chi_F)(\Gamma_\tau)}
 =\tau(\alpha)^{-1}=1.
\]
Here $\chi_K=\chi_F\circ \Nm$, $\Nm(\widetilde{\xi})=\xi$, and $\alpha$ is a norm.

It remains to collect the factors $\theta(\beta)^{-1}$ from
\eqref{eq:stationary-factor-definitions}.  In the high range they give
\[
\begin{aligned}
 &\chi_K\!\left(\alpha \kappa(\mathcal N-u)\right)^{-1}
  \chi_F(\alpha \kappa\mathcal N)
  (\tau\chi_F)\!\left(\alpha \kappa(\mathcal N+1)\right)\\
 &\qquad=\tau(\alpha \kappa)\,
  \chi_F\!\left(\frac{\mathcal N(\mathcal N+1)}{\Nm(\mathcal N-u)}\right)
  \tau(\mathcal N+1),
\end{aligned}
\]
Indeed, $\Nm(\alpha \kappa)=(\alpha \kappa)^2$, so the scalar factors
in the two $\chi_F$-values cancel that in the $\chi_K$-value.
The low-range contribution is
\[
 \begin{aligned}
 &\chi_K\!\left(\frac{\alpha}{\widetilde{\xi}}(\mathcal N-u)\right)^{-1}
  \chi_F\!\left(\frac{\alpha}{\xi}\mathcal N\right)
  (\tau\chi_F)\!\left(\alpha(\mathcal N+1)\right)\\
 &\qquad=\chi_F\!\left(
   \frac{\mathcal N(\mathcal N+1)}{\Nm(\mathcal N-u)}\right)
   \tau(\mathcal N+1).
 \end{aligned}
\]
because $\Nm(\alpha/\widetilde{\xi})=\alpha^2/\xi$ and $\tau(\alpha)=1$.
Multiplication by the admissible-element quotients therefore leaves
the same character-value factor in both ranges; the high-range factor
$\tau(\alpha \kappa)$ cancels its inverse.

By \eqref{eq:quad-corrections},
\[
 \mathcal N-u=-u(1-\mathcal N/u),\qquad
 \Nm(\mathcal N-u)=\mathcal N \Nm(1-\mathcal N/u)
               =\mathcal N(1+\mathcal N)z_0.
\]
Also $\Nm(1+u)=(1+\mathcal N)z_1$, and $\tau$ is trivial on norms.
Consequently the combined character-value factor is
$\chi_F(z_0)^{-1}\tau(z_1)^{-1}$.

Since $\psi_K=\psi_F\circ\Tr_{K/F}$, the additive stationary values
have signed sum
\[
 \Tr_{K/F}(\mathcal N-u)+1-\mathcal N-(\mathcal N+1)
 =-\Tr_{K/F}(u).
\]
Together with this character-value factor and the residual quotient
\eqref{eq:quadratic-raw-Hasse-quotient}, this gives
\[
 \Err_{K/F}(\chi_F)=\mathcal R_{\mathrm{even}}\,
 \left[\psi_F(A\Tr_{K/F}(u))\chi_F(z_0)\tau(z_1)\right]^{-1}.
\]
Substitute \eqref{eq:quadratic-correction-values}.  The additive
arguments combine as
\[
 -A\bigl(\Tr_{K/F}(u)+\mathcal N(z_0-1)+(z_1-1)\bigr)=-AX
\]
by the definition of $X$ in Lemma~\ref{lem:quad-X}.
This proves \eqref{eq:quadratic-assembly}.
\end{proof}

\subsection{The norm--trace term and final cancellation}

\begin{proposition}[Depth and phase of the quadratic norm--trace term]
\label{prop:quadratic-norm-trace-phase}
Retain the minimal-conductor choice of $\chi_F$ and $X$ from
Lemma~\ref{lem:quad-X}.
If $\operatorname{char}F=2$, then $X=0$.
If $\operatorname{char}F=0$, then $X\in\p_F^T$ when $m\ne T$ or $T$ is even,
whereas $v_F(X)=T-1$ when $m=T$ and $T$ is odd.
In all cases,
\begin{equation}\label{eq:quadratic-norm-trace-phase}
 \psi_F(-AX)=1.
\end{equation}
\end{proposition}

\begin{proof}
Equation~\eqref{eq:quadratic-linearization-domains} gives
$\psi_F(A\p_F^T)=1$, since $\tau$ has conductor $T$.
If $\operatorname{char}F=2$, equation~\eqref{eq:quad-zxy} gives $X=0$,
proving the assertions in this case.

Assume $\operatorname{char}F=0$ and put $e=v_F(2)$.
Since $2=\Tr_{K/F}(1)$, equation~\eqref{eq:trace-ideal} gives
$e\ge\lfloor T/2\rfloor$. By \eqref{eq:quad-zxy}, $X=2(z_1-1)$,
so Lemma~\ref{lem:quadratic-correction-depth} gives
\[
 v_F(X)\ge e+\floor{\frac{T+|a|}{2}}
 \ge\floor{\frac T2}+\floor{\frac{T+|a|}{2}}.
\]
If $a\ne0$ or $T$ is even, the last expression is at least $T$.
Thus $X\in\p_F^T$ and $\psi_F(-AX)=1$ in these cases.

It remains to take $m=T$ and $T$ odd.  Put
$C=\rho/(1+\rho^2)$ and $y=z_1-1$.
The critical element in \eqref{eq:quadratic-boundary-polar-normalization}
is $\delta_\tau=2$, and
\eqref{eq:quadratic-boundary-correction-classes} gives
$\overline{y/2}=C$.
By \eqref{eq:quad-zxy}, $X=2y$; since $v_F(4)=T-1$, it follows that
$X\equiv4\widetilde C\pmod{\p_F^T}$
for any lift $\widetilde C\in\cO_F$ of $C$.
Since $C\ne0$ and $v_F(4)=T-1$, this gives $v_F(X)=T-1$.
Using $\psi_F(A\p_F^T)=1$ and
\eqref{eq:quadratic-boundary-polar-normalization}, we obtain
$\psi_F(-AX)=\psi_0(-C)=\psi_0(C)$,
because $k$ has characteristic two.  In $k$, put $c=(1+\rho)^{-1}$.
Then
$c+c^2=\frac{\rho}{(1+\rho)^2}=C$.
Equation~\eqref{eq:AS-phase} gives $\psi_0(C)=1$, proving
\eqref{eq:quadratic-norm-trace-phase} also at the boundary.
\end{proof}

\begin{proposition}[Wild quadratic cancellation]
\label{prop:wild-quadratic-cancellation}
For every wildly ramified quadratic extension $K/F$ of nonarchimedean
local fields and every quasi-character $\chi_F$ of $F^\times$,
\[
 \Err_{K/F}(\chi_F)=1.
\]
\end{proposition}

\begin{proof}
By \eqref{eq:error-twist-invariance}, choose $\chi_F$ of minimal
conductor in its $S(K/F)$-orbit.  Conductors $0$ and $1$ were treated
in Section~\ref{sec:wild-endpoints}, so assume $m_F(\chi_F)>1$.
Equation~\eqref{eq:quadratic-norm-trace-phase} gives $\psi_F(-AX)=1$.
It remains to cancel the residual quotient with the additional
critical-function factors in \eqref{eq:quadratic-assembly}.

If $m<T$ or $m$ is even, both the residual quotient in
\eqref{eq:quadratic-raw-Hasse-quotient} and the additional factor in
\eqref{eq:quadratic-assembly} are $1$.
If $m>T$ and $m$ is odd, those factors are $h_{\lambda_{\chi_F}}(1)^{-1}$ and
$h_{\lambda_{\chi_F}}(1)$, respectively, and cancel.  This proves the result when
$m\ne T$ or $T$ is even.

Suppose $m=T$ and $T$ is odd.  Put
\[
 A_\rho=\frac{\rho}{1+\rho},\qquad
 B=\frac{\rho^2}{1+\rho^2},\qquad C=\frac{\rho}{1+\rho^2}.
\]
Equations~\eqref{eq:quadratic-raw-Hasse-quotient} and
\eqref{eq:quadratic-assembly}, with $\psi_F(-AX)=1$, give
\[
 \Err_{K/F}(\chi_F)
 =h_{\lambda_{\tau\chi_F}+1}(A_\rho)h_{\lambda_{\chi_F}}(B)h_{\lambda_{\tau}}(C).
\]
Since $B+C=A_\rho$, the refinement identities
\eqref{eq:q-refinement} and \eqref{eq:q-phase-translation} imply
\[
 \mathfrak q(A_\rho)\mathfrak q(B)\mathfrak q(C)
 =\mathfrak q(A_\rho)^2\psi_0(BC)=\psi_0(A_\rho+BC).
\]
Expanding $h_\lambda(z)=\mathfrak q(z)\psi_0(\lambda z)$ therefore yields
\[
 \Err_{K/F}(\chi_F)
 =\psi_0(\lambda_{\tau\chi_F}A_\rho+\lambda_{\chi_F} B+\lambda_{\tau}C+BC).
\]
The boundary relation in \eqref{eq:quadratic-gamma-table} gives
\[
 \lambda_{\tau\chi_F}A_\rho+\lambda_{\chi_F} B+\lambda_{\tau}C
 =\frac{\rho(\lambda_{\tau}+\rho\lambda_{\chi_F}+\sqrt\rho)
              +\lambda_{\chi_F}\rho^2+\lambda_{\tau}\rho}{(1+\rho)^2}
 =\frac{\rho\sqrt\rho}{1+\rho^2}.
\]
Put $z=\rho\sqrt\rho/(1+\rho^2)$.  Then
$BC=\frac{\rho^3}{(1+\rho^2)^2}=z^2$.
Thus $\Err_{K/F}(\chi_F)=\psi_0(z+z^2)=1$ by \eqref{eq:AS-phase}.
\end{proof}

\section*{Acknowledgments}
I am grateful to Chung Pang Mok for discussions on a local proof of the local
Langlands correspondence, which provided the original motivation for the First
and Second Main Lemma projects, and for his continued interest in this work.

OpenAI ChatGPT assisted with drafting and wording revisions of this manuscript;
the author audited the manuscript, checked the proof, and assumes full responsibility.


\begin{thebibliography}{99}
\bibitem{DavenportHasse}
H.~Davenport and H.~Hasse,
\emph{Die Nullstellen der Kongruenzzetafunktionen in gewissen
zyklischen F\"allen},
J. Reine Angew. Math. \textbf{172} (1935), 151--182.

\bibitem{Dwork}
B.~M.~Dwork,
\emph{On the Root Number in the Functional Equation of the Artin--Weil
$L$-Series}, Ph.D. thesis, Columbia University, 1954.

\bibitem{Langlands}
R.~P.~Langlands,
\emph{On the Functional Equation of the Artin $L$-functions},
Yale University preprint (1970), IAS electronic edition.

\bibitem{Serre}
J.-P.~Serre,
\emph{Local Fields},
Graduate Texts in Mathematics \textbf{67}, Springer, 1979.

\bibitem{UedaSML}
F.~Ueda,
\emph{A Local Proof of Langlands's Second Main Lemma for Local Epsilon
Factors over Nonarchimedean Local Fields},
2026.
\href{https://arxiv.org/abs/2609.13960}{arXiv:2609.13960}.
\end{thebibliography}
\end{document}